%% file: arxiv.tex
\documentclass[12pt]{article}
\usepackage{amsmath,amssymb,amsfonts,amsthm}
\usepackage{mathrsfs}

\usepackage{mathtools, nccmath, relsize}

\usepackage{indentfirst}
\usepackage[pdftex]{color,graphicx}

\usepackage{xcolor}
\definecolor{lightblue}{rgb}{0,0.2,0.5}
\usepackage[colorlinks=true, urlcolor=lightblue,linkcolor=lightblue, citecolor=lightblue]{hyperref}
\usepackage{ucs}

\definecolor{ForestGreen}{RGB}{34,139,34}
\definecolor{mauve}{rgb}{0.7,0,0.43}
\definecolor{dkgreen}{rgb}{0,0.6,0}
\definecolor{darkgreen}{rgb}{0,0.6,0}
\definecolor{darkorange}{rgb}{1.0, 0.55, 0.0}
\definecolor{lightblue}{rgb}{0,0.2,0.5}
\definecolor{blue1}{rgb}{0,0.1,0.9}
\usepackage{xcolor}

\definecolor{lightblue}{rgb}{0,0.2,0.5}
\usepackage[colorlinks=true, urlcolor=lightblue,linkcolor=lightblue, citecolor=lightblue]{hyperref}
\usepackage{ucs}

\definecolor{dblackcolor}{rgb}{0.0,0.0,0.0}
\definecolor{dbluecolor}{rgb}{0.01,0.02,0.7}
\definecolor{dgreencolor}{rgb}{0.2,0.4,0.0}
\definecolor{dgraycolor}{rgb}{0.30,0.3,0.30}

\definecolor{ForestGreen}{RGB}{34,139,34}
\definecolor{mauve}{rgb}{0.7,0,0.43}
\definecolor{dkgreen}{rgb}{0,0.6,0}
\definecolor{darkgreen}{rgb}{0,0.6,0}
\definecolor{darkorange}{rgb}{1.0, 0.55, 0.0}
\definecolor{lightblue}{rgb}{0,0.2,0.5}
\definecolor{blue1}{rgb}{0,0.1,0.9}

\definecolor{lightblue}{rgb}{0,0.2,0.5}

\usepackage{enumerate} %
\usepackage[numbers]{natbib} %
\usepackage{yfonts} %
\usepackage[title]{appendix} %

\usepackage{comment}
\usepackage[margin=0.25in]{geometry}
\usepackage{pgfplots,pgfplotstable}
\pgfplotsset{width=10cm,compat=1.9}
\usepackage[edges]{forest}

\usepackage{empheq} 

\usepackage{cases}

\usepackage{subcaption} 

\usepackage{xr} %

\usepackage{float} 

\newenvironment{Proof}{\removelastskip\par\medskip
\noindent{\em Proof.} \rm}{\penalty-20\null\hfill$\square$\par\medbreak}

\newenvironment{Proofy}{\removelastskip\par\medskip
\noindent{\em Proof} \rm}{\penalty-20\null\hfill$\square$\par\medbreak}

\allowdisplaybreaks

\makeatletter
\let\@fnsymbol\@arabic %
\makeatother

\theoremstyle{plain}
\newtheorem{theorem}{Theorem}[section]
\newtheorem{proposition}[theorem]{Proposition}
\newtheorem{corollary}[theorem]{Corollary}
\newtheorem{lemma}[theorem]{Lemma}
\newtheorem{remark}[theorem]{\underline{Remark}}

\theoremstyle{remark}

\theoremstyle{definition}
\newtheorem{definition}[theorem]{Definition}
\newtheorem{assumption}{Assumption}

\newcommand{\alphastar}{\alpha} 
\newcommand{\jstar}{j}

\newcommand{\R}{\mathbb R}
\newcommand{\N}{\mathbb N}

\newcommand{\C}{\mathscr{C}}

\newcommand{\F}{\mathcal F}

\newcommand{\ind}{\mathbf 1}

\newcommand{\id}{\mathrm{Id}}
\newcommand{\M}{\mathcal M}
\newcommand{\pt}{\partial}

\def\real{{\mathord{\mathbb R}}}
\def\N{{\mathord{\mathbb N}}}

\def\P{\mathbb{P}}
\def\E{\mathbb{E}}

\makeatletter
\newcommand*\rel@kern[1]{\kern#1\dimexpr\macc@kerna}
\newcommand*\widebar[1]{%
  \begingroup
  \def\mathaccent##1##2{%
    \rel@kern{0.8}%
    \overline{\rel@kern{-0.8}\macc@nucleus\rel@kern{0.2}}%
    \rel@kern{-0.2}%
  }%
  \macc@depth\@ne
  \let\math@bgroup\@empty \let\math@egroup\macc@set@skewchar
  \mathsurround\z@ \frozen@everymath{\mathgroup\macc@group\relax}%
  \macc@set@skewchar\relax
  \let\mathaccentV\macc@nested@a
  \macc@nested@a\relax111{#1}%
  \endgroup
}
\makeatother

\makeatletter
\DeclareRobustCommand\widecheck[1]{{\mathpalette\@widecheck{#1}}}
\def\@widecheck#1#2{%
    \setbox\z@\hbox{\m@th$#1#2$}%
    \setbox\tw@\hbox{\m@th$#1%
       \widehat{%
          \vrule\@width\z@\@height\ht\z@
          \vrule\@height\z@\@width\wd\z@}$}%
    \dp\tw@-\ht\z@
    \@tempdima\ht\z@ \advance\@tempdima2\ht\tw@ \divide\@tempdima\thr@@
    \setbox\tw@\hbox{%
       \raise\@tempdima\hbox{\scalebox{1}[-1]{\lower\@tempdima\box
\tw@}}}%
    {\ooalign{\box\tw@ \cr \box\z@}}}
\makeatother

\numberwithin{equation}{section} %

\usepackage{hyphenat}
\usepackage{booktabs}
\usepackage{graphicx}

\begin{document}

\title{
\huge
Stability analysis of a branching diffusion solver for semilinear heat equations %
} 

\author{
 Qiao Huang\footnote{
 School of Mathematics,
  Southeast University,
  Nanjing 211189,
  P.R. China. 
 \href{mailto:qiao.huang@seu.edu.cn}{qiao.huang@seu.edu.cn}}
 \qquad
 Nicolas Privault\footnote{%
 School of Physical and Mathematical Sciences, 
 Nanyang Technological University, 
 21 Nanyang Link, Singapore 637371.
 \href{mailto:nprivault@ntu.edu.sg}{nprivault@ntu.edu.sg}
}
}

\maketitle

\vspace{-0.5cm}

\begin{abstract} 
Stochastic branching algorithms provide a useful alternative to grid-based schemes for the numerical solution of partial differential equations, particularly in high\hyp{}dimensional settings. However, they require a strict control of the integrability of random functionals of branching processes in order to ensure the non-explosion of solutions. In this paper, we study the stability of a functional branching representation of PDE solutions by deriving sufficient criteria for the integrability of the multiplicative weighted progeny of stochastic branching processes. We also prove the uniqueness of mild solutions under uniform integrability assumptions on random functionals. 
\end{abstract}


\baselineskip0.7cm

\section{Introduction}%
\noindent
 Consider the $d$-dimensional semi-linear heat equation
 \begin{equation}
   \label{nl-heat}
  \begin{cases}
  \displaystyle
  \frac{\partial u}{\partial t} (t, x)+
  \frac{1}{2} \Delta
  u(t, x) + f( u(t, x) ) =0, & (t,x)\in[0,T)\times\R^d,
      \medskip
      \\
    u(T, x) = \phi (x), \quad x\in\R^d, & 
  \end{cases}
\end{equation}
 where $f\in {\cal C}^\infty(\R^d)$ and $\phi :\R^d\to\R$ are given functions,
 and 
$$
 \frac{1}{2} \Delta = \frac{1}{2} \sum_{i=1}^d \pt_{x_i}^2
$$
is the generator of the
standard $d$-dimensional Brownian motion $B= ( B_t )_{t\geq 0}$.

\medskip
 
Probabilistic schemes for the numerical solution of
partial differential equations 
provide a promising direction to overcome the curse of
dimensionality. 
 In the linear case with $f(u(t, x)) = r u(t, x)$ and
 $r\in \real$ a constant, the solution of the PDE
$$
\begin{cases}
      \displaystyle
      \frac{\partial u}{\partial t}
      (t,x) + \frac{1}{2}\Delta u(t,x)
        + r u(t,x) = 0,
  \medskip
  \\
u(T,x) = \phi (x), \qquad (t,x) = (t,x_1, \ldots, x_d) \in [0,T] \times \real^d,
\end{cases}
$$
 is well known to admit the probabilistic representation
$$
    u(0, x) = e^{rT} \E [ \phi (x + W_T) ],
$$
where $(W_t)_{t \geq 0}$ is a standard Brownian motion.

\medskip

More generally, probabilistic representations for the solutions of
first and second order nonlinear PDEs %
 can be obtained
 by
 representing $u(t,x)$ as
 $u(t,x) =
 Y_{t,t}^x$, $(t,x)\in [0,T]\times \real$,
 where $(Y_{t,s}^x)_{t \leq s \leq T}$ is the
 solution of a backward stochastic differential equation (BSDE),
 see \cite{peng2}, \cite{pardouxpeng}.
 The BSDE method has been implemented in
 \cite{han2018solving} using a deep learning
 algorithm in the case where $f$ depends on the first order derivative,
 see also \cite{hure2019some}
 for recent improvements. 
 This method also extends to second order fully nonlinear PDEs %
 by the use of second order backward stochastic differential
 equations, see e.g. \cite{touzi}, \cite{soner}.

  \medskip

 Numerical solutions of semilinear PDEs have also
 been obtained by the multilevel Picard method
 \cite{hutzenthaler-mlp0},
 \cite{hutzenthaler-mlp2},
 \cite{hutzenthaler-mlp3},
 \cite{hutzenthaler-mlp1},
 \cite{NeufeldNguyenWu2025MLPGradient},
 \cite{HanHuLongZhao2026},
 \cite{Nguyen2026MLP},
 with numerical experiments provided in \cite{hutzenthaler-mlp4},
 see also 
 \cite{Neufeld2025RandomDeepSplitting}
 for the deep splitting method. 

 \medskip 

 Stochastic branching diffusion mechanisms
 \cite{sevastyanov},
 \cite{skorohodbranching},  \cite[Example~3.4]{inw}, represents an alternative to the BSDE method.
 See \cite{hpmckean}
 for an application to the Kolmogorov-Petrovskii-Piskunov (KPP) equation,
 \cite{chakraborty} for existence of solutions
 of parabolic PDEs with %
 power series nonlinearities, and \cite{henry-labordere2012}
 for the probabilistic representation of solutions of
 more general PDEs with polynomial nonlinearities. 
 A linear wave equation has been treated in  
 \cite{dalang}
 using Markov branching process, 
 and a family of 
 semilinear and higher-order hyperbolic PDEs has been considered in 
 \cite{labordere2}. 

 \medskip

 Probabilistic representations for the solutions of nonlinear
 PDEs with polynomial nonlinearities have also been obtained using 
 stochastic cascades, see 
 \cite{sznitman},
 \cite{ramirez}, 
 \cite{blomker},
 \cite{bakhtin}, 
 \cite{waymire},
 \cite{waymire1},
 \cite{waymire2},
 which provide numerical approximations of PDE solutions by
 series truncation. 
 Stochastic branching methods have also been applied
 in e.g. \cite{lm}, \cite{labordere} to polynomial gradient nonlinearities.  
  
\medskip

 In \cite{penent2022fully}, a stochastic branching method 
 that carries information on (functional) nonlinearities along a random tree
 has been introduced, with the aim of providing
 Monte Carlo schemes for the numerical solution of
 fully nonlinear PDEs with gradients of arbitrary orders.
 The performance of this scheme has been investigated in
 \cite{nguwipenentprivaultnv}, 
 \cite{burgers}, 
 \cite{nguwipenentprivault} 
 via numerical experiments. 
 
 \medskip

 However, the above references generally 
 assume the uniform integrability of 
 random functionals and/or the existence of a solution, see 
 \cite[Theorem~3.5]{labordere},
 \cite[Theorem~4.2]{penent2022fully}.  
 The goal of this paper is to carry out the stability analysis
of this numerical scheme by ensuring the
 (uniform) integrability 
of the random functionals involved in the algorithm, 
based on conditions on the
coefficients $f$ and $\phi$ appearing in the PDE~\eqref{nl-heat}. 
For this, we aim at controlling 
the explosion of the underlying branching process
in order to derive an interval of existence for the probabilistic
solution. 

\medskip

Although there exists a rich classical literature on random trees and
branching processes, see for example 
 \cite{kendall1948},
 \cite{Ott49},
 \cite{harris1963}, 
 \cite{athreya}, 
 \cite{brown-shubert}, 
 it focuses mainly on estimating 
 the progeny of processes. 
 However, in our setting we are dealing with random
 product functionals based on the derivatives
 of the equation coefficients $f$ and $\phi$,
 which requires us to analyse the multiplicative
 weighted progeny of branching processes. 

\medskip

We proceed by 
dominating the original branching process using a binary branching process, whose integrability is studied separately.
For this, we control the generating function
of the corresponding multiplicative weighted progeny
by showing that it satisfies a contact Hamilton--Jacobi
(HJ) equation. 
Although this equation cannot be solved explicitly, we show 
that its solution satisfies a simpler 
Hamilton--Jacobi equation admitting a closed-form solution,
from which we derive explicit bounds. 

\medskip

 In Theorems~\ref{thm-1} and \ref{main-visc}
 we respectively obtain probabilistic representations 
 for the classical, mild and viscosity
 solutions of the PDE~\eqref{nl-heat}
 under integrability conditions on random functionals.
 Those integrability properties are shown 
 to hold in Propositions~\ref{integ-final-1}-\ref{integ-final-2}
 under factorial and exponential growth assumptions 
 on the terminal condition $\phi$, 
 written in terms of the existence time $T$
 of solutions and on the derivatives 
$$ 
\| \pt^\alpha \phi  \|_\infty
  , \ 
\big\| \pt^\alpha [f^{(j)} (\phi )]\big\|_\infty,
 \quad j\geq 0. 
$$ 
 The proof of those results relies on 
 the control of the multiplicative weighted progeny of random
 binary trees
 and on stochastic dominance results
 derived Sections~\ref{s7} and \ref{s5}
 via the solution of a related Hamilton--Jacobi equation in
 Section~\ref{s7}.
 Our results also generalize the construction of
 \cite{huangprivault1} which has been
 applied to the Monte Carlo generation of ODE solutions
 in \cite{huangprivault2}. 

 \medskip 

 We proceed as follows.
After introducing the notation of relevant branching diffusion mechanism
in Subsections~\ref{sec-construction} and \ref{s-3}, 
 Subsection~\ref{s4} presents our main results
 Theorem~\ref{thm-1} and \ref{main-visc}
 on the probabilistic representation of the solutions
 of semilinear PDEs.
Sections~\ref{s9} and \ref{s9-2} are devoted to fully nonlinear examples
with explicit solutions and numerical comparisons.
 
 \medskip 

 Our proof proceeds in Section~\ref{s4.1} by rewriting
 the PDE~\eqref{nl-heat}
 as a PDE system~\eqref{pde-system-FK}
 which is solved in Proposition~\ref{mild-system}
 under suitable integrability conditions
 on relevant random functionals, 
 which are derived in Section~\ref{s7}.

 \medskip

 This approach is implemented by studying
 the integrability of the multiplicative weighted progeny
 of binary branching processes in Section~\ref{s7},
 using a dominating binary branching process
 introduced in Section~\ref{s5}. 
 Finally, Section~\ref{s11} deals with the
uniqueness and probabilistic representation
of PDE system solutions under 
 uniform integrability assumptions on random functionals. 

 \medskip

 {A numerical implementation of
 the branching mechanism proposed
 in this paper with application to the
 functional nonlinearity example~\eqref{b2}
 in Section~\ref{s9-2} is available at 
 \url{https://github.com/nprivaul/coding_trees_v2/},
    and includes
      comparisons of the algorithm of this paper
      with that of
      \cite{penent2022fully,nguwipenentprivault}, 
 and with the deep BSDE method, %
    for 
 Allen-Cahn equations %
   and Example~\eqref{nonl0}, 
 in dimensions up to $1000$.
 We observe in particular that the branching algorithms
 remain more stable than the BSDE method
 in dimensions $d=1000$ with $T=0.5$. 

 \subsubsection*{Multi-index notations}
\noindent 
We let
$\N_0 := \{0,1,2,\ldots \}$, 
$\N_1:= \{1,2,3,\ldots \}$,
and
$$
\N_{-1}:= \{-1,0,1,2,3,\ldots \}.
$$ 
 Given $\alpha = (\alpha_1, \alpha_2, \ldots, \alpha_d)
 \in \N_0^d$ 
 a $d$-dimensional multi-index or $d$-tuple
 of non-negative integers, we set 
\begin{equation*}
  |\alpha| : = \alpha_1 + \alpha_2 + \cdots + \alpha_d, 
\end{equation*}
and we let the factorial of $\alpha \in \N_0^d$
be defined as 
\begin{equation*}
  \alpha! = \alpha_1! \alpha_2! \cdots \alpha_d!.
\end{equation*}
 Using the partial order 
\begin{equation*}
  \beta \leq \alpha \quad \Leftrightarrow \quad \beta_i \leq \alpha_i \quad i = 1, \ldots, d, 
\end{equation*}
 on $\N_0^d$,
 the binomial coefficients can be generalized to multi-indices
 $\beta \leq \alpha \in \N_0^d$ as 
\begin{equation*}
  \binom{\alpha}{\beta}=\binom{\alpha_1}{\beta_1}\binom{\alpha_2}{\beta_2} \cdots\binom{\alpha_d}{\beta_d}=\frac{\alpha!}{\beta!(\alpha-\beta)!}. 
\end{equation*}
 The $\alpha$-power of a vector $x\in\R^d$ (or $\mathbb{C}^d$) is
\begin{equation*}
  x^\alpha = x_1^{\alpha_1} x_2^{\alpha_2} \ldots x_d^{\alpha_d}.
\end{equation*}
The partial derivative with order $\alpha$ is defined by
\begin{equation*}
  \pt^\alpha = \pt_1^{\alpha_1} \pt_2^{\alpha_2} \ldots \partial_d^{\alpha_d}, 
\end{equation*}
 where we denote $\pt_i := \pt_{x_i}$,  $i=1,\ldots,d$.
The zero multi-index is $\mathbf{0}  := (0,\ldots, 0)$, while the unit multi-index is $\mathbf{1}  := (1,\ldots, 1)$.
For $i\in\{1, \ldots, d\}$, we denote by $\ind_i$ the multi-index whose $i$-th component is 1 and all others are 0, i.e.,
\begin{equation*}
  \ind_i := (0, \ldots, 0, \underbrace{1}_{i\text{-th}}, 0,\ldots, 0).
\end{equation*}
Clearly, $\pt^{\mathbf{0}} = \id$, $\pt^{\ind_i} = \pt_i$. We will also use the gradient notation $\nabla = (\pt_1, \ldots, \pt_d)^\top$.

\section{Monte Carlo representation} %
\label{section2}
\noindent
 This section presents the probabilist representation
of solutions of the PDE~\eqref{nl-heat}
using a coded branching diffusion process. 
\subsection{Coding mechanism}
\label{sec-construction}
\noindent
In this section we introduce the use of {\em codes},
 as operators which will be used to rewrite the PDE~\eqref{nl-heat} as a PDE system in Lemma~\ref{classical-sol}.
\begin{definition}
  Let $\C_\partial$ and $\C_f$ denote the
  code sets consisting of the following maps:
\begin{equation}
  \label{codes-org}
 \C_\partial :=
  \left\{ \frac{1}{\alpha!} \pt^\alpha  \right\}_{\alpha \in \N_0^d}
  \quad
  \mbox{and}
  \quad 
    \C_f :=
  \left\{ \frac{1}{\alpha!} \pt^\alpha \circ f^{(j)} \right\}_{ ( \alpha , j ) \in \N_0^d \times \N_0 }
, 
\end{equation}
 and let $\C(f) := \C_f \cup \C_\partial $.
\end{definition}
\noindent
 In what follows, codes in $\C (f)$ will also be represented as
 elements of $\N_0^d \times \N_{-1}$, as
\begin{equation}
\label{u-indexed}
  c =
  \begin{cases}
    (\alpha,-1)
    &
    \mbox{ if } \displaystyle c = \frac{1}{\alpha!} \pt^\alpha,
    \ \alpha \in \N_0^d,
        \medskip
    \\
   (\alpha,j)
  & \mbox{ if } \displaystyle c = \frac{1}{\alpha!} \pt^\alpha \circ f^{(j)},
    \ (\alpha , j ) \in \N_0^d \times \N_0, 
  \end{cases}
\end{equation}
 and the action of a given code
 $c\in \C (f)$ on a smooth function $g:\real^d \to \real$
 is given as follows: 
\begin{equation}
\nonumber %
  c (g) (x) = 
  \begin{cases}
    \displaystyle
    \frac{1}{\alpha!} \pt^\alpha g (x) 
    & \mbox{ if } \displaystyle c = \frac{1}{\alpha!} \pt^\alpha,
        \medskip
    \\
    \displaystyle
    \frac{1}{\alpha!} \pt^\alpha [ f^{(j)} (g(x) ) ]
    & \mbox{ if } \displaystyle c = \frac{1}{\alpha!} \pt^\alpha \circ f^{(j)}. 
  \end{cases}
\end{equation} 
 Definition~\ref{djklsd11} introduces a mechanism ${\cal M}$
 acting on codes and only requiring at most
 binary branching, therefore simplifying the construction of
 \cite{penent2022fully},
 \cite{nguwipenentprivault}
 in which the expansion of $\pt^\beta \circ f^{(j+1)}$
 via the Fa\`a di Bruno formula can yield an 
 exponentially explosive number of branches. 
\begin{definition}
  \label{djklsd11}
   Let 
$\M: \C (f) \to ( \R \times \C (f) )
\bigcup
( \R \times \C (f) \times \C (f) )$
 denote the mechanism
defined by
\begin{align}
   \label{mechanism-org}
   &    \M\left( \frac{1}{\alpha!} \pt^\alpha \right)  =\ \left\{ \left( 1,\ \frac{1}{\alpha!} \pt^\alpha \circ f \right) \right\},
   \\
   \nonumber 
   &    \M\left( \frac{1}{\alpha!} \pt^\alpha \circ f^{(j)} \right)  =\ \left\{ \left( 1,\ \frac{1}{(\alpha-\beta)!} \pt^{\alpha-\beta} \circ f,\ \frac{1}{\beta!} \pt^\beta \circ f^{(j+1)} \right)
   \right\}_{\mathbf{0}\leq \beta\leq \alpha}
   \\
   \nonumber 
   & \quad 
   \bigcup
   \left\{
   \left( -\frac{1}{2} ( 1 + \beta_i )( 1 + \alpha_i-\beta_i ),\ \frac{1}{(\alpha-\beta+\ind_i)!} \pt^{\alpha-\beta+\ind_i},\ \frac{1}{(\beta+\ind_i)!} \pt^{\beta+\ind_i} \circ f^{(j+1)} \right) \right\}_{\mathbf{0}\leq \beta\leq \alpha \atop i=1,\ldots,d}, 
\end{align} 
$(\alpha , j)\in\N_0^d \times \N_0$.
\end{definition}
\noindent 
\begin{remark}
 We note that for any $c \in {\cal C}(f)$, the first
 component $z_1$ of $z \in {\cal M}(c)$ is a scalar, while the
 others are codes in the set ${\cal C}(f)$.
 The reason why we keep track of those first components in the
 mechanism ${\cal M}$ is twofold:
\begin{enumerate}[1)]
\item
 the scalar components 
 play the role of coefficients in the formal
 iteration \eqref{pde-system}
 and naturally give rise to a
 Hamilton--Jacobi equation,
 which is the key for the
 analysis of integrability, see Proposition~\ref{gfdkj4}; 
\item
 codes are kept in a same form 
 depending only on the orders of partial derivatives
 of the functions appearing in the remaining components, so that
 the assumptions on code bounds can be formulated using
 the partial derivatives of terminal data, see
 \eqref{bound-code-1}
 and
 \eqref{bound-code-2}.
\end{enumerate}
\end{remark} 
\subsection{Coded branching diffusion}
\label{s-3}
\noindent
 In this subsection, we construct
 an age-dependent coded branching diffusion process 
 using the following ingredients defined on a
 sufficiently large filtered probability space $(\Omega,\F,\P,(\F_s)_{s\geq 0})$ and with the set of labels $\mathbb{K} := \{\varnothing\}\cup \bigcup_{n\geq 1} \N_1^n$. 
\begin{itemize}
\item Diffusion structure: a family $(B^{\mathbf{k}})_{{\mathbf{k}}\in\mathbb{K}}$
  of independent standard $d$-dimensional Brownian motions. 
\item Lifetimes: a family
 $(\tau_{\mathbf{k}})_{{\mathbf{k}}\in\mathbb{K}}$
 of i.i.d. random times %
 with common probability density $\rho:\R^+\to\R$. 
\item Coding mechanism: the pair $({\cal C}(f) ,\M)$ constructed in
 \eqref{codes-org}-\eqref{mechanism-org}.  
\item Offsprings distributions: to every $c\in {\cal C}(f)$,
  we associate a random variable $I(c) = (I_k(c))_{1\leq k \leq |I(c)|}$
  in $\M(c)$ with nowhere zero probability distribution
  $q_c(z) := \P(I(c) = z) > 0$, $z\in \M(c)$,
  defined as follows.  
\begin{enumerate}[i)]
    \item 
 If $c$ is of the form $c= {\alpha!}^{-1} \pt^\alpha$,
 we let $q_c(z) :=1$, $z\in \M (c)$. 
\item
 If $c$ 
  is of the form
  $c = {\alpha!}^{-1} \pt^\alpha \circ f^{(j)}$,
  we let

\begin{subequations}
\begin{empheq}[left={\hskip-1.2cm q_c(z) :=\hskip-.2cm \;\empheqlbrace}]{align}
  \label{offsp-dist}
   &  \displaystyle
      q_\alpha^{(0)} :=
      \displaystyle
      \frac{\prod_{k=1}^d ( 1 + \alpha_k )^{-1}}{d+1}
      ,
      \quad
      \qquad
      \qquad
      \
      z_1 = 1,
    \medskip
   \\
  \label{offsp-dist-2}
   & \displaystyle
      q_\alpha^{(i)}(\beta) := \displaystyle
      \frac{12 |z_1|
        \prod_{k=1}^d ( 1 + \alpha_k )^{-1}
      }{(d+1)
        (2 + \alpha_i)(3 + \alpha_i )
        \displaystyle
      },
      \quad 
\displaystyle
      z_1 = -\frac{1}{2} ( 1 + \beta_i )( 1 + \alpha_i-\beta_i ), 
\end{empheq}
\end{subequations}
\noindent
$\mathbf{0}\leq \beta\leq \alpha$,
  $i=1,\ldots,d$,
  $z\in \M (c)$, 
\end{enumerate}
 based on the first component $z_1$ of $z\in \M (c)$, $c\in \C (f)$. 
\end{itemize}
 We also assume that the families $(B^{\mathbf{k}})_{{\mathbf{k}}\in\mathbb{K}}$,
 $(\tau_{\mathbf{k}})_{{\mathbf{k}}\in\mathbb{K}}$ and the $I(c)$'s are mutually independent.
\begin{remark}
\begin{enumerate}[i)]
\item
 See Lemma~\ref{lA3} for the fact that
 $\{q_c(z)\}_{z\in \M(c)}$ is a probability distribution.
\item
 While the choice of offsprings distributions $\{ q_c\}_{c\in {\cal C}(f)}$
 is not unique, the choice made in \eqref{offsp-dist}-\eqref{offsp-dist-2}
 allows us to formulate suitable bound assumptions on terminal
 data when studying integrability.
\item
 It is also worth noticing that the choice of
 $\{q_c\}_{c\in {\cal C}(f)}$ does not affect the emergence of
 Hamilton--Jacobi equations in Section~\ref{s7}. 
\end{enumerate}
\end{remark}
 Given an initial data $(t,x) \in [0,T]\times\R^d$ and an
 initial code $c\in {\cal C}(f)$, an
 age-dependent branching diffusion with killing
 is constructed as follows:
\begin{itemize}
\item Start from a branch at position $x\in\R^d$ at time $t$,
  labelled $\varnothing$ and coded by $c$. This is the zeroth generation and is the common ancestor of all branches.
  The branch evolves according to $X^\varnothing := x+ B^\varnothing_{\cdot-t}$ with lifetime $\tau_\varnothing$.
  If the end time of life $T_\varnothing := t+\tau_\varnothing> T$, the branch $\varnothing$ is killed and yields no offspring, else the branch splits at its end of life into $|I(c)|-1$ independent offsprings, which belong to the first generation. These offsprings are indexed by label $k$ and coded by $I_{k+1}(c)$, for $k = 1,\ldots, |I(c)|-1$.
\item For generation $n \geq 1$, let the label for a branch be given by ${\mathbf{k}}=(k_1, \ldots, k_{n-1}, k_n) \in \mathbb{N}_+^n$ and the code be given by $c^{{\mathbf{k}}} \in\C (f)$. Denote by ${\mathbf{k}}^{-} := (k_1, \ldots, k_{n-2}, k_{n-1})$ the parent branch of ${\mathbf{k}}$. The branch ${\mathbf{k}}$ starts from $X_{T_{{\mathbf{k}}^-}}^{{\mathbf{k}}^{-}}$ at time $T_{{\mathbf{k}}^{-}}$, according to $X^{\mathbf{k}} := X_{T_{{\mathbf{k}}^-}}^{{\mathbf{k}}^{-}}+ B^{\mathbf{k}}_{\cdot- T_{{\mathbf{k}}^-}}$ with lifetime $\tau_{\mathbf{k}}$. If the end time of life $T_{{\mathbf{k}}} := T_{{\mathbf{k}}^{-}} + \tau_{\mathbf{k}} > T$, the branch ${\mathbf{k}}$ is killed without splitting to any offspring, else the branch splits at its end of life into $|I ( c^{{\mathbf{k}}} )|-1$ independent offsprings, which belong to the $(n+1)$-th generation. These offsprings are indexed by label $({\mathbf{k}},k_{n+1})$ and coded by $I_{k_{n+1}+1}( c^{{\mathbf{k}}})$,
  for $k_{n+1} = 1,\ldots, |I ( c^{{\mathbf{k}}}) |-1$.
\end{itemize} 

\vskip-0.2cm

\noindent
 Figure~\ref{fjkld23-1} and \ref{fjkld23-2} present
 samples of coded and labelled random tree
 started the codes 
 $c=\alpha!^{-1}\partial^\alpha$. 
 and
 $c=\alpha!^{-1} \partial^\alpha \circ f^{(j)}$,
 respectively. 
 
 \vspace{-0.4cm}

\begin{figure}[H]
\centering
\tikzstyle{level 1}=[level distance=4cm, sibling distance=6cm]
\tikzstyle{level 2}=[level distance=4cm, sibling distance=5cm]
\tikzstyle{level 3}=[level distance=7.5cm, sibling distance=5.2cm]
\tikzstyle{level 4}=[level distance=7.5cm, sibling distance=5cm]
\begin{center}
\resizebox{0.88\textwidth}{!}{
\begin{tikzpicture}[scale=1.0,grow=right, sloped]
\node[rectangle,draw,black,text=black,thick, rounded corners=5pt,fill=gray!05]{\large $t$} 
child {
  node[%
     draw,black,text=black,thick, rounded corners=5pt,fill=gray!05,xshift=-0.5cm] {\large $T_\varnothing$} %
            child {
  node[name=data, %
         draw, rounded corners=5pt,fill=gray!05,xshift=0.5cm] (main) 
        {\large $T_{(1)}$} %
                child{
                  node[%
                     draw,black,text=black,thick,xshift=2cm,yshift=-1cm,rounded corners=5pt,fill=gray!05]{\large $T_{(1,2)}$} %
                    child{
    node[name=data, %
         draw, rounded corners=5pt,yshift=0.5cm,xshift=3.5cm,fill=gray!05]  
        {\large $T$} %
                    edge from parent
                    node[above,trapezium,draw,fill=gray!05]{$(1,2,2)$}
                    node[rectangle,draw,below,fill=gray!05]{$\cdots$}%
                    }
                    child{
    node[name=data, %
         draw, rounded corners=5pt,yshift=-1.1cm,xshift=3.5cm,fill=gray!05]  
        {\large $T$ \nodepart{two} $f$}
                    edge from parent
                    node[above,trapezium,draw,fill=gray!05]{$(1,2,1)$}
                    node[rectangle,draw,below,fill=gray!05]{\large $\cdots$} %
                    }
                edge from parent
                node[above,trapezium,draw,fill=gray!05]{$(1,2)$}
                node[rectangle,draw,below,fill=gray!05]{\large $\beta!^{-1} \partial^{\alpha - \beta} \circ f^{(1)}$}%
                }
                child{
                  node[%
                     draw,black,text=black,thick,yshift=0.4cm,rounded corners=5pt,fill=gray!05]{\large $T_{(1,1)}$} %
                    child{
                      node[%
                        draw,black,text=black,thick,yshift=0.4cm,rounded corners=5pt,fill=gray!05]{\large $T_{(1,1,2)}$} %
                    child{
    node[name=data, %
         draw, rounded corners=5pt,yshift=0.6cm,xshift=-2cm,fill=gray!05]  
        {\large $T$} %
                    edge from parent
                    node[above,trapezium,draw,fill=gray!05]{$(1,1,2,2)$}
                    node[rectangle,draw,below,fill=gray!05]{$\cdots$}%
                    }
                    child{
    node[name=data, %
         draw, rounded corners=5pt,yshift=-1cm,xshift=-2cm,fill=gray!05]  
        {\large $T$} %
                    edge from parent
                    node[above,trapezium,draw,fill=gray!05]{$(1,1,2,1)$}
                    node[rectangle,draw,below,fill=gray!05]{\large $\cdots$} %
                    }
                    edge from parent
                    node[above,trapezium,draw,fill=gray!05]{$(1,1,2)$}
                    node[rectangle,draw,below,fill=gray!05]{\large $(\alpha-\beta)!^{-1}\partial^{\alpha-\beta} \circ f^{(2)}$} %
                    }
                    child{
    node[name=data, %
         draw, rounded corners=5pt,fill=gray!05,xshift=5.5cm]  
        {\large $T$ \nodepart{two} $f$}
                    edge from parent
                    node[above,trapezium,draw,fill=gray!05]{$(1,1,1)$}
                    node[rectangle,draw,below,fill=gray!05]{\large $(\alpha-\beta-\gamma)!^{-1}\partial^{\alpha-\beta-\gamma}$} %
                    }
                edge from parent
                node[above,trapezium,draw,fill=gray!05]{$(1,1)$}
                node[rectangle,draw,below,fill=gray!05]{\large $(\alpha - \beta)!^{-1} \partial^{\alpha - \beta} \circ f$} %
                }
                edge from parent
                node[above,trapezium,draw,fill=gray!05] {$(1)$}
                node[rectangle,draw,below,fill=gray!05] {\large $\alpha!^{-1}\partial^\alpha \circ f$} %
            }
                edge from parent
                node[above,trapezium,draw,fill=gray!05] {$\varnothing$}
                node[rectangle,draw,fill=gray!05,below] {\large $\alpha!^{-1}\partial^\alpha $} %
};
\end{tikzpicture}
}
\end{center}
\vskip-0.5cm
\caption{Sample of random tree
 started from $c=\alpha!^{-1}\partial^\alpha$.}
\label{fjkld23-1}
\end{figure}

\vskip-0.1cm

\begin{figure}[H]
\centering
\tikzstyle{level 1}=[level distance=6cm, sibling distance=6cm]
\tikzstyle{level 2}=[level distance=7cm, sibling distance=6.5cm]
\tikzstyle{level 3}=[level distance=8.5cm, sibling distance=5cm]
\tikzstyle{level 4}=[level distance=7.5cm, sibling distance=4cm]
\begin{center}
\resizebox{0.83\textwidth}{!}{
\begin{tikzpicture}[scale=1.0,grow=right, sloped]
        \node[draw,black,text=black,thick, rounded corners=5pt,fill=gray!05] {\large $t$} 
            child {node[name=data, fill=gray!05,%
         draw, rounded corners=5pt] (main) 
        {\large $T_\varnothing$} %
                child{
                  node[%
                    draw,fill=gray!05,text=black,thick,yshift=0.4cm,rounded corners=5pt]{\large $T_{(2)}$ \nodepart{two} $\nabla c$} 
                    child{
    node[name=data, %
         draw, fill=gray!05,rounded corners=5pt,xshift=6cm,yshift=0cm]  
        {\large $T$ \nodepart{two} $\nabla^2 c$}
                    edge from parent
                    node[above,trapezium,draw,fill=gray!05]{$(2,2)$}
                    node[rectangle,draw,fill=gray!05,below]{\large $\gamma!^{-1}\partial^{\gamma+\ind_i}\circ f^{(j+2)}$}%
                    }
                    child{
                      node[%
                         draw,black,text=black,thick,xshift=1cm,yshift=0cm,fill=gray!15,rounded corners=5pt]{\large $T_{(2,2)}$ \nodepart{two} $\nabla f$} 
                    child{
    node[name=data, %
         draw,fill=gray!05, rounded corners=5pt,xshift=-2.5cm,yshift=0.5cm]  
        {\large $T$ \nodepart{two} $\nabla^2 f$}
                    edge from parent
                    node[above,trapezium,draw,fill=gray!05]{$(1,2,2)$}
                    node[rectangle,draw,fill=gray!05,below]{$\cdots $}%
                    }
                    child{
    node[name=data, %
         draw,fill=gray!05, rounded corners=5pt,yshift=-.5cm,xshift=-2.5cm]  
        {\large $T$ \nodepart{two} $faa$}
                    edge from parent
                    node[above,trapezium,draw,fill=gray!05]{$(1,2,1)$}
                    node[rectangle,draw,fill=gray!05,below]{\large $\cdots$} %
                    }
                    edge from parent
                    node[above,trapezium,draw,fill=gray!05]{$(2,1)$}
                    node[rectangle,draw,fill=gray!05,below]{\large $(\alpha-\beta-\gamma)!^{-1}\partial^{\alpha-\beta - \gamma+\ind_i}$} %
                    }
                edge from parent
                node[above,trapezium,draw,fill=gray!05]{$(2)$}
                node[rectangle,draw,fill=gray!05,below]{\large $(\alpha -\beta)!^{-1}\partial^{\alpha -\beta }\circ f^{(j+1)}$}%
                }
                child{
                  node[%
                     draw,black,text=black,thick,yshift=0.4cm,xshift=2cm,fill=gray!15,rounded corners=5pt]{\large $T_{(1)}$ \nodepart{two} $fgg$} %
                    child{
                      node[%
                         draw,black,text=black,thick,xshift=4cm,yshift=1cm,fill=gray!15,rounded corners=5pt]{\large $T$ \nodepart{two} $\nabla f$} 
                    edge from parent
                    node[above,trapezium,draw,fill=gray!05]{$(1,2)$}
                    node[rectangle,draw,fill=gray!05,below]{\large $\gamma!^{-1} \partial^{\gamma + \ind_i}\circ f^{(j+1)}$} %
                    }
                    child{
    node[name=data, %
         draw, fill=gray!05, yshift=0.5cm,xshift=4cm,rounded corners=5pt]  
        {\large $T$ \nodepart{two} $fqqq$}
                    edge from parent
                    node[above,trapezium,draw,fill=gray!05]{$(1,1)$}
                    node[rectangle,draw,fill=gray!05,below]{\large $(\alpha-\beta-\gamma)!^{-1}\partial^{\alpha-\beta-\gamma + \ind_i}$} %
                    }
                edge from parent
                node[above,trapezium,draw,fill=gray!05]{$(1)$}
                node[rectangle,draw,fill=gray!05,below]{\large $(\alpha-\beta)!^{-1}\partial^{\alpha - \beta} \circ f^{(j)}$} %
                }
                edge from parent
                node[above,trapezium,draw,fill=gray!05] {$\varnothing$}
                node[rectangle,draw,fill=gray!05,below]  {\large $\alpha!^{-1}\partial^\alpha \circ f^{(j)}$} %
}; %
\end{tikzpicture}
}
\end{center}
\vskip-0.3cm
\caption{Sample of random tree
  started from $c=\alpha!^{-1} \partial^\alpha \circ f^{(j)}$.}
\label{fjkld23-2}
\end{figure}

\vskip-0.3cm

\noindent  
 The above construction yields a coded branching process, denoted by
 $\mathcal{X}_t^x(c) = \big(\mathcal{X}_{t,s}^x(c)\big)_{s\in[t,T]}$,
 in which every branch has at most two offsprings.
 The label set of the branches which
 die before or at time $T$ is denoted by $\mathcal{K}^\circ_{t,T}(c)$,
 whereas the set of those which die after time $T$
 is denoted by $\mathcal{K}^{\rm b}_{t,T}(c)$.
 For $n\geq 1$, we also denote by $\mathcal{K}^n_{t,T}(c)$ the label set of branch of the $n$-th generation, and we let
 $$
 \mathcal{K}^{\circ,n}_{t,T}(c) := \mathcal{K}^\circ_{t,T}(c) \cap \mathcal{K}^n_{t,T}(c)
 \quad
 \mbox{and}
 \quad
 \mathcal{K}^{{\rm b},n}_{t,T}(c) := \mathcal{K}^{\rm b}_{t,T}(c) \cap \mathcal{K}^n_{t,T}(c),
 $$
 where the superscripts ``$\circ$'' and ``${\rm b}$'' stand
 for ``interior'' and ``boundary'', respectively.
 
\subsection{Probabilistic representation of PDE solutions}
\label{s4}
\noindent
 Consider the random functional 
\begin{equation}
\label{functional}
 \mathcal{H}_{t,T}\big(\mathcal{X}_t^x(c)\big) := \prod_{{\mathbf{k}} \in \mathcal{K}^{\rm b}_{t,T}(c)} \frac{c^{{\mathbf{k}}}(\phi ) \big(X_T^{\mathbf{k}}\big)}{\bar\rho(T-T_{{\mathbf{k}}^{-}})} \prod_{{\mathbf{k}} \in \mathcal{K}^\circ_{t,T}(c)} \frac{I_1( c^{{\mathbf{k}}})}{\rho(\tau_{\mathbf{k}}) q_{c^{{\mathbf{k}}}}(I(c^{{\mathbf{k}}}))} 
\end{equation}
 on the path space, where $\bar\rho$ is the tail distribution function 
$$
\bar\rho(s) := \P( \tau_{\mathbf{k}} >s ) = \int_s^\infty \rho(r) dr,
 \qquad s>0. 
$$

\begin{assumption}
\hypertarget{H}
  The probability density function $\rho$ satisfies
  \begin{enumerate}[i)]
  \item
 \hypertarget{density-waiting-time}{$\rho_*(T) := \min_{s\in[0,T]} \rho(s) > 0$}; and 
 \hfill %
\item 
 there exists $\lambda>0$ such that 
 \hypertarget{asmpt-1}{$\bar\rho(r) \geq e^{-\lambda r}$} for all $r\geq 0$.
\end{enumerate} 
\end{assumption}
\noindent
Assumption~\hyperlink{H}{\bf H}$_\rho$-$i)$ 
guarantees that the denominators of \eqref{functional} are nonzero,
and Assumption~\hyperlink{H}{\bf H}$_\rho$-$ii)$  
 states that the lifetime distribution can be stochastically
 dominated by an exponential distribution.

 \medskip
 
 The existence of $\rho$ satisfying
  Assumption~\hyperlink{H}{\bf H}$_\rho$
  and conditions similar to those in
  Propositions~\ref{integ-final-1}-\ref{integ-final-2}
  has been discussed in
  \cite[Lemma 8.1]{huangprivault2}.
  Examples include exponential distributions,
  as well as combinations of uniform and exponential
  distributions. 
 Propositions~\ref{integ-final-1}-\ref{integ-final-2},
 which are proved in Section~\ref{s7},
 ensure the integrability of $\mathcal{H}_{t,T}\big(\mathcal{X}_t^x(c)\big)$ 
 under growth conditions on the
 derivatives of the terminal condition $\phi$ and nonlinearity
 $f$. 
\begin{proposition}[Factorial growth] 
\label{integ-final-1}
 Under Assumption~\hyperlink{H}{\bf H$_\rho$},
 let $\theta, r>0$
 be such that
\begin{subequations}
\begin{empheq}[left=\empheqlbrace]{align}
\label{*}
  &    \displaystyle
 \theta r \geq \sqrt{\frac{2}{d}},
     \medskip
      \\
\label{*-2}  %
   &   \displaystyle
 \min ( \theta^2 r^2,1)
 \geq \frac{\rho_*(T) \bar\rho(T)}{d+1}, 
\end{empheq}
\end{subequations}
 let 
 \begin{equation}
\label{gthetar}
  g_{\theta , r} (\alpha) :=
  \frac{\theta^{|\alpha|}}{\alpha!}
  \prod_{l=0}^{|\alpha|-1} ( r + l), \quad \alpha\in \N_0^d,
\end{equation}
 and suppose that 
\begin{subequations}
\begin{empheq}[left=\empheqlbrace]{align}
 \label{bound-code-1}
  &
  \displaystyle
  \
  \max \Big( \| \pt^\alpha \phi  \|_\infty
  ,
  \frac{1}{\min ( \rho_*(T) , 1)}
  \sup_{k\geq 0} \big\| \pt^\alpha [f^{(k)} (\phi )]\big\|_\infty
  \Big) \leq
  \alpha! g_{\theta , r} (\alpha), 
  \ \ \alpha\in \N_0^d, 
  \smallskip
  \\
    \label{bound-radius}
    & \displaystyle
    \ 
    2^{r+2}
    \frac{1- e^{-\lambda T}}{\rho_*(T)\bar\rho(T) } <
    R_{\theta,r} := \frac{(r+1)^{r+1}}{2\theta^2 r (r+2)^{r+2}d}. %
\end{empheq}
\end{subequations} 
  Then, 
 $\mathcal{H}_{t,T}\big(\mathcal{X}_t^x(c)\big)$
 is integrable uniformly in $(t,x)\in[0,T]\times\R^d$,
 with 
\begin{equation*}
    \E \left[ \left| \mathcal{H}_{t,T}\big(\mathcal{X}_t^x(c)\big) \right| \right] \leq C_{\theta,r,d, \lambda,T} (2\theta d)^{|\alpha|}, \quad (t,x) \in [0,T]\times \R^d, 
\end{equation*}
 for any code $c\in \C (f)$ written as 
 $c = (\alpha , j)\in\N_0^d \times \N_{-1}$, 
 where $C_{\theta,r,d, \lambda,T}>0$ is a constant. 
\end{proposition}
\begin{proposition}[Exponential growth] 
\label{integ-final-2}
 Under Assumption~\hyperlink{H}{\bf H}$_\rho$,
 let $\theta, r>0$ be such that 
\begin{subequations}
\begin{empheq}[left=\empheqlbrace]{align}
\label{cond-theta-2-0}
  &    \displaystyle
     \medskip
     \theta \geq \sqrt{\frac{2}{d}},
     \\
\label{cond-theta-2-0-2}  
   &   \displaystyle
\min ( \theta^2 ,1)
 \geq \frac{\rho_*(T) \bar\rho(T)}{d+1}, 
\end{empheq}
\end{subequations}
 let 
\begin{equation}
\label{gtheta}
 g_\theta (\alpha) :=
 \frac{\theta^{|\alpha|}}{\alpha!}, \quad \alpha\in \N_0^d,
\end{equation} 
 and suppose that 
\begin{subequations}
\begin{empheq}[left=\empheqlbrace]{align}
     \label{bound-code-2}
  &    \displaystyle
     \max
     \Big( \| \pt^\alpha \phi  \|_\infty ,
     \frac{1}{\min ( \rho_*(T) , 1)} \sup_{k\geq 0} \big\| \pt^\alpha [f^{(k)} (\phi )]\big\|_\infty \Big) \leq
     \alpha! g_\theta (\alpha),
     \quad \alpha\in \N_0^d, 
      \medskip
      \\
\label{bound-radius-2}
   &   \displaystyle
 \frac{1- e^{-\lambda T}}{\rho_*(T)\bar\rho(T) }
 <
 R_\theta : = \frac{1}{2e\theta^2 d}. 
\end{empheq}
\end{subequations}
 Then, 
 $\mathcal{H}_{t,T}\big(\mathcal{X}_t^x(c)\big)$ is integrable uniformly in $(t,x)\in[0,T]\times\R^d$, with 
 \begin{equation*}
    \E \left[ \left| \mathcal{H}_{t,T}\big(\mathcal{X}_t^x(c)\big) \right| \right] \leq C_{\theta,d,\lambda,T} \left( \frac{d}{\log
       \big( R_\theta \rho_*(T)\bar\rho(T) /
      (1- e^{-\lambda T} ) \big)} \right)^{|\alpha|-1}, \quad (t,x) \in [0,T]\times \R^d,
\end{equation*}
 for any code $c\in \C (f)$ written
 as $c = (\alpha , j )\in\N_0^d \times \N_{-1}$,
 where $C_{\theta,d,\lambda,T}>0$ is a constant.   
\end{proposition}
\begin{remark}
  \begin{enumerate}[i)]
\item
  On the one hand, the expressions in \eqref{bound-radius} and
  \eqref{bound-radius-2} tell that when either $\theta$ or
  $r$ gets larger, $R_{\theta, r}$ and $R_\theta$ get smaller.
  On the other hand, when $r$ gets large,
  $2^{-(r+2)}R_{\theta, r}$ is less than
  $R_\theta$, while the bound of \eqref{bound-radius}
  is bigger than that of \eqref{bound-radius-2}.
  In summary, when the bounds on initial data and nonlinearity get looser, the condition on the existence time gets more strict, and vice versa.
  \item
 Using the probability density function
 $\rho(s) = \lambda e^{-\lambda s}$ which fulfills
      Assumption~\hyperlink{H}{\bf H}$_\rho$-\hyperlink{asmpt-1}{\rm ii)
      for $\lambda>0$}, %
 Conditions~%
\eqref{bound-radius} and \eqref{bound-radius-2} 
on the time interval $[0,T]$
yield the more explicit estimates
\begin{equation*}
T<\begin{cases}
    \displaystyle
   \frac{1}{\lambda} \log \left(
  \frac{1}{2}
  +
  \sqrt{\frac{1}{4} + \lambda 2^{-(r+2)} R_{\theta,r} } \right)
  < 2^{-(r+2)} R_{\theta,r}, 
 \medskip
    \\
    \displaystyle
    \frac{1}{\lambda} \log \left( \frac{1}{2}
    + \sqrt{\frac{1}{4} + \lambda R_\theta } \right) < R_\theta. %
\end{cases}
\end{equation*}
\item
 In Proposition~\ref{integ-final-2},
 $\E \left[ \left| \mathcal{H}_{t,T}\big(\mathcal{X}_t^x(c)\big) \right| \right]$
 decays exponentially
 with respect to the absolute
 order $|\alpha |$ of the derivative in the code $c \in \C (f)$ %
 under the condition 
\begin{equation*}
 \frac{1- e^{-\lambda T}}{\rho_*(T)\bar\rho(T) }< e^{-d} R_\theta. 
\end{equation*}
\item
\noindent 
 In \eqref{bound-code-1} and \eqref{bound-code-2},
 only the values of $f^{(k)}$ within the interval
 $\big[\inf_{x\in\R^d} \phi (x), \sup_{x\in\R^d} \phi (x)\big]$
  have an impact on the bounds of $\pt^\alpha [f^{(k)} (\phi )]$,
 hence the nonlinearity $f$
 can be modified arbitrarily outside of that interval.
  \end{enumerate}
\end{remark}
\noindent
Theorem~\ref{thm-1} provides the probabilistic
representation of solutions of
 the PDE~\eqref{nl-heat} as a consequence of
 Propositions~\ref{integ-final-1}-\ref{integ-final-2},
 Proposition~\ref{mild-system}, and Proposition~\ref{uniqueness-system}. 
\begin{theorem}
\label{thm-1}
 Suppose that
  \begin{enumerate}[i)]
  \item
    Assumption~\hyperlink{H}{\bf H}$_\rho$ 
 is satisfied,
\item
  the PDE~\eqref{nl-heat} admits a classical solution
  $u\in {\cal C}^{1,\infty}([0,T]\times\R^d)$, and
\item the function $c(u)(t)$ satisfies 
  \begin{equation}
    \label{***}
  \sup_{(\alpha , j ) \in
    \N_0^d \times \N_{-1}
    } C^{-|\alpha|} \| c(u)(t) \|_\infty < \infty
\end{equation}
 for all $t\in[0,T]$ and $c\in\C (f)$, where either:   
\begin{enumerate}[a)]
      \item
 $C := 2\theta d$ 
 for some $\theta > 0$ 
 such that \eqref{*}-\eqref{*-2} and 
 \eqref{bound-code-1}-\eqref{bound-radius}
 are satisfied for some $r>0$, 
 or 
\item
  $C := \max \left( 1 ,
 d / \log
       \big(\rho_*(T)\bar\rho(T) /
       (2 (1- e^{-\lambda T}) e \theta^2 d) \big)
 \right)$
  for some $\theta >0$ such that
  \eqref{cond-theta-2-0}-\eqref{cond-theta-2-0-2} and \eqref{bound-code-2}-\eqref{bound-radius-2}
  are satisfied. 
\end{enumerate}
\end{enumerate}
  Then, the solution $u$ of the PDE~\eqref{nl-heat} can be
  uniquely represented as
 the expectation 
 \begin{equation}
   \label{fkldsaa1} 
    u(t,x) = \E \big[ \mathcal{H}_{t,T}\big(\mathcal{X}_t^x ( \id ) \big) \big],
    \quad (t,x)\in [0,T]\times \R^d.
  \end{equation}
\end{theorem}
\begin{proof}
  Under the respective conditions $a)$ and $b)$,
  Propositions~\ref{integ-final-1} and \ref{integ-final-2}
  ensure the integrability of $\mathcal{H}_{t,T}\big(\mathcal{X}_t^x(c)\big)$ uniformly in $(t,x)\in[0,T]\times\R^d$ for each $c\in\C (f)$,
  and $c(\phi )$ is bounded on
  $\real^d$ for every $c\in\C (f)$ by \eqref{***}. 
  Thus, by Proposition~\ref{mild-system}-$i)$
  the family $\{u_c(t,x)\}_{c\in\C (f)}
  := \big\{ 
  \E \big[
    \mathcal{H}_{t,T}\big(\mathcal{X}_t^x(c)\big) ]
  \big\}_{c\in\C (f)}$
  is a mild solution of the PDE system
      \begin{equation}
\label{pde-system}
\left( \pt_t + \frac{1}{2} \Delta \right)
u_c + \sum_{z \in \mathcal{M}(c)} z_1 \prod_{i=2}^{|z|} u_{z_i} =0, \quad u_c(T) = c(\phi ).
\end{equation}
  Moreover, in both cases $a)$ and $b)$, 
  by Proposition~\ref{uniqueness-system}
  the family $\{u_c\}_{c\in\C (f)}$ is the unique mild solution
  of the PDE system~\eqref{pde-system}
  such that 
$$ 
  \sup_{(\alpha , j ) \in
    \N_0^d \times \N_{-1}
    } C^{-|\alpha|} \| c(u)(t) \|_\infty < \infty. 
  $$
  Hence, by Proposition~\ref{mild-system}-$ii)$
  we obtain that
  $$
  c(u) (t,x)= u_c(t,x)
  =
  \E \big[ \mathcal{H}_{t,T}\big(\mathcal{X}_t^x ( c ) \big) \big],
  \quad (t,x)\in [0,T]\times \R^d, 
  $$
  for all $c\in\C (f)$.
  The particular case $c = \id$ yields \eqref{fkldsaa1}. 
\end{proof}
\noindent 
Next, we address the existence of mild and viscosity solutions for the 
PDE~\eqref{nl-heat}. 
 Let 
 $( S(t))_{t\geq 0} = (e^{t \Delta / 2})_{t\geq 0}$ denote the heat semigroup
 generated by $\Delta / 2$, defined as $S(0)={\rm Id}$ and 
\begin{equation*}
  S(t) v(x) := \int_{\R^d} v(y) p_0(t,x-y) dy, \quad
  (t,x)\in (0,\infty )\times\R^d, 
\end{equation*}
where
$$
p_0(t,x) := \frac{1}{\sqrt{2\pi t}} e^{- {|x|^2} / (2t ) }, \quad
(t,x)\in (0,\infty )\times\R^d, 
$$
 is the standard heat kernel.
\begin{definition}
 Let $u\in {\cal C}([0,T]\times\R^d)$ be such that $u(T) = \phi $.
\begin{enumerate}[i)]
\item
 The function $u$ is called a mild solution of 
 the PDE~\eqref{nl-heat} if it satisfies 
$$ 
  u = S(T-t) \phi  + \int_t^T S(s-t) f(u)(s) ds, \quad
   t\in [0,T].
$$ 
\item
 The function $u$ is called a viscosity solution of 
 the PDE~\eqref{nl-heat} if the following two statements are satisfied: 
\begin{itemize}
  \item 
 (viscosity subsolution) for each $v\in {\cal C}^\infty((0,T)\times\R^d)$, if $u-v$ has a local maximum at a point $(t_0,x_0) \in (0,T)\times\R^d$ such that $u=v$ at $(t_0,x_0)$, then
\begin{equation*}
  \frac{\partial v}{\partial t} (t_0, x_0)+
  \frac{1}{2} \Delta
  v(t_0, x_0) + f(v(t_0, x_0)) \geq 0;
\end{equation*}
  \item 
 (viscosity supersolution) for each $v\in {\cal C}^\infty((0,T)\times\R^d)$, if $u-v$ has a local minimum at a point $(t_0,x_0) \in (0,T)\times\R^d$ such that $u=v$ at $(t_0,x_0)$, then
\begin{equation*}
  \frac{\partial v}{\partial t} (t_0, x_0)+
  \frac{1}{2} \Delta
  v(t_0, x_0) + f(v(t_0, x_0)) \leq 0.
\end{equation*}
\end{itemize} 
  \end{enumerate}
\end{definition}
\noindent
The existence of viscosity
  solutions for the PDE~\eqref{nl-heat}
 in Theorem~\ref{main-visc}
 is obtained 
  by applying the argument of e.g. \cite[Proposition~3.4]{labordere}
  or \cite[Theorem~3.1]{claisse} to
  the me\-cha\-nism~${\cal M}$. 
  In Proposition~\ref{unf-intg-2} we will provide sufficient
  integrability conditions
 ensuring that the consistency condition 
 $u_f = f(u_{\id})$ in Theorem~\ref{main-visc}
 is satisfied. 
\begin{theorem}
\label{main-visc}
 Assume that 
\begin{itemize}
\item $\mathcal{H}_{t,T}\big(\mathcal{X}_t^x(c)\big)$ is $L^1$-integrable
  for all $(t,x)\in[0,T]\times\R^d$
  and $c\in \C (f)$, 
  and
\item
  the functions
  $$
  u_\id (t,x)  := \E [ \mathcal{H}_{t,T}\big(\mathcal{X}_t^x(\id )\big) ]
  \quad
  \mbox{and}
  \quad 
  u_f(t,x) := 
  \E \big[
    \mathcal{H}_{t,T}\big(\mathcal{X}_t^x(f)\big) ],
  \quad (t,x)\in[0,T]\times\R^d,
  $$ 
  satisfy $u_f = f(u_{\id})$. 
\end{itemize}
   The following assertions hold true. 
\begin{enumerate}[i)]
\item
 If $c(\phi )$ is bounded on $\real^d$
  for all $c\in \C (f)$, 
 then 
  $u_{\id}$ is a mild solution of
            the PDE~\eqref{nl-heat}.
\item 
  If
 $u_\id$ 
  is continuous in
  $(t,x)\in[0,T]\times\R^d$
  and 
  $\mathcal{H}_{t,T}\big(\mathcal{X}_t^x(c)\big)$ is $L^1$-integrable uniformly in $(t,x)\in[0,T]\times\R^d$
  for all $c\in \C (f)$, 
 then $u_\id$ is a viscosity solution of
  the PDE~\eqref{nl-heat}.
\end{enumerate}
\end{theorem}
\begin{proof}
\noindent
$i)$ Since $u_f = f(u_{\id})$, it follows from
Proposition~\ref{mild-system}-$i)$ that
\begin{align*}
  u_\id(t) & = S(T-t) \phi  + \int_t^T S(s-t) u_f(s) ds
  \\
  & = S(T-t) \phi  + \int_t^T S(s-t) f(u_\id)(s) ds, \quad
   t\in [0,T].
\end{align*}
Thus, $u_{\id}$ is a mild solution of the PDE~\eqref{nl-heat}.

\smallskip

\noindent
 $ii)$
 Let $\{u_c\}_{c\in\C (f)}$ denote
 the family of functions defined as 
   \begin{equation}
\nonumber 
    u_c(t,x) := \E \big[ \mathcal{H}_{t,T}\big(\mathcal{X}_t^x(c)\big) \big], \quad (t,x)\in[0,T]\times\R^d, \ c\in \C (f). 
 \end{equation}
  For any $\delta>0$, it follows from Lemma~\ref{lemma-uc}  
  and the Markov property that
\begin{equation*}
  \begin{split}
    u_c(t,x) &= \E \left[ u_c \big(t+\delta, X_{t+\delta}^\varnothing\big) + \int_t^{t+\delta} \sum_{z \in \mathcal{M}(c)} z_1 \prod_{i=2}^{|z|} u_{z_i}(s, X^\varnothing_s) ds \right].
  \end{split}
\end{equation*}
Let $v\in {\cal C}^\infty((0,T)\times\R^d)$ such that $u-v$ has a local minimum $0$ at $(t,x) \in (0,T)\times\R^d$.
 By It\^o's formula we have 
\begin{equation*}
  \begin{split}
    \E \left[ v_c \big(t+\delta, X_{t+\delta}^\varnothing\big) \right] &= v_c(t,x) + \E\left[ \int_t^{t+\delta} \left(
      \frac{\partial v_c}{\partial s}(s,X^\varnothing_s) 
      + \frac{1}{2} \Delta
      v_c(s,X^\varnothing_s)
      \right)
      ds \right], 
  \end{split}
\end{equation*}
 hence 
\begin{equation*}
  \begin{split}
    \E\left[ \int_t^{t+\delta} \left(
      \frac{\partial v_c}{\partial s} (s,X^\varnothing_s)
      + \frac{1}{2} \Delta v_c(s,X^\varnothing_s)
      + \sum_{z \in \mathcal{M}(c)} z_1 \prod_{i=2}^{|z|} u_{z_i}
      (s,X^\varnothing_s)
      \right)
      ds \right] \geq 0.
  \end{split}
\end{equation*}
It then follows from the mean value and
dominated convergence theorems that
\begin{equation*}
  \frac{\partial v_c}{\partial t} (t,x)
  + \frac{1}{2} \Delta v_c (t,x)
  + \sum_{z \in \mathcal{M}(c)} z_1 \prod_{i=2}^{|z|} u_{z_i} (t,x) \geq 0.
\end{equation*}
Set now $c=\id$. Since $\M(\id) = \{(1, f) \}$ and $u_f (t,x) = f(u(t,x)) = f(v(t,x))$, it follows that
\begin{equation*}
 \frac{\partial }{\pt t} v(t,x) + \frac{1}{2} \Delta v (t,x)+ f ( v (t,x) ) \geq 0,
\end{equation*}
 which shows that $u$ is a viscosity subsolution. The assertion that $u$ is a viscosity supersolution follows similarly.
\end{proof}
\section{Examples}
\label{s9}
\subsection{Generic examples}
\label{jklfda1} 
\noindent 
We consider the upper bounds on codes in \eqref{bound-code-1}
and \eqref{bound-code-2}, written as 
\begin{equation}\label{upper-bound-unf}
  \max \Big(
  \| \pt^\alpha \phi  \|_\infty
  ,
  \sup_{k\geq 0} \big\| \pt^\alpha [f^{(k)} (\phi )]\big\|_\infty
  \Big)
  \leq \Theta(|\alpha|), \quad \alpha\in \N_0^d,
\end{equation}
for some sequence $(\Theta(m))_{m\geq 0}$.
 We note that \eqref{upper-bound-unf} 
 simplifies in the following two trivial situations. 
 \begin{enumerate}[1)] 
 \item The nonlinearity $f$ is either a constant or 
   the identity function. %
   In this case, \eqref{upper-bound-unf} reduces to
  \begin{equation*}
    \| \pt^\alpha \phi  \|_\infty \leq \Theta(|\alpha|), \quad
     \alpha\in \N_0^d.
  \end{equation*}
\item The terminal data $\phi $ is constant. 
  In this case, \eqref{upper-bound-unf} is satisfied
  for any smooth nonlinearity~$f$.   
\end{enumerate} 
 In Lemma~\ref{jkld133}, we
 separate the assumptions on the terminal data $\phi $
 and the nonlinearity $f$ in \eqref{upper-bound-unf}.
 \begin{lemma}
   \label{jkld133}
     Assume that $f^{(j)}(x)$ is uniformly bounded in $x\in \real^d$ and
 $j \geq 0$, and that    
\begin{equation*}
  \| \pt^\alpha \phi  \|_\infty \leq K(|\alpha|), \quad
   \alpha\in \N_0^d, %
\end{equation*}
 where $(K(m))_{m\geq 0}$ is a positive sequence.
 Then, condition~\eqref{upper-bound-unf} is satisfied provided that 
 \begin{enumerate}[i)]
 \item
   $
K(m) \leq \Theta(m)$, $m \geq 1$,
and
\item 
  $\displaystyle 
     \sup_{k\geq 0} \big\| f^{(k)}(\phi ) \big\|_\infty
     B_m(K(1), \ldots, K(m))
     \leq  \Theta(m)$,
     $m \geq 1$,
\end{enumerate} 
  where $B_m (x_1,\ldots , x_m )$
 denotes the complete exponential Bell polynomial
 of order $m\geq 1$.
\end{lemma} 
\begin{Proof}
  By the multivariate Fa\`a di Bruno formula 
\begin{equation*}
  \frac{\pt^n}{\pt x_{k_1} \cdots \pt x_{k_n}} [f^{(l)} (\phi )] = \sum_{\pi \in \Pi [ n ]} f^{(l+\#\pi)}(\phi ) \prod_{A \in \pi} \frac{\pt^{\#A} \phi }{\prod_{j\in A} \pt x_{k_j}},
\end{equation*}
where $\Pi [ n ]$ is the set of all partitions of
$\{1, \ldots, n\}$ and $\#$ denotes set cardinality,
  we have 
\begin{equation*}
  \begin{split}
    \big\| \pt^\alpha [f^{(l)} (\phi )] \big\|_\infty &\leq \sum_{\pi \in \Pi [ |\alpha| ] } \| f^{(l+\#\pi)}(\phi ) \|_\infty \prod_{A \in \pi} K(\#A) \\
    &\leq \sup_{k\geq 0} \big\| f^{(k)}(\phi ) \big\|_\infty \sum_{\pi \in \Pi[ |\alpha| ] } \prod_{A \in \pi} K(\#A) \\
    &\leq \sup_{k\geq 0} \big\| f^{(k)}(\phi ) \big\|_\infty B_{|\alpha|}(K(1), \ldots, K(|\alpha|)).
  \end{split}
\end{equation*}
\end{Proof}
\noindent
In the following cases, assuming
$\sup_{k\geq 0} \big\| f^{(k)}(\phi ) \big\|_\infty<\infty$
we have several possible choices of $K(m)$ 
when $\Theta (m)$ has the factorial form used in \eqref{bound-code-1}:
\begin{enumerate}[1)] 
\item $K(m) = \theta^m$ for some $\theta >0$ 
  and $\Theta(m) \sim \theta^m m!$.
 In this case, we have 
  $$
  B_m( \theta, \theta^2 , \ldots, \theta^m ) = \sum_{k=1}^m \theta^k S(m,k) \leq \theta^m B_m(1, \ldots, 1) \leq \Theta(m), %
  $$
  where $S(m,k)$'s are Stirling numbers of the second kind.   
  This includes the case where
  the derivatives of $\phi $ are uniformly bounded
  when $K (m) \equiv 1$ is constant.   
\item
  $K(m) = \theta^m (m-1)!$
  and $\Theta(m) \sim \theta^m m!$
  for some $\theta>0$.
 In this case, we have 
  $$
  B_m(\theta 0!, \ldots, \theta^m (m-1)!) = \theta^m B_m(0!, \ldots, (m-1)!) = \theta^m m!,
  $$
  same choice for $\Theta$ as in $(1)$ above. 
\item $K(m) = \theta^m m!$ for some $\theta>0$ 
    and $\Theta(m) \sim (2\theta)^m m!$
  for some $\theta>0$.
 In this case, we have 
  $$
    B_m(\theta 1!, \ldots, \theta^m m!)  = \theta^m B_m(1!, \ldots, m!)
     = \theta^m \sum_{k=1}^m \frac{(m-1)!}{(k-1)!} \binom{m}{k}
     < (2\theta)^m (m-1)!. 
$$ 
\end{enumerate} 
\subsection{Functional nonlinearity example} 
\noindent
 In what follows, we denote 
 $$
 \langle x \rangle := x_1+ \cdots +x_d,
 \quad
 x\in \real^d.$$
 We consider the semilinear PDE %
\begin{numcases}{} 
  \nonumber
  \displaystyle
    \frac{\partial u}{\partial t} (t,x) + \frac{1}{2} \Delta u (t,x) + 4 e^{-u (t,x) }-10 e^{-u(t,x) / 2}+e^{u (t,x) / 2}-e^{u(t,x) } + 6 =0,
    \medskip
    \\
    \label{b2}
    \displaystyle
   u(T,x) %
   = 2 \log \left( \frac{2+e^{\langle x \rangle / \sqrt{d}}}{1+e^{\langle x \rangle / \sqrt{d}}} \right),
\end{numcases}
 $(t,x) \in [0,T] \times \real^d$,
 which admits the ${\cal C}^{\infty}$ solution
$$
 u(t, x)=2 \log \left( \frac{2+e^{\langle x \rangle / \sqrt{d}-(T-t)}}{1+e^{\langle x \rangle / \sqrt{d}-(T-t)}} \right) 
 \in (0, 2\log 2), \quad(t, x) \in[0, T] \times \mathbb{R}^d. 
$$
\begin{proposition}
\label{fjld1a}
  The nonlinearity $f(u)=4 e^{-u}-10 e^{-u / 2}+e^{u / 2}-e^u + 6$
  and terminal data $\phi (x) = 2 \log ( 1 + 1/(1+e^{\langle x \rangle / \sqrt{d}}) )$
  satisfy condition~\eqref{bound-code-1} for $r=2$, i.e. we have      
  $$ 
  \big\| \partial^\alpha \phi \big\|_\infty < 2 \left(
  \frac{(1+\theta)\theta}{\sqrt{d}}\right)^{|\alpha |} (|\alpha |-1)!
  \ \mbox{ and } \ 
   \sup_{k\geq 0}
   \big\| \partial_\alpha [f^{(k)}(\phi )] \big\| < 10
   \left( \frac{(1+\theta)\theta}{\sqrt{d}} \right)^{|\alpha |} (|\alpha |+1)!, 
$$
 $\alpha\in \N_0^d$. 
\end{proposition} 
\begin{Proof}
  For simplicity, the proof is first stated with $d=1$,
  and then extended to $d\geq 1$ using the relation
  $\partial_\alpha g ( \langle x \rangle /\sqrt{d} )
  = d^{-|\alpha |/2} g^{(|\alpha |)}( \langle x \rangle /\sqrt{d} )$,
  $x\in \real^d$, for any $g\in {\cal C}^\infty (\real^d)$.
  Letting $\zeta (x) := {1} / ( 1+e^x) \in (0,1)$, 
 by the Fa\`a di Bruno formula we have 
\begin{equation*}
  \zeta^{(m)}(x)  =\sum_{k=1}^m \frac{(-1)^k k!}{(1+e^x)^{k+1}} B_{m, k}(e^x, \ldots, e^x) = \sum_{k=1}^m \frac{(-e^x)^k k!}{(1+e^x)^{k+1}} S(m, k),
  \quad m\geq 1, 
\end{equation*}
 by Lemma~\ref{polylog-stirling} we have 
\begin{equation}\label{phi-bound}
  \big\|\zeta^{(m)}\big\|_\infty \leq \sum_{k=1}^m k!S(m, k) = \frac{1}{2} \mathrm{Li}_{-m}(1 / 2) \sim \frac{m!}{2(\log 2)^{m+1}} \leq \theta^{m+1} m!
\end{equation}
for some $\theta>1$.
As $\phi (x) = 2\log (1+\zeta (x))$, we use Fa\`a di Bruno's formula again to derive
\begin{equation*}
  \begin{split}
    \phi^{(m)}(x) &= 2 \sum_{k=1}^m
    (-1)^{k-1}
    \frac{(k-1)!}{(1+\zeta (x))^k} B_{m, k}(\zeta'(x), \ldots, \zeta^{(m-k+1)}(x)),
    \quad m \geq 1,
  \end{split}
\end{equation*}
which, together with \eqref{phi-bound}, implies
\begin{align*}
  \big\|\phi^{(m)}\big\|_\infty &\leq 2 \sum_{k=1}^m (k-1)! B_{m, k}
  ( \theta^2 1!, \ldots, \theta^{m-k+2} (m-k+1)! )
  \\
  &= 2 \sum_{k=1}^m (k-1)! \theta^{m+k} B_{m, k}
  ( 1!, \ldots, (m-k+1)! )
  \\
    &= 2 \sum_{k=1}^m (k-1)! \theta^{m+k} \frac{m!}{k!} \binom{m-1}{k-1}
    \\
     & = 2 \theta^m (m-1)! \sum_{k=1}^m \theta^k \binom{m}{k} \\
    &< 2 ( \theta(1+\theta) )^m (m-1)!,
       \quad m \geq 1.
\end{align*}
 On the other hand, the nonlinearity $f(u)=4 e^{-u}-10 e^{-u / 2}+e^{u / 2}-e^u + 6$ satisfies 
\begin{align*}
  f^{(k)}(\phi (x)) & = \frac{4 (-1)^k}{(1+\zeta (x))^2}
  - \frac{10 (-1)^k}{2^k (1+\zeta (x))}
  +\frac{1+\zeta (x)}{2^k} -(1+\zeta (x))^2 +6 \ind_{\{k=0\}}, 
  \quad k \geq 0, 
\end{align*}
 i.e. $f^{(k)}(\phi (x)) = \psi_k(\zeta (x))$ with 
\begin{align*}
  \psi_k(z) := \frac{4 (-1)^k}{(1+z)^2}
  - \frac{10 (-1)^k}{2^k (1+z)}
  +\frac{1+z}{2^k} -(1+z)^2 +6 \ind_{\{k=0\}}, 
  \quad k \geq 0.
\end{align*}
Since $\zeta $ takes values in the interval $(0,1)$, it can be seen that
$\sup_{k\geq 0} \big\| f^{(k)}(\phi ) \big\|_\infty < \infty$.
One may also estimate the derivatives $f^{(l)}(\phi )$ in the same way as those of $\phi $, as 
\begin{equation*}
  \big\| \psi_l^{(k)} \big\|_\infty \leq 4(k+1)! + 10 k! + 1 +2 < 10(k+1)!,
  \quad
 k,l\geq 0.
\end{equation*}
Again, this together with \eqref{phi-bound} and Fa\`a di Bruno's formula yields 
\begin{align*}
  \big\| [f^{(k)}(\phi )]^{(m)} \big\| & \leq 10 \sum_{l=1}^m (l+1)! B_{m, l}
  ( \theta^2 1!, \ldots, \theta^{m-l+2} (m-l+1)! )
  \\
   & < 10 ( \theta(1+\theta) )^m (m+1)!, \quad k,m\geq 0.
\end{align*}
\end{Proof}
\section{Numerical examples}
\label{s9-2}
\noindent 
 In this section, we compare our binary branching to the
 Fa\`a di Bruno branching of \cite{penent2022fully}
 by adapting the PyTorch implementation of \cite{nguwipenentprivault}, 
 and to the BSDE method of 
 \cite{han2018solvingarxiv,han2018solving},
 for which we use the TensorFlow implementation 
 available at \url{https://github.com/frankhan91/DeepBSDE},
 run with $4000$ iterations.
 The three algorithms are run on separate
 NVIDIA L40 GPUs with 46GB of memory. %

 \medskip

 In the following examples, the
 terminal condition $\phi (x)$ is such that
 $\Vert \phi^{(n)}\Vert_\infty$ has at most factorial growth
 in $n\geq 1$, hence they belong to the setting of either
 Lemma~\ref{jkld133} or Proposition~\ref{fjld1a}.
 Regarding numerical performance, we note the following. 
\begin{itemize}
\item
 The algorithm of 
 \cite{penent2022fully,nguwipenentprivault}
 achieves the best overall numerical performance.
\item Our binary branching algorithm is generally less
  stable numerically, probably due to the need to
  the computation of iterated gradients
  $\pt^\alpha [f^{(k)} (\phi )]$
  instead of applying explicit 
  Fa\`a di Bruno expansions as in
  \cite{penent2022fully,nguwipenentprivault}.
\item The branching algorithms can be 
  a hundred times faster than the BSDE method for smaller time horizons,
  while retaining comparable numerical performance.
\item The BSDE method may blow up and produce NaN values in dimension
  $d=1000$. 
  \end{itemize} }    
\begin{enumerate}[1)]  
\item %
  Let $f(u) = u - u^3$, and
  consider the Allen-Cahn (or Ginzburg-Landau) equation
\begin{numcases}{} 
\label{gl0}
  \displaystyle
    \frac{\partial u}{\partial t} u(t,x) + \Delta_x u(t,x) + u(t,x) - u^3(t,x) = 0,
    \medskip
    \\
    \nonumber
    \displaystyle
   u(T,x) 
   = \frac{1}{2 + 2 ( x_1^2+\cdots + x_d^2) / 5},
  \end{numcases}
 $(t,x) \in [0,T]\times \real^d$, 
as in \S~4.2 of \cite{han2018solvingarxiv}
or \cite{han2018solving}, which does not admit a closed form
solution. 
 
\begin{figure}[H]
\centering
\begin{subfigure}{.49\textwidth}
\includegraphics[width=\textwidth]{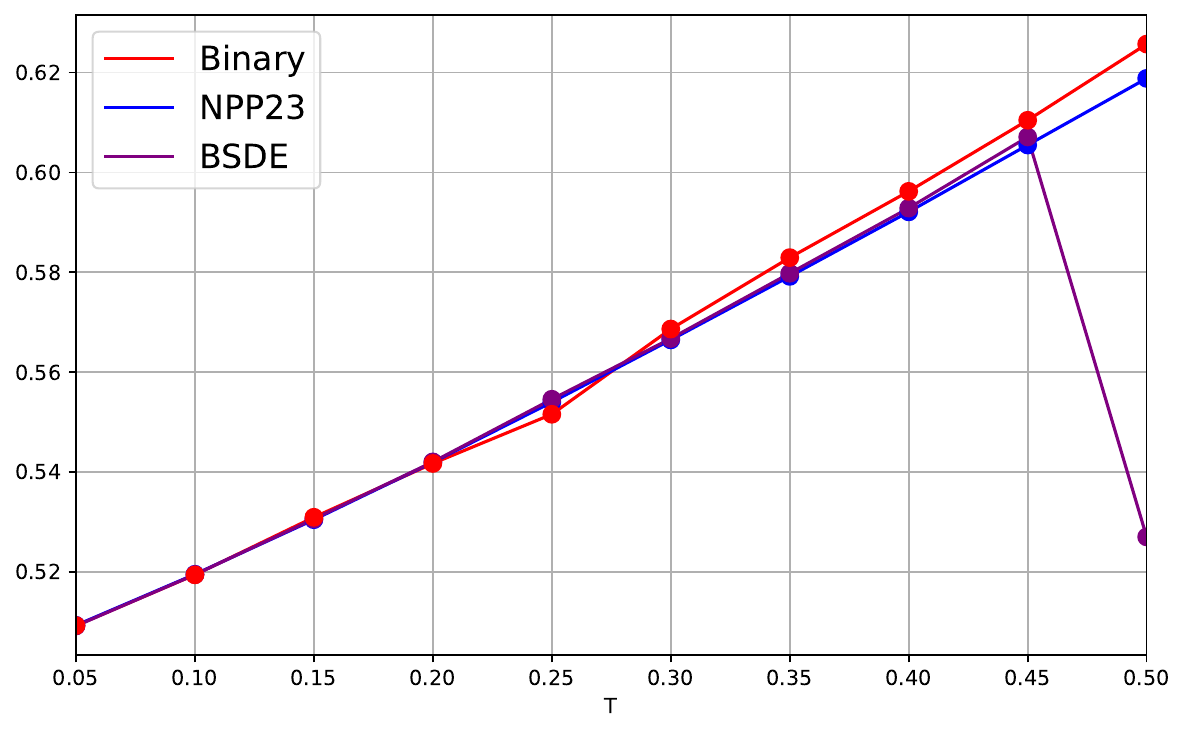}
\vskip-0.1cm
\caption{Dimension $d=1$.}
\end{subfigure}
\begin{subfigure}{.49\textwidth}
  \includegraphics[width=\textwidth]{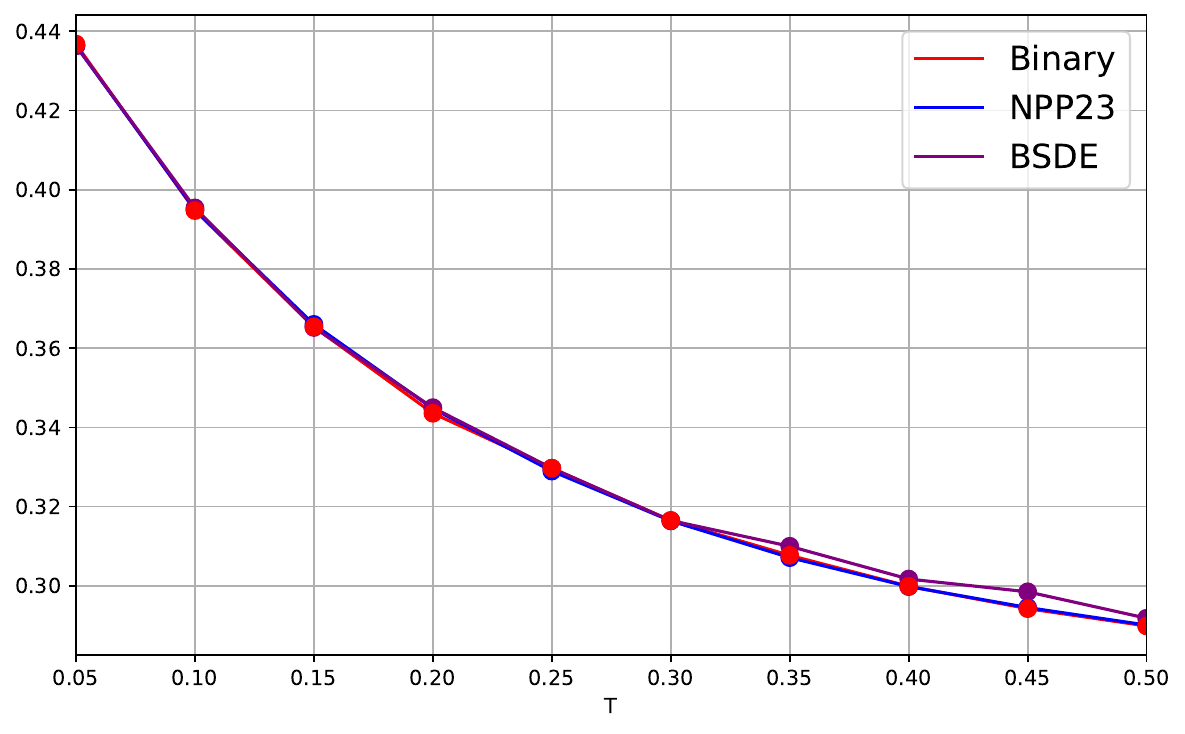}
\vskip-0.1cm
\caption{Dimension $d=10$.}
\end{subfigure}
\begin{subfigure}{.49\textwidth}
  \includegraphics[width=\textwidth]{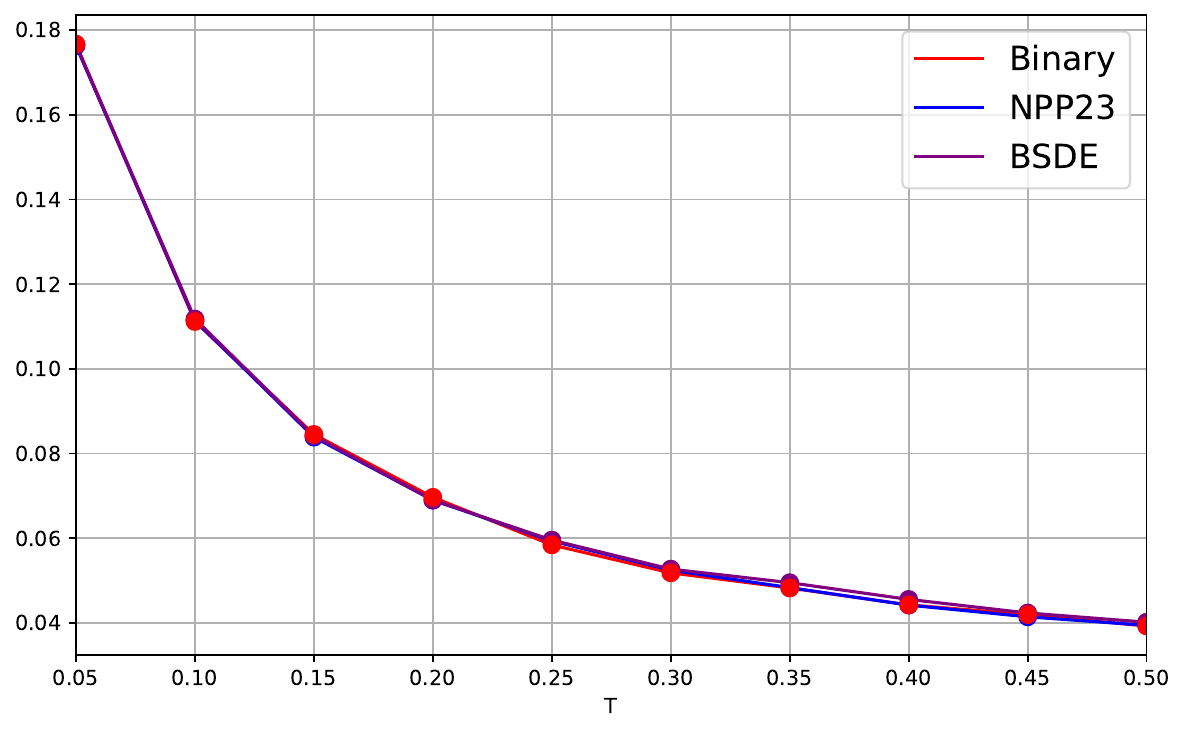}
\vskip-0.1cm
\caption{Dimension $d=100$.}
\end{subfigure}
\begin{subfigure}{.49\textwidth}
  \includegraphics[width=\textwidth]{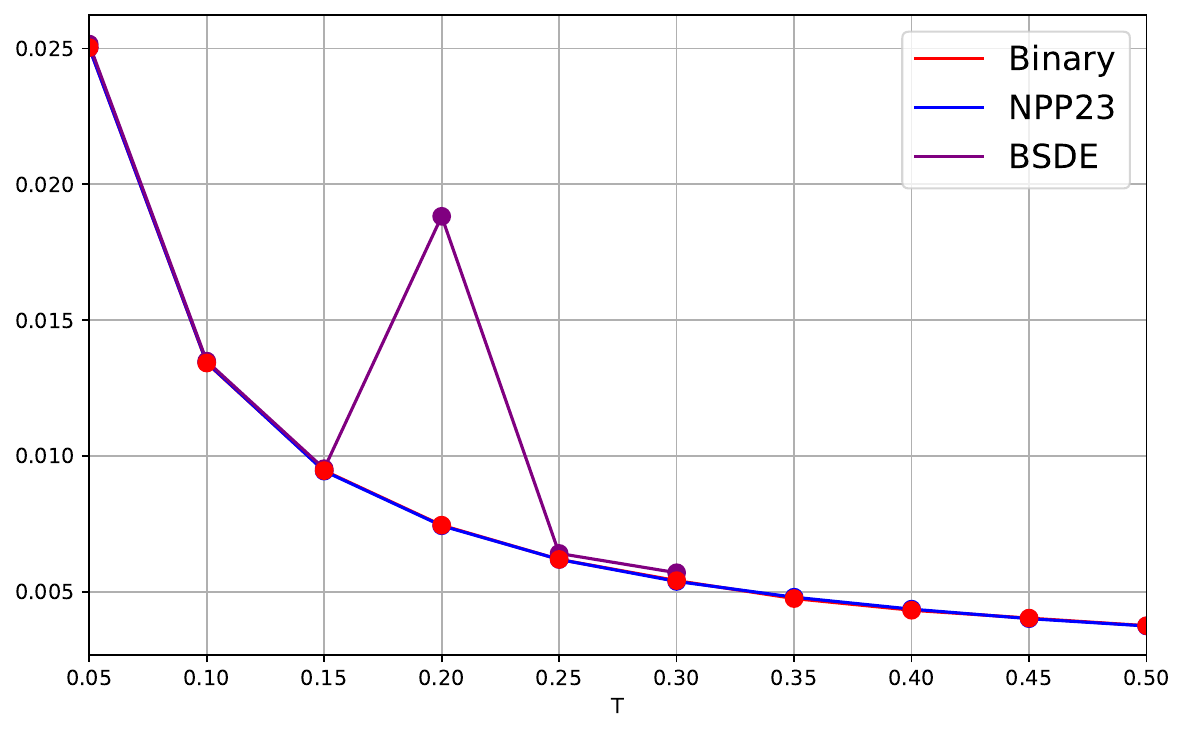}
\vskip-0.1cm
\caption{Dimension $d=1000$.}
\end{subfigure}
\caption{Numerical solution $u(0,x)$ of \eqref{gl0}.} %
\end{figure}

\vspace{-0.5cm}

\begin{table}[H]
  \centering
  \footnotesize
  \input{allen_cahn_3_summary.tex}
  \caption{Numerical estimates and computation times (s) for \eqref{gl0} at $T=0.5$.}
\end{table}

\vspace{-0.3cm}

\begin{table}[H]
  \centering
  \footnotesize
  \input{allen_cahn_3_dim_1000_time_behavior_t.tex}
  \caption{Numerical estimates and computation times (s) for \eqref{gl0} with $d=1000$.}
\end{table}

\item Consider the Allen-Cahn equation
\begin{numcases}{} 
  \label{gl}
  \displaystyle
    \frac{\partial u}{\partial t} u(t,x) + \frac{1}{2} \Delta_x u(t,x) + u(t,x) - u^3(t,x) = 0, 
    \medskip
    \\
    \nonumber
    \displaystyle
   u(T,x) 
   = 
 -\frac{1}{2} + \frac{1}{2}
\tanh \left( \sum_{i=1}^d \frac{x_i}{2 \sqrt{d}} \right),
 \end{numcases}
$(t,x) \in [0,T]\times \real^d$,
with the traveling wave solution
\begin{equation}
\nonumber %
  u(t,x) = -\frac{1}{2} - \frac{1}{2}
\tanh \bigg( \frac{3}{4} (T-t) - \sum_{i=1}^d \frac{x_i}{2 \sqrt{d}} \bigg),
\quad
(t,x) \in [0,T]\times \real^d. 
\end{equation}

\begin{figure}[H]
\centering
\begin{subfigure}{.49\textwidth}
  \includegraphics[width=\textwidth]{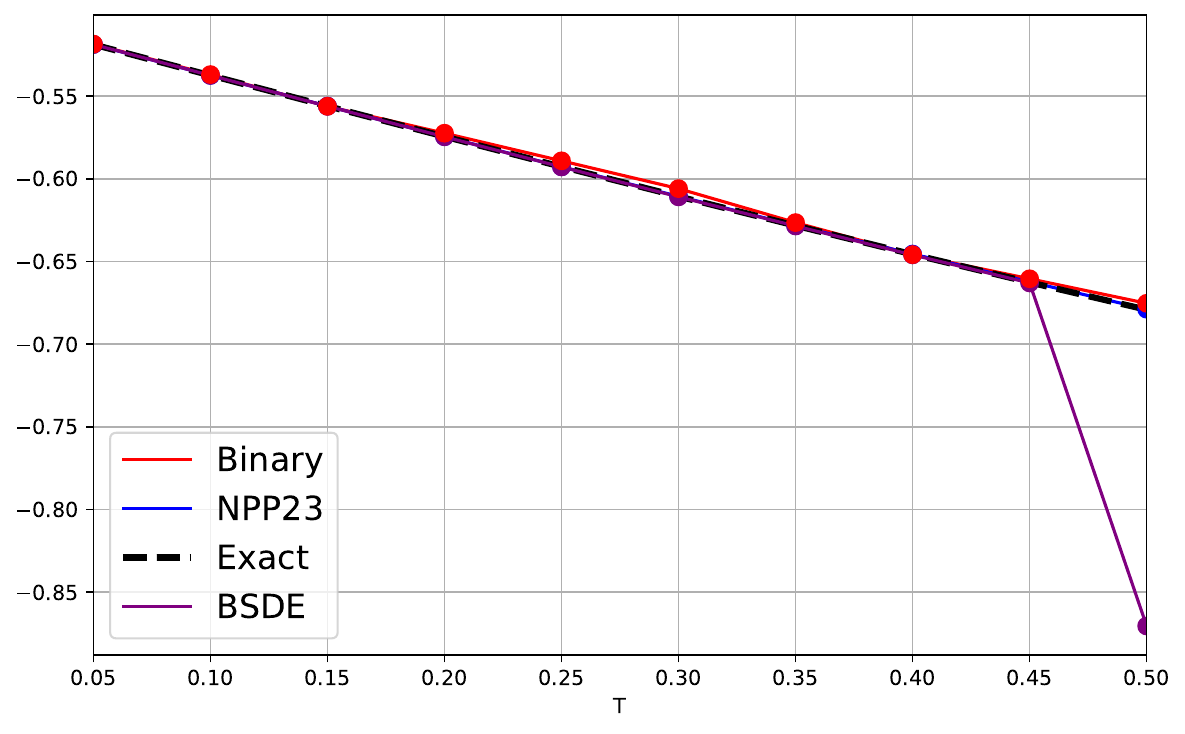}
\vskip-0.1cm
\caption{Dimension $d=1$.}
\end{subfigure}
\begin{subfigure}{.49\textwidth}
  \includegraphics[width=\textwidth]{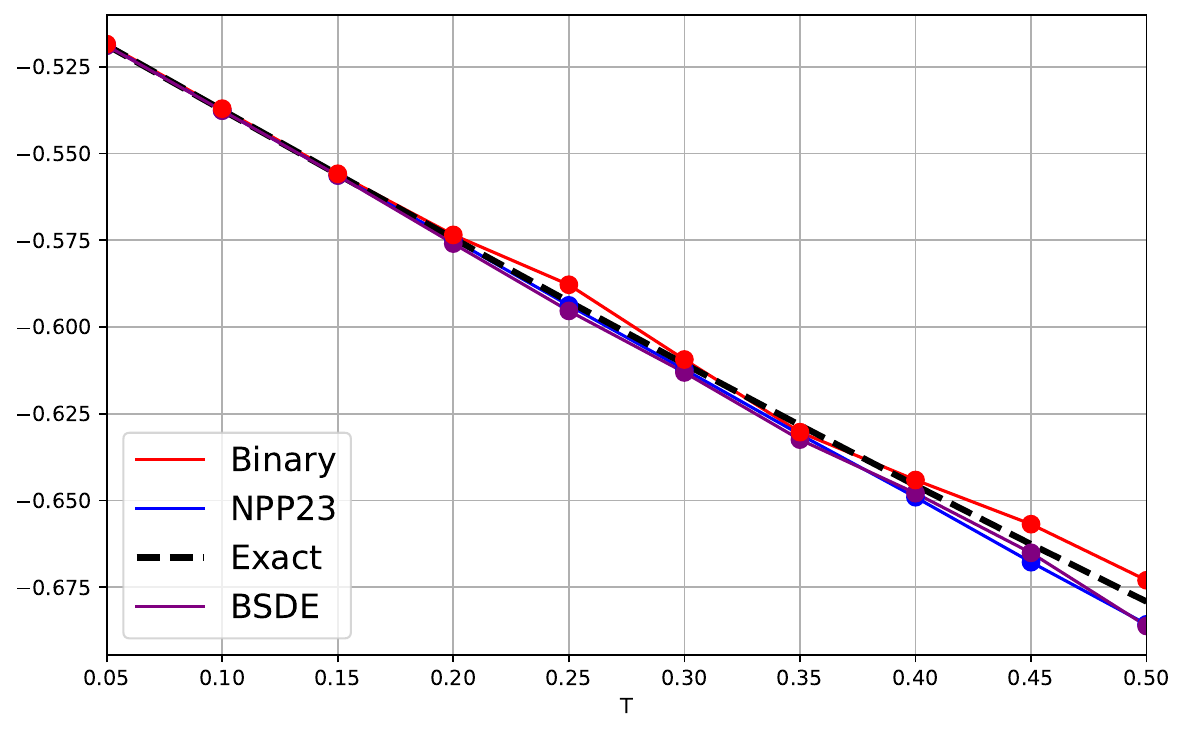}
\vskip-0.1cm
\caption{Dimension $d=10$.}
\end{subfigure}
\begin{subfigure}{.49\textwidth}
  \includegraphics[width=\textwidth]{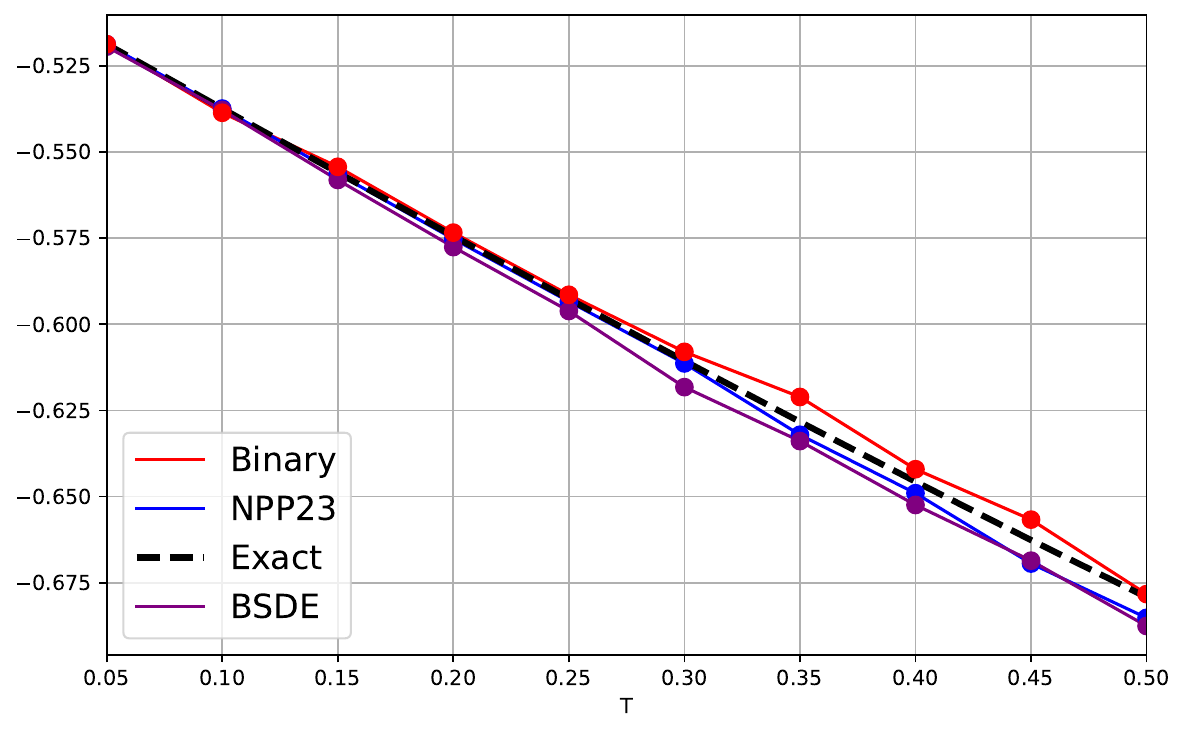}
\vskip-0.1cm
\caption{Dimension $d=100$.}
\end{subfigure}
\begin{subfigure}{.49\textwidth}
  \includegraphics[width=\textwidth]{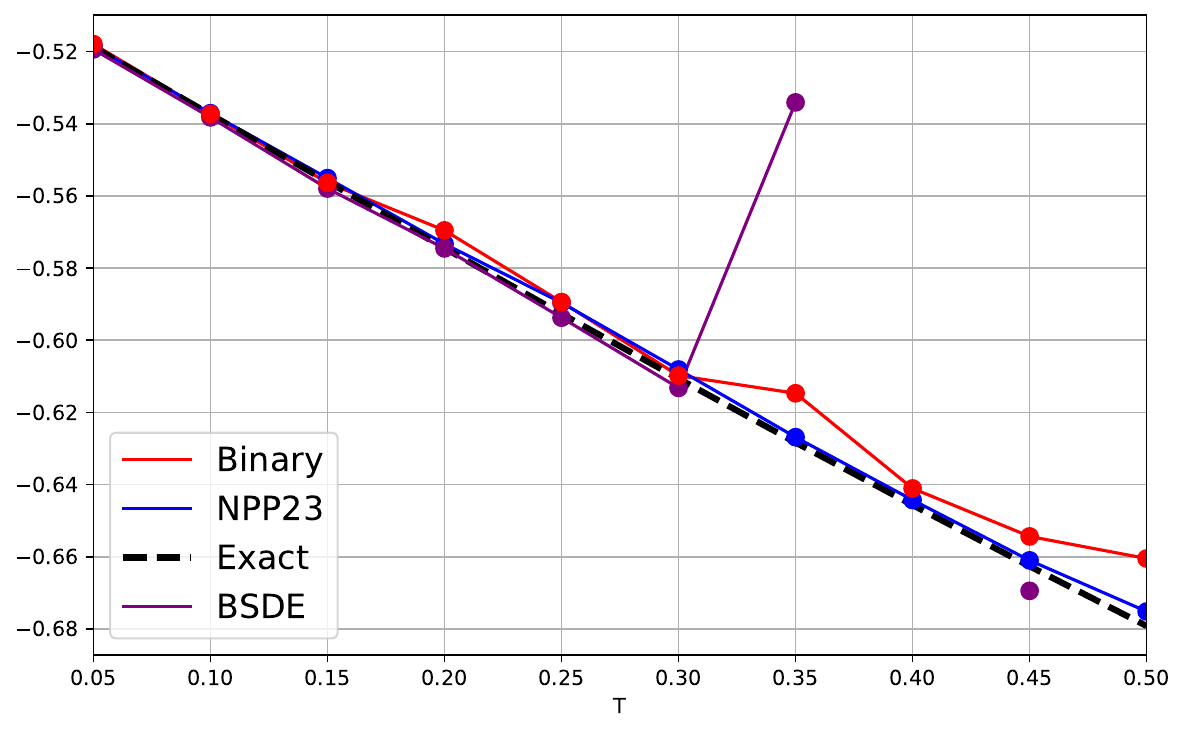}
\caption{Dimension $d=1000$.}
\end{subfigure}
\caption{Numerical solution $u(0,x)$ of \eqref{gl}.} %
\end{figure}

\vspace{-0.5cm}

\begin{table}[H]
  \centering
  \footnotesize
  \input{allen_cahn_summary.tex}
  \caption{Numerical estimates and computation times (s) for \eqref{gl} at $T=0.5$.}
\end{table}

\vspace{-0.3cm}

\begin{table}[H]
  \centering
  \footnotesize
  \input{allen_cahn_dim_1000_time_behavior_t.tex}
  \caption{Numerical estimates and computation times (s) for \eqref{gl} with $d=1000$.}
\end{table}

\item Consider the Allen-Cahn equation \eqref{gl}
  with the constant terminal condition
  $\phi (x) = \phi $ where $\phi >0$,
 i.e. 
\begin{numcases}{} 
  \label{gl-3}
  \displaystyle
    \frac{\partial u}{\partial t} u(t,x) + \frac{1}{2} \Delta_x u(t,x) + u(t,x) - u^3(t,x) = 0, 
    \medskip
    \\
    \nonumber
    \displaystyle
   u(T,x) 
   = \phi ,
 \end{numcases}
$(t,x) \in [0,T]\times \real^d$,
with the traveling wave solution
\begin{equation}
\nonumber %
 u(t,x) = \frac{1}{\sqrt{1-(1- \phi^{-2}) e^{-2(T-t)}}}, \qquad t\in [0,T]. 
\end{equation} 

\begin{figure}[H]
\centering
\begin{subfigure}{.49\textwidth}
  \includegraphics[width=\textwidth]{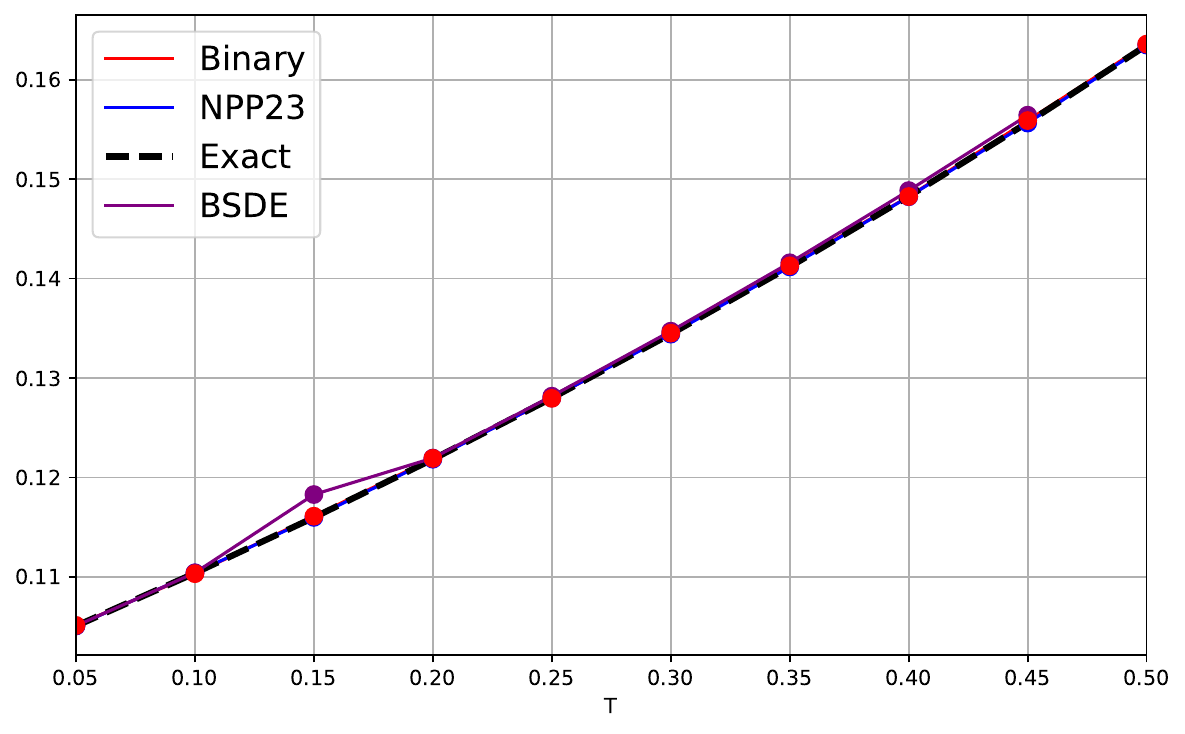}
\vskip-0.1cm
\caption{Dimension $d=1$.}
\end{subfigure}
\begin{subfigure}{.49\textwidth}
  \includegraphics[width=\textwidth]{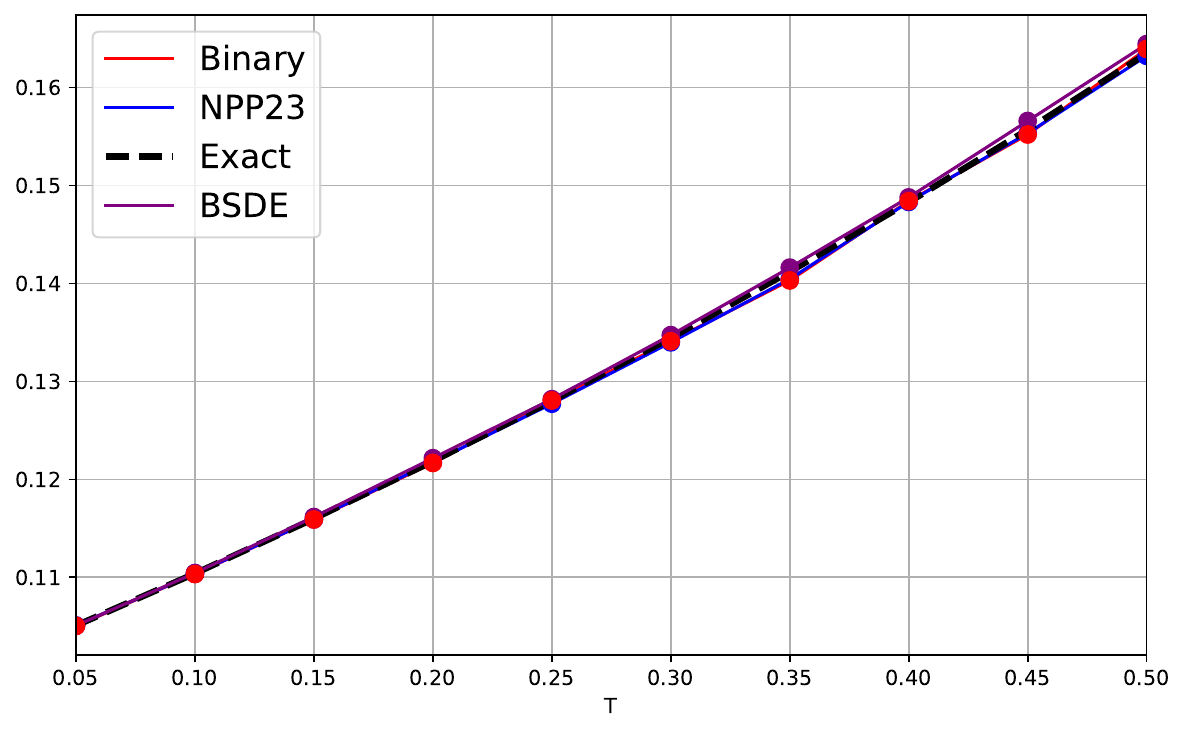}
\vskip-0.1cm
\caption{Dimension $d=10$.}
\end{subfigure}
\begin{subfigure}{.49\textwidth}
\includegraphics[width=\textwidth]{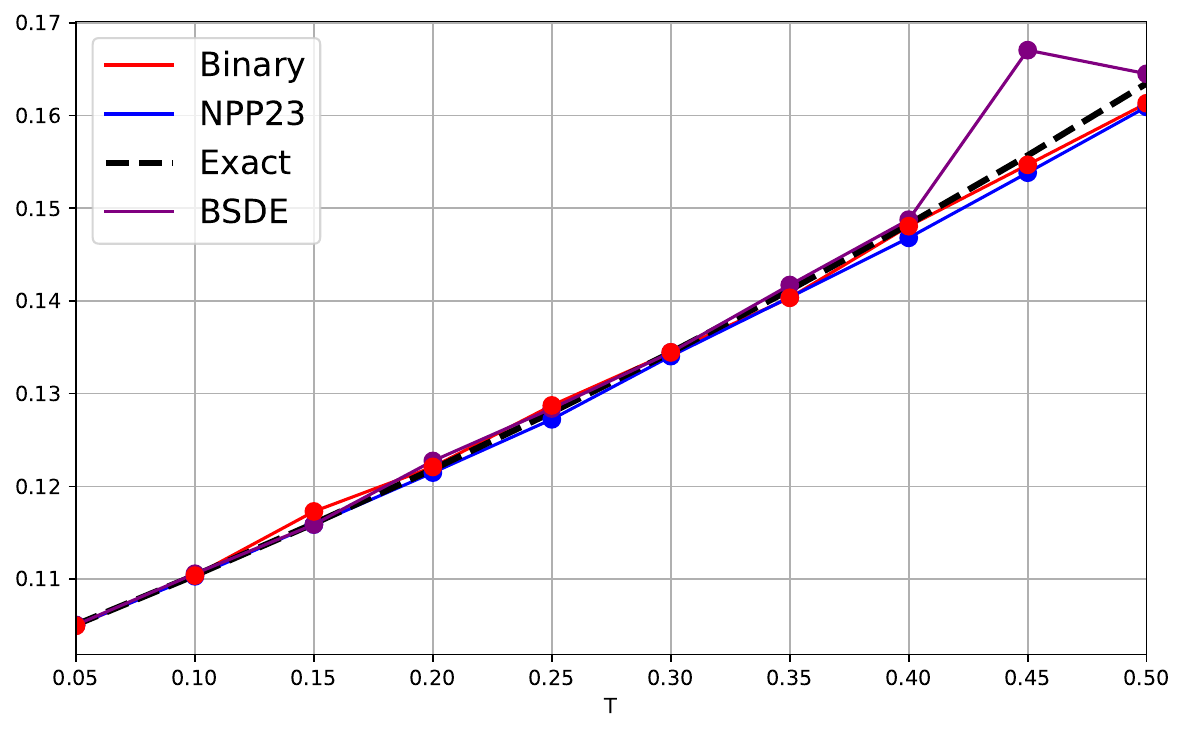}
\vskip-0.1cm
\caption{Dimension $d=100$.}
\end{subfigure}
\begin{subfigure}{.49\textwidth}
\includegraphics[width=\textwidth]{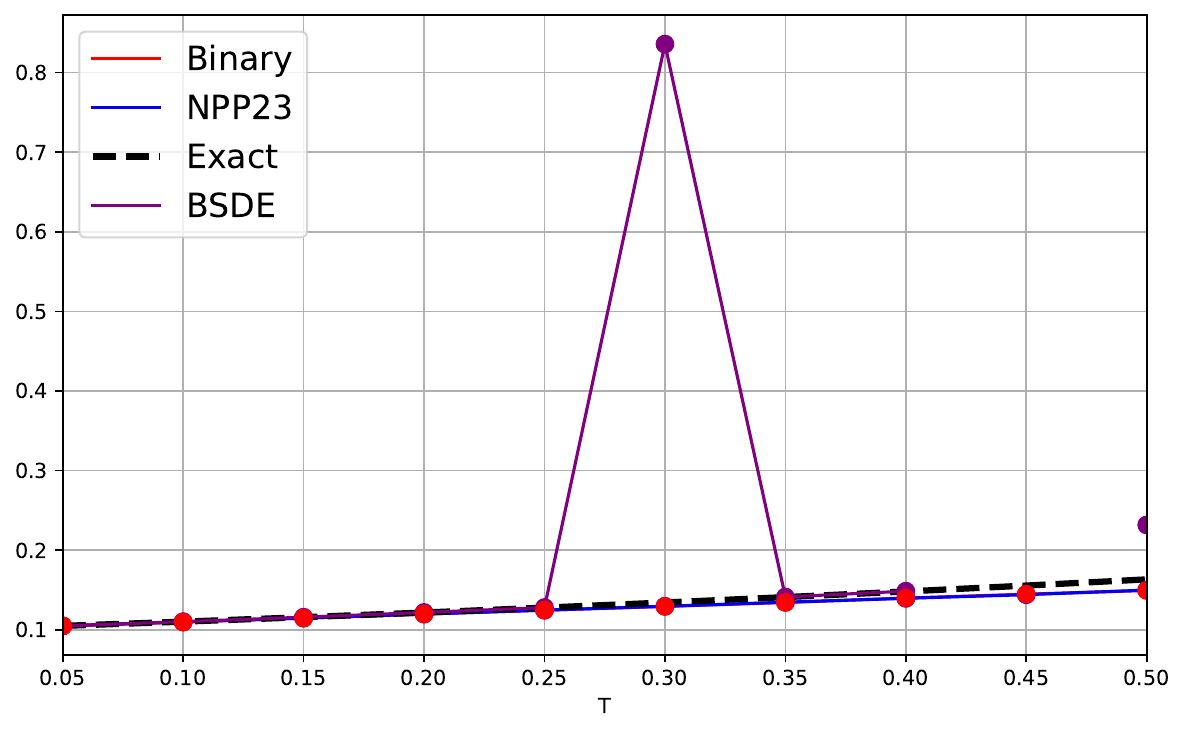}
\vskip-0.1cm
\caption{Dimension $d=1000$.}
\end{subfigure}
\caption{Numerical solution $u(0,x)$ of \eqref{gl-3}.} %
\label{fig1-03}
\end{figure}

\vspace{-0.5cm}

\begin{table}[H]
  \centering
  \footnotesize
  \input{allen_cahn_2_summary.tex}
  \caption{Numerical estimates and computation times (s) for \eqref{gl-3} at $T=0.5$.}
  \label{t11}
  \vspace{-0.6cm}
\end{table}

\vspace{-0.3cm}

\begin{table}[H]
  \centering
  \footnotesize
  \input{allen_cahn_2_dim_1000_time_behavior_t.tex}
  \caption{Numerical estimates and computation times (s) for \eqref{gl-3} with $d=1000$.}
\end{table}

\item
  In the framework of
  Proposition~\ref{fjld1a},
  consider the multidimensional version of 
  \eqref{b2} given by   
\begin{numcases}{} 
  \label{nonl0}
  \displaystyle
    \frac{\partial u}{\partial t} (t,x) + \frac{1}{2} \Delta u (t,x) + 4 e^{-u (t,x) }-10 e^{-u(t,x) / 2}+e^{u (t,x) / 2}-e^{u(t,x) } + 6 =0,
      \medskip
    \\
\nonumber 
    \displaystyle
   u(T,x) %
   = 2 \log \left(
   1
   +
   \frac{1}{1+
     e^{
       (x_1+\cdots + x_d) / \sqrt{d} }} \right),
\end{numcases}
$(t,x) \in [0,T] \times \real$,
 which admits the traveling wave solution
$$
 u(t, x)=
 2 \log \left(
   1
   +
   \frac{1}{1+
     e^{-(T-t) +
 (x_1+\cdots + x_d) / \sqrt{d} }} \right)
 \in (0, 2\log 2), \quad(t, x) \in[0, T] \times \mathbb{R}^d. 
$$
\vspace{-0.5cm}
\begin{figure}[H]
\centering
\begin{subfigure}{.49\textwidth}
  \includegraphics[width=\textwidth]{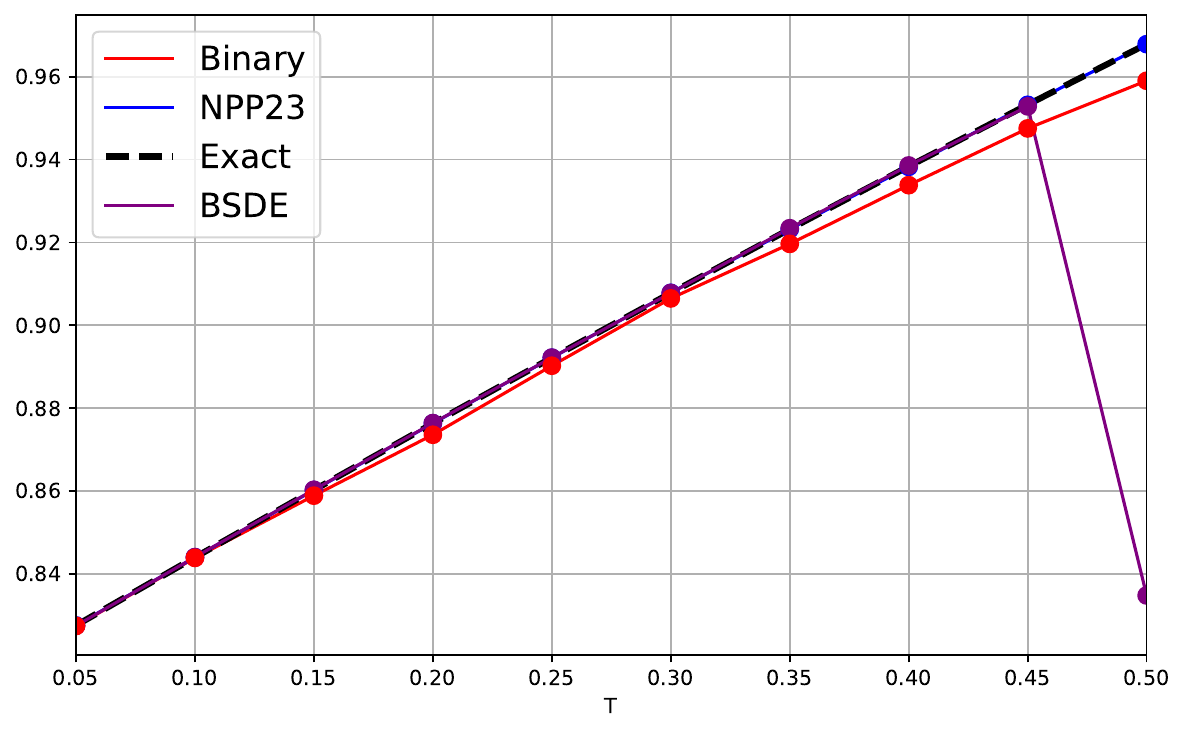}
\vskip-0.1cm
\caption{Dimension $d=1$.}
\end{subfigure}
\begin{subfigure}{.49\textwidth}
  \includegraphics[width=\textwidth]{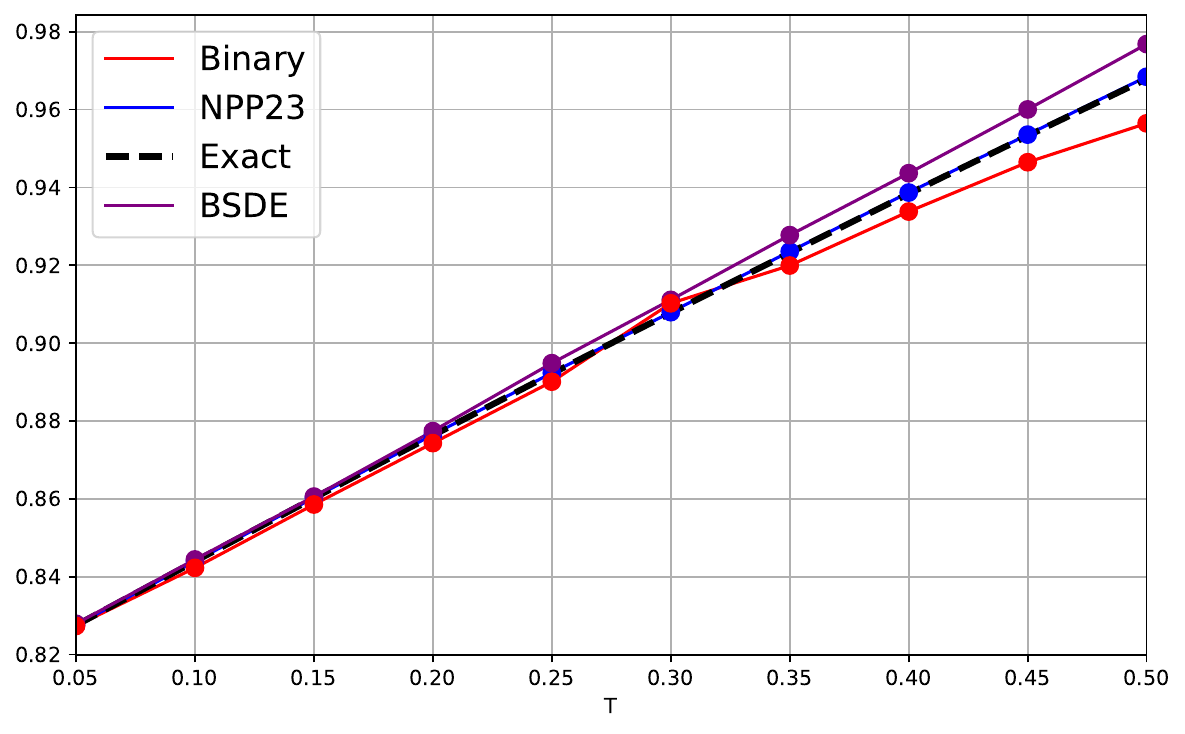}
\vskip-0.1cm
\caption{Dimension $d=10$.}
\end{subfigure}
\begin{subfigure}{.49\textwidth}
  \includegraphics[width=\textwidth]{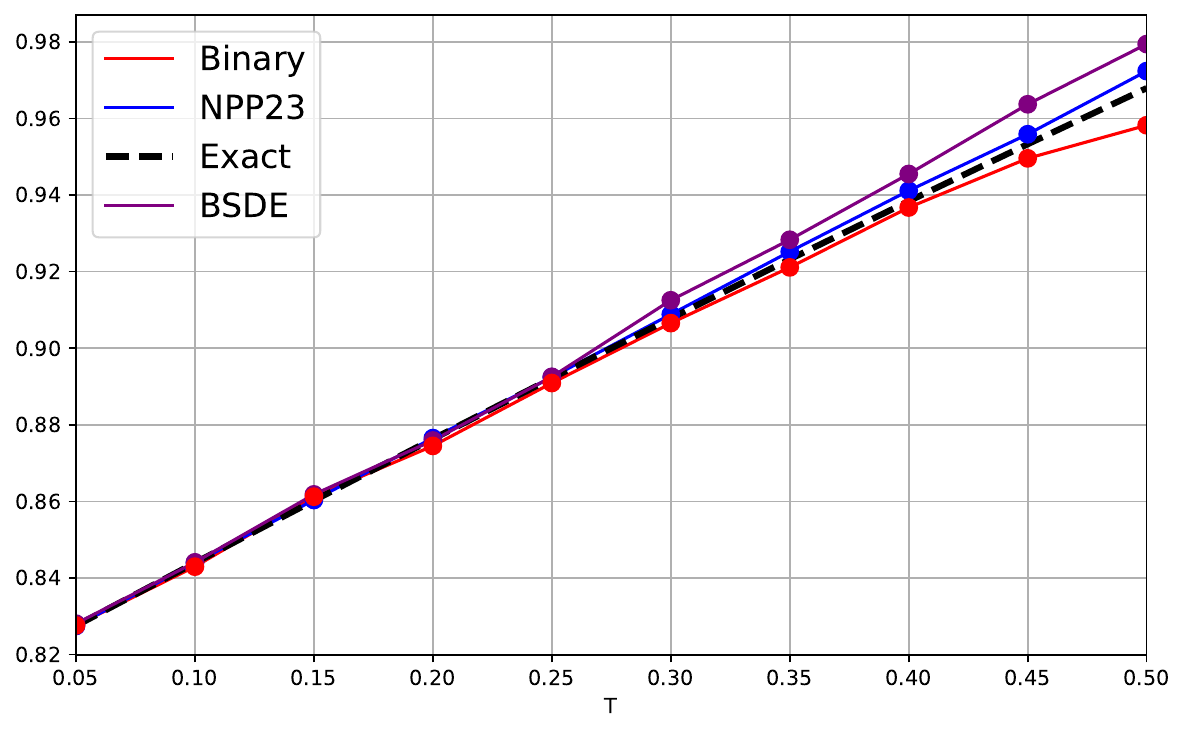}
\vskip-0.1cm
\caption{Dimension $d=100$.}
\end{subfigure}
\begin{subfigure}{.49\textwidth}
  \includegraphics[width=\textwidth]{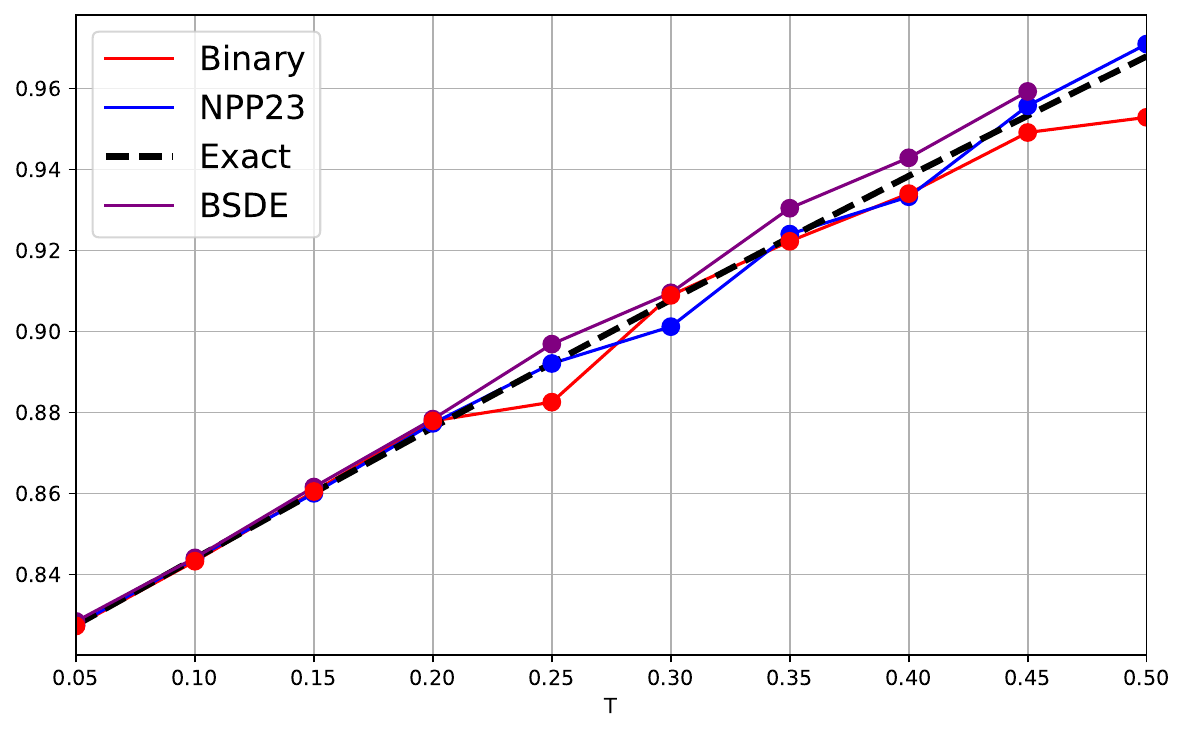}
\vskip-0.1cm
\caption{Dimension $d=1000$.}
\end{subfigure}
\caption{Numerical solution $u(0,x)$ of \eqref{nonl0}.}
\end{figure}

\vspace{-0.5cm}

\begin{table}[H]
  \centering
  \footnotesize
  \input{exp_example_summary.tex}
  \caption{Numerical estimates and computation times (s) for \eqref{nonl0} at $T=0.5$.}
\end{table}

\vspace{-0.3cm}

\begin{table}[H]
  \centering
  \footnotesize
  \input{exp_example_dim_1000_time_behavior_t.tex}
  \caption{Numerical estimates and computation times (s) for \eqref{nonl0} with $d=1000$.}
\end{table}

\end{enumerate}

\section{PDE system} 
\label{s4.1}
\subsection{Mild solutions}
\begin{definition}%
 A \emph{mild solution} of the PDE system~\eqref{pde-system} is 
 a family $\{u_c\}_{c\in \C (f)} \subset {\cal C}_b([0,T]\times \R^d)$
 satisfying the system of integral equations 
 \begin{equation}
   \label{pde-system-FK}
  u_c(t) = S(T-t) c(\phi ) + \sum_{z \in \mathcal{M}(c)} z_1 \int_t^T S(s-t) \left( \prod_{i=2}^{|z|} u_{z_i}(s) \right) ds, \quad t\in [0,T], \ c\in \C (f).
\end{equation}
\end{definition} 
\noindent
 The goal of this section is to prove the following result.
\begin{proposition} 
\label{mild-system}
Assume that  for all $c\in \C (f)$, 
\begin{itemize}
  \item $\mathcal{H}_{t,T}\big(\mathcal{X}_t^x(c)\big)$ is $L^1$-integrable uniformly in $(t,x)\in[0,T]\times\R^d$, 
    and
    \item $c(\phi )$ is bounded on $\real^d$, 
\end{itemize}
 and consider the family $\{u_c\}_{c\in\C (f)}$ of functions defined as 
   \begin{equation}
  \label{uc}
    u_c(t,x) := \E \big[ \mathcal{H}_{t,T}\big(\mathcal{X}_t^x(c)\big) \big], \quad (t,x)\in[0,T]\times\R^d, \ c\in \C (f). 
 \end{equation}
  Then, the following assertions hold.   
  \begin{enumerate}[i)] 
  \item
 The family $\{u_c\}_{c\in\C (f)}$ 
 yields a mild solution of the PDE system~\eqref{pde-system}. 
         \item
 Let 
           $
             \mathcal{B} \subset {\cal C}_b(\R^d)^{\C (f)}
           $
 be a set of bounded continuous functions such that: 
\begin{enumerate}[a)]
\item
  the PDE~\eqref{nl-heat} admits a classical solution
  $u\in {\cal C}^{1,\infty}([0,T]\times\R^d)$ that satisfies 
$$
 \{c(u)(t)\}_{c\in \C (f)} \in \mathcal{B},
 \quad t\in [0,T], 
$$
 and
  \item
  the family 
  $\{u_c(t)\}_{c\in\C (f)}$ in \eqref{uc}
  is the unique mild solution
  of the PDE system~\eqref{pde-system} that satisfies 
$$
 \{u_c(t)\}_{c\in \C (f)} \in \mathcal{B},
 \quad t\in [0,T]. 
$$
\end{enumerate}
 Then, we have $c(u) = u_c$ for all $c\in\C (f)$,
 and in particular,
 $u = u_{\id}$ is a classical solution of the PDE~\eqref{nl-heat}.
\end{enumerate}
\end{proposition}
\noindent
 Lemma~\ref{classical-sol} shows that
 if the PDE~\eqref{nl-heat} admits a classical solution $u\in {\cal C}^{1,\infty}([0,T]\times \R^d)$, then $\{c(u)\}_{c\in\C (f)}$ is also a mild solution
 of the PDE system~\eqref{pde-system}. 
 However, if $u$ is only a mild solution of the PDE~\eqref{nl-heat}
 then $\{c(u)\}_{c\in\C (f)}$ is not necessarily a mild solution of \eqref{pde-system}.
\begin{lemma}
\label{classical-sol}
 Assume that the PDE~\eqref{nl-heat} admits
 a classical solution $u\in {\cal C}^{1,\infty}([0,T]\times \R^d)$.
 Then, the family $\{c(u)\}_{c\in \C (f)}$
 is a classical solution of the PDE system~\eqref{pde-system}. 
\end{lemma}
\begin{Proof}
 For any test function $g\in {\cal C}^\infty(\R )$   
 we apply the operator $\displaystyle
 \frac{\partial}{\partial t} +\frac{1}{2} \Delta$
 to $g(u)$, and get
\begin{align*}
    \left( \frac{\partial}{\partial t} +\frac{1}{2} \Delta \right) g(u) &= g'(u) \left( \frac{\partial u}{\partial t} + \frac{1}{2} \Delta u \right) + \frac{1}{2} \sum_{i=1}^d \pt_i [ g'(u) ] \pt_i u \\
    &= - g'(u) f(u) + \frac{1}{2} \sum_{i=1}^d \pt_i [ g'(u) ] \pt_i u,
\end{align*}
 and thus, by the Leibniz formula, for all $\alpha\in\N_0^d$ we have 
\begin{align*}
  \left( \frac{\partial}{\partial t} +\frac{1}{2} \Delta \right) \pt^\alpha [g(u)] & = - \sum_{\mathbf{0}\leq \beta\leq \alpha} \binom{\alpha}{\beta} \pt^\beta [g'(u)] \pt^{\alpha-\beta} [f(u)]
  \\
  & \quad + \frac{1}{2} \sum_{i=1}^d \sum_{\mathbf{0}\leq \beta\leq \alpha} \binom{\alpha}{\beta} \pt^{\beta+\ind_i} [g'(u)] \pt^{\alpha-\beta+\ind_i} u,
\end{align*}
 which is \eqref{pde-system}
 for $c = {\alpha!}^{-1} \pt^\alpha \circ g^{(j)}$. 
 Similarly, we apply the PDE~\eqref{nl-heat} to derive
\begin{equation*}%
  \left( \frac{\partial}{\partial t} +\frac{1}{2} \Delta \right) \pt^\alpha u = \pt^\alpha \left( \frac{\partial}{\partial t} +\frac{1}{2} \Delta \right) u = -\pt^\alpha [f(u)],
  \quad
 \alpha\in\N_0^d, 
\end{equation*}
 which is \eqref{pde-system}
    for $c = {\alpha!}^{-1} \pt^\alpha$.
\end{Proof}
 \noindent 
\begin{lemma}
\label{lemma-uc}
Assume that $\mathcal{H}_{t,T}\big(
\mathcal{X}_t^x(c)\big)$
is $L^1$-integrable for all $(t,x)\in[0,T]\times\R^d$ and 
$c\in \C (f)$.
  Then, the family $\{u_c\}_{c\in\C (f)}$
  of functions defined in \eqref{uc}
  satisfies the recursive relation
\begin{equation}\label{rec-uc}
\begin{split}
 u_c(t,x) &= \E \left[ c(\phi ) \big(X_T^\varnothing\big)+ \int_t^T \sum_{z \in \mathcal{M}(c)} z_1 \prod_{i=2}^{|z|} u_{z_i}(s, X^\varnothing_s) ds \right], \quad (t,x)\in[0,T]\times\R^d, 
\end{split}
\end{equation}
 for every $c\in \C (f)$. 
\end{lemma}
\begin{proof}
 From \eqref{functional}, we have 
\begin{align}
\nonumber %
    \mathcal{H}_{t,T}\big(\mathcal{X}_t^x(c)\big) & = \mathcal{H}_{t,T}\big(\mathcal{X}_t^x(c)\big) \big( \ind_{\{T_\varnothing>T\}} + \ind_{\{T_\varnothing\leq T\}} \big)
    \\
    \nonumber
    & = \frac{c(\phi ) \big(X_T^\varnothing\big)}{\bar\rho(T-T_{\varnothing^{-}})} \ind_{\{T_\varnothing>T\}}
    \\
    \nonumber
    & \quad
    + \ind_{\{T_\varnothing\leq T\}} \frac{I_1(c)}{\rho(\tau_\varnothing) q_c(I(c))} \prod_{{\mathbf{k}} \in \mathcal{K}^{\rm b}_{t,T}(c)} c^{{\mathbf{k}}}(\phi ) \big(X_T^{{\mathbf{k}}}\big) \prod_{{\mathbf{k}} \in \mathcal{K}^\circ_{t,T}(c) \setminus \{\varnothing\}} \frac{I_1(c^{{\mathbf{k}}})}{\rho(\tau_{\mathbf{k}}) q_{c^{{\mathbf{k}}}}(I(c^{{\mathbf{k}}}))}
        \nonumber
\\
\nonumber
& = \frac{c(\phi ) \big(X_T^\varnothing\big)}{\bar\rho(T-t)} \ind_{\{T_\varnothing>T\}} + \ind_{\{T_\varnothing\leq T\}} \frac{I_1(c)}{\rho(\tau_\varnothing) q_c(I(c))} \prod_{i=2}^{|I(c)|} \mathcal{H}_{T_\varnothing , T}\Big(\mathcal{X}_{T_\varnothing}^{X_{T_\varnothing}^\varnothing} ( I_i(c) )\Big), 
\end{align} 
 hence %
\begin{align*}
  &  u_c(t,x) = \E \left[ \mathcal{H}_{t,T}\big(\mathcal{X}_t^x(c)\big) \right]
    \\
    & = \E \left[ \frac{c(\phi )\big(X_T^\varnothing\big)}{\bar\rho(T-t)} \ind_{\{\tau_\varnothing>T-t\}} + \ind_{\{\tau_\varnothing\leq T-t\}} \frac{I_1(c)}{\rho(\tau_\varnothing) q_c(I(c))} \prod_{i=2}^{|I(c)|} \mathcal{H}_{t+\tau_\varnothing ,T}
      \Big(
      \mathcal{X}_{t+\tau_\varnothing}^{X_{t+\tau_\varnothing}^\varnothing} ( I_i(c))
      \Big) \right]
    \\
    & = \E \left[
      c(\phi )\big(X_T^\varnothing\big)
     + \ind_{\{\tau_\varnothing\leq T-t\}} \frac{I_1(c)}{\rho(\tau_\varnothing) q_c(I(c))} \prod_{i=2}^{|I(c)|}
      \E \left[
        \mathcal{H}_{t+\tau_\varnothing ,T}
      \Big(
      \mathcal{X}_{t+\tau_\varnothing}^{X_{t+\tau_\varnothing}^\varnothing} (I_i(c))
      \Big)
      \ \! \Big| \ \!
      \tau_\varnothing,
      \
      X_{t+\tau_\varnothing}^\varnothing
      ,
      \
      I_i(c)
      \right]
      \right]
    \\
    & = \E \left[
      c(\phi )\big(X_T^\varnothing\big)
     + \ind_{\{\tau_\varnothing\leq T-t\}} \frac{I_1(c)}{\rho(\tau_\varnothing) q_c(I(c))} \prod_{i=2}^{|I(c)|}
     u_{I_i(c)}
     \big( t+\tau_\varnothing
     ,
     X_{t+\tau_\varnothing}^\varnothing
     \big)
      \right]
    \\
    & = \E \left[ c(\phi )\big(X_T^\varnothing\big)
          + 
          \int_0^{T-t} \frac{I_1(c)}{q_c(I(c))} \prod_{i=2}^{|I(c)|}
          u_{I_i(c)}\big(
          t+s ,X^\varnothing_{t+s} \big) ds \right]
    \\
    & = \E \left[ c(\phi )\big(X_T^\varnothing\big)
          + 
            \int_t^T
            \sum_{z \in \mathcal{M}(c)} z_1
            \prod_{i=2}^{|z|}
            u_{z_i} \big(s , X^\varnothing_s \big) ds \right]
    .
\end{align*}
\end{proof}
 
\begin{Proofy} \hskip-0.14cm {\em of Proposition~\ref{mild-system}.} 
  $i)$ By assumption, the functions
  $\{u_c\}_{c\in \C (f)}$ defined in
  \eqref{uc} 
  are bounded for all $c\in \C (f)$,
  and they satisfy \eqref{rec-uc} by Lemma~\ref{lemma-uc}. 
 In addition, \eqref{rec-uc} is equivalent to the
 definition~\eqref{pde-system-FK} of mild solutions
 after exchange of expectation and integral
 by dominated convergence, since
 $X^\varnothing$ is a $d$-dimensional
 standard Brownian motion starting from $x\in \R^d$. 
 Finally, we note that, by dominated convergence, 
 for any bounded measurable function $F: [0,T]\times \R^d \to \R$, the map
\begin{equation*}
 (s,x)\mapsto \E\big[ F\big(s, X_s^\varnothing\big) \mid X_0^\varnothing = x
    \big]=
     \int_{\R^d} F(s,y) p_0(s,x-y) dy
\end{equation*}
is continuous on $[t,T] \times \R^d$, 
 since $p_0(s,x-y)$ is continuous in $(s,x,y)$ and integrable in $y$.
 Hence, by \eqref{rec-uc} the function $u_c$ defined in \eqref{uc}
 belongs to ${\cal C}_b([0,T]\times\R^d)$ for all $c\in \C (f)$.

 \smallskip
  
\noindent
$ii)$ follows from $i)$ and Lemma~\ref{classical-sol}.
\end{Proofy}
\noindent
\subsection{Uniqueness of PDE system solution}
\subsubsection{Probabilistic approach}
\noindent
In Proposition~\ref{unf-intg}
we obtain the uniqueness and probabilistic representation
of the solution of the PDE system~\eqref{pde-system} under 
 uniform integrability assumptions on random functionals. 
\begin{proposition}\label{unf-intg}
  Assume that
  \begin{enumerate}[i)]
  \item
    the PDE system~\eqref{pde-system} has a mild solution $\{u_c\}_{c\in\C (f)} \subset {\cal C}_b([0,T]\times \R^d)$, %
    and that
  \item
    for each $(t,x)\in[0,T]\times\R^d$ and $c\in\C (f)$, the sequence of random variables %
\begin{equation*}
  \mathcal{H}_{t,T,n}^x(c) := \prod_{{\mathbf{k}} \in \cup_{i=0}^n \mathcal{K}^{{\rm b},i}_{t,T}(c)} \frac{c^{{\mathbf{k}}}(\phi ) \big(X_T^{{\mathbf{k}}} \big) }{\bar\rho(T-T_{{\mathbf{k}}^{-}})} \prod_{{\mathbf{k}} \in \cup_{i=0}^n \mathcal{K}^{\circ,i}_{t,T}(c)} \frac{I_1( c^{{\mathbf{k}}})}{\rho (\tau_{\mathbf{k}}) q_{c^{{\mathbf{k}}}}
    (I( c^{{\mathbf{k}}}) )} \prod_{{\mathbf{k}} \in \mathcal{K}^{n+1}_{t,T}(c)}
  \hskip-0.3cm
  u_{c^{{\mathbf{k}}}} \big(T_{{\mathbf{k}}^{-}}, X_{T_{{\mathbf{k}}^{-}}}^{{\mathbf{k}}}\big) 
  \end{equation*}
 $n\geq 0$, is uniformly integrable.
  \end{enumerate}
   Then, $\{u_c\}_{c\in\C (f)}$ can be represented as 
 $$
  u_c(t,x) := \E \left[ \mathcal{H}_{t,T}(\mathcal{X}_t^x(c)) \right],
  \quad
   (t,x)\in[0,T]\times\R^d, \ \ c\in\C (f).
    $$ 
In particular, uniqueness of mild solutions holds for the PDE system~\eqref{pde-system}.
\end{proposition}

\begin{proof}
  By the definition~\eqref{pde-system-FK} of mild solutions and recalling that
  $( S(t) )_{t\geq 0}$
  is the heat semigroup generated by the Brownian motion $X^\varnothing$,
  we have the Feynman--Kac representation 
\begin{equation}\label{eqn-1}
  \begin{split}
    u_c(t,x) &= \E \left[ u_c \big(T, X_T^\varnothing\big) + \int_t^T \sum_{z \in \mathcal{M}(c)} z_1 \prod_{i=2}^{|z|} u_{z_i}(s, X^\varnothing_s) ds \right] \\
    &= \E \left[ \frac{c(\phi ) \big(X_T^\varnothing\big)}{\bar\rho(T-t)} \ind_{\{ T_\varnothing>T \}} + \ind_{\{T_\varnothing\leq T \}}\frac{I_1(c)}{\rho(\tau_\varnothing) q_c(I(c))} \prod_{i=2}^{|I(c)|} u_{I_i(c)}(T_\varnothing,X_{T_\varnothing}^\varnothing ) \right],
  \end{split}
\end{equation}
where in the first equality we used the dominated convergence theorem to exchange expectation and integral,
and from which it follows that 
 $$
u_c (t,x) = \E[\mathcal{H}_{t,T,0}^x(c)],   
  \quad
   (t,x)\in[0,T]\times\R^d, \ \ c\in\C (f).
$$ 
 Next, since each offspring has the same dynamic as the parent branch, we can repeat the above calculations to write for
 the branch $k \in \mathcal{K}^1_{t,T}(c)$ with code $c^k \in\C (f)$,
\begin{align*}
  u_{c^k}(T_\varnothing,X_{T_\varnothing}^\varnothing ) & = u_{c^k}
  \big( T_{k-},X_{T_{k-}}^k \big)
  \\
   & = \E \left[ \frac{c^k (\phi ) \big(X_T^k \big)}{\bar\rho (T-T_{k-})} \ind_{\{ T_k>T \}} + \ind_{\{T_k\leq T \}} \frac{I_1(c^k)}{\rho\big(\tau^k\big) q_{c^k}(I(c^k))} \prod_{i=2}^{|I( c^k)|} u_{I_i(c^k)}(T_k,X_{T_k}^k ) \Bigg| \F_{T_\varnothing}^{X^\varnothing} \right].
\end{align*}
 Plugging the above identity back into \eqref{eqn-1}, we obtain
 $$
u_c (t,x) = \E[\mathcal{H}_{t,T,1}^x(c)], 
  \quad
   (t,x)\in[0,T]\times\R^d, \ \ c\in\C (f).
$$ 
 by conditional
 independence of branches in $\mathcal{K}^1_{t,T}(c)$ given $\F_{T_\varnothing}^{X^\varnothing}$.
 By iterating the above argument we can show similarly
 that for any $n\in \N_0$ we have
 $$
u_c (t,x) = \E[\mathcal{H}_{t,T,n}^x(c)], 
  \quad
   (t,x)\in[0,T]\times\R^d, \ \ c\in\C (f).
$$ 
  Using the fact that
  $\mathcal{H}_{t,T,n}^x(c)$ converges to $\mathcal{H}_{t,T}(
  \mathcal{X}_t^x(c))$ almost surely
  and the uniform integrability of $(\mathcal{H}_{t,T,n}^x(c))_{n\geq 0}$, we get the desired result.
\end{proof}

\begin{remark}
  Condition~$i)$ in Proposition~\ref{unf-intg}
  can be strengthened
  by assuming that the PDE system~\eqref{pde-system}
  admits a classical solution $\{u_c\}_{c\in\C (f)} \subset {\cal C}^{1,2}_b([0,T]\times \R^d)$. In this case,
  the application of It\^o's formula to
  $\big\{ u_c \big(s, X_s^\varnothing\big) \big\}_{s\geq t}$
  together with a localization argument
  (cf. \cite[Theorem 4.4.2]{Karatzas-Shreve-1991})
  allows us to recover
  the Feynman--Kac representation~\eqref{eqn-1},
  which yields uniqueness of classical solutions of \eqref{pde-system}.
\end{remark}
\subsubsection{Analytical approach}
\noindent
 We close this section with a mild uniqueness result
 for the PDE system~\eqref{pde-system} in a weighted function space,
 which has been used in the proof of Theorem~\ref{thm-1}.
\begin{proposition}
  \label{uniqueness-system}
  For any $C\geq 1$,
  the PDE system~\eqref{pde-system} has at most one mild solution in the
  weighted function space 
  \begin{equation}
\nonumber %
  \ell^\infty_w({\cal C}_b(\R^d)): =
  \left\{
  \{ v_{\alpha,j} \}_{( \alpha , j ) \in \N_0^d \times \N_{-1} } \subset {\cal C}_b(\R^d)
  \ : \
  \sup_{(\alpha , j ) \in
    \N_0^d \times \N_{-1}
    } w_\alpha \|v_{\alpha,j}\|_\infty < \infty
  \right\}
\end{equation}
  with weight function $w_\alpha := C^{-|\alpha|}$,
  $\alpha \in\N_0^d$.
\end{proposition}
\begin{proof}
  \noindent
  $i)$
  In what follows, using the notation \eqref{u-indexed} 
 we rewrite the PDE system~\eqref{pde-system} as 
\begin{equation*}
  \left\{
  \begin{aligned}
    & \hskip-0.1cm
    \left( \frac{\partial}{\partial t} + \frac{1}{2} \Delta \right) u_{\alpha,-1} + u_{\alpha,0} =0, 
    \medskip
    \\
    &
    u_{\alpha,-1}(T) = \frac{1}{\alpha!} \pt^\alpha \phi ,
    \medskip
    \\
    & \hskip-0.1cm
    \left( \frac{\partial}{\partial t} + \frac{1}{2} \Delta \right) u_{\alpha,j} + \hskip-0.1cm
    \sum_{\mathbf{0}\leq \beta\leq \alpha} u_{\alpha-\beta,0} u_{\beta,j+1} -
    \sum_{i=0}^d \sum_{\mathbf{0}\leq \beta\leq \alpha}
    \hskip-0.1cm
    \frac{( 1 + \beta_i)( 1 + \alpha_i-\beta_i )}{2}
     u_{\alpha-\beta+\ind_i,-1} u_{\beta+\ind_i,j+1} =0,
        \medskip
    \\
    &  u_{\alpha,j}(T) = \frac{1}{\alpha!} \pt^\alpha [f^{(j)}(\phi )], \quad
    j\geq 0. 
  \end{aligned}
  \right.
\end{equation*}
 Similarly, letting  
$
  U := ( u_{\alpha,j})_{(\alpha ,j)\in \N^d \times \N_{-1}}$, 
 we rewrite \eqref{pde-system-FK} as 
\begin{equation*}
  U(t) = S(T-t) U(T) + \int_t^T S(s-t) F(U(s)) ds, \quad t\in [0,T], 
\end{equation*}
 where 
\begin{align}
  \label{pde-system-nonlinearity}
  & F(U) %
= \bigg( u_{\alpha,0}, \sum_{\mathbf{0}\leq \beta\leq \alpha} u_{\alpha-\beta,0} u_{\beta,1} - \frac{1}{2} \sum_{i=0}^d \sum_{\mathbf{0}\leq \beta\leq \alpha} ( 1 + \beta_i )( 1 + \alpha_i-\beta_i ) u_{\alpha-\beta+\ind_i,-1} u_{\beta+\ind_i,1}, \ldots
  \\
  \nonumber
  &\qquad
  \quad
  \qquad %
  \hfill \sum_{\mathbf{0}\leq \beta\leq \alpha} u_{\alpha-\beta,0} u_{\beta,j+1} - \frac{1}{2} \sum_{i=0}^d \sum_{\mathbf{0}\leq \beta\leq \alpha} ( 1 + \beta_i )( 1 + \alpha_i-\beta_i ) u_{\alpha-\beta+\ind_i,-1} u_{\beta+\ind_i,j+1}, \ldots \bigg). 
\end{align} 
\noindent
 $ii)$
 Suppose %
 that for some sufficiently small $T>0$ there exist two mild solutions
 $U(t)$, $\widetilde{U}(t) \in \ell^\infty_w ({\cal C}_b(\R^d))$, $t\in[0,T]$, 
 to the PDE system~\eqref{pde-system}. 
 Without loss of generality, we assume that both $U$ and $\widetilde{U}$ lie in the unit ball $B(0,1)$ of $\ell^\infty_w ({\cal C}_b(\R^d))$, %
 i.e. 
\begin{equation*}
  \|u_{\alpha,j}(t)\|_\infty \leq C^{|\alpha|}, \quad
   \|\tilde{u}_{\alpha,j}(t)\|_\infty \leq C^{|\alpha|}, \quad \alpha\ge\mathbf{0}, \ j\geq -1, \ t\in[0,T].
\end{equation*}
 The operator norm of the Jacobian 
\begin{equation*}
  DF(V) = \left( \frac{\pt F_{\alpha,j}}{\pt u_{\beta,i}} (V) \right)_{
      \alpha , \beta\ge\mathbf{0},
 \atop 
 i,j\geq -1
   } 
\end{equation*}
of $F$ in \eqref{pde-system-nonlinearity}
as a mapping from $\ell^\infty_w({\cal C}_b(\R^d))$ to $\ell^\infty_{w'}({\cal C}_b(\R^d))$ is given by 
 \begin{equation}
   \label{eqn-25}
  \|DF(V)\|_{{\cal L} ( \ell^\infty_w({\cal C}_b(\R^d)) , \ell^\infty_{w'}({\cal C}_b(\R^d)))} =
  \sup_{\alpha\ge\mathbf{0} \atop j\ge-1}
  \left(
  w'_\alpha \sum_{\beta\ge\mathbf{0} \atop i\ge-1}
  \frac{1}{w_\beta }
  \left\| \frac{\pt F_{\alpha,j}}{\pt u_{\beta,i}}(V) \right\|_\infty
  \right),
\end{equation}
see \cite[Section 5.2]{Mey23}.
When $j=-1$ we have 
\begin{equation*}
  \frac{\pt F_{\alpha,-1}}{\pt u_{\beta,i}} =
 \mathbf{1}_{\{\alpha = \beta \}}
 \mathbf{1}_{\{i = 0 \}}, 
\end{equation*}
 hence the term inside the supremum in \eqref{eqn-25}
 is
\begin{equation*}
  w'_\alpha \sum_{\beta \geq {\bf 0}\atop i\ge-1}
  \mathbf{1}_{\{\alpha = \beta \}} 
  \mathbf{1}_{\{i = 0 \}}
  w_\beta^{-1}
  = \frac{w'_\alpha}{w_\alpha },
\end{equation*}
 while when $j\geq 0$ we have 
\begin{equation*}
  \frac{\pt F_{\alpha,j}}{\pt u_{\beta,i}} =
  \hskip-0.1cm
  \begin{cases}
  \displaystyle
  -\frac{1}{2} \sum_{k=1}^d  \beta_k (2+\alpha_k-\beta_k) u_{\alpha-\beta+2\ind_k,j+1} \ind_{\{\ind_k \leq \beta\leq \alpha+\ind_k \}}, & \!\!\! i=-1,
  \medskip
  \\
  \displaystyle
  u_{\alpha-\beta,j+1} \ind_{\{\mathbf{0}\leq \beta\leq \alpha\}}, & \!\!\! i=0,
  \medskip
    \\
  \displaystyle
  \left( u_{\alpha-\beta,0} \ind_{\{\mathbf{0}\leq \beta\leq \alpha\}} - \frac{1}{2} \sum_{k=1}^d  \beta_k (2 + \alpha_k -\beta_k ) u_{\alpha-\beta+2\ind_k ,-1} \ind_{\{ \ind_k \leq \beta\leq \alpha+\ind_k \}} \right)
  \mathbf{1}_{\{ i = j+1 \}}, 
  & \!\!\! i\geq 1, 
  \end{cases}
\end{equation*}
 thus the term inside the supremum in \eqref{eqn-25} 
 is
\begin{align*}
  & w'_\alpha \sum_{\beta\ge\mathbf{0}}
  \frac{1}{w_\beta }
  \left( \frac{1}{2} \sum_{k=1}^d  \beta_k (2 + \alpha_k-\beta_k) \|v_{\alpha-\beta+2\ind_k,j+1}\|_\infty \ind_{\{\ind_k\leq \beta\leq \alpha+\ind_k\}} + \|v_{\alpha-\beta,j+1}\|_\infty \ind_{\{\mathbf{0}\leq \beta\leq \alpha\}} \right) 
  \\
  &\quad +\ w'_\alpha \sum_{\beta\ge\mathbf{0}}
  \frac{1}{w_\beta }
  \left\| v_{\alpha-\beta,0} \ind_{\{\mathbf{0}\leq \beta\leq \alpha\}} -
  \frac{1}{2} \sum_{k=1}^d  \beta_k (2 + \alpha_k -\beta_k ) v_{\alpha-\beta+2\ind_k ,-1} \ind_{\{\ind_k \leq \beta\leq \alpha+\ind_k \}} \right\|_\infty .
  \end{align*}
It follows from \eqref{eqn-25} that for
 the family of weights 
\begin{equation*}
 w_\alpha^{(n)} := \frac{1}{C^{|\alpha|} |\alpha|^{(d+3)n}}, \quad n\geq 0, 
\end{equation*}
 which satisfies $w^{(0)}=w$ and  
 $\ell^\infty_w ({\cal C}_b(\R^d)) \subset \ell^\infty_{w^{(n)}}({\cal C}_b(\R^d))$,
 for all $n\geq 0$ and $V\in B(0,1)$ we have 
\begin{align*}
  & \|DF(V)\|_{{\cal L}(\ell^\infty_{w^{(n)}}({\cal C}_b(\R^d)) , \ell^\infty_{w^{(n+1)}}({\cal C}_b(\R^d)))}
  \\
  & \leq \sup_{\alpha\in\N_0^d} \frac{1}{C^{|\alpha|} |\alpha|^{(d+3)(n+1)}}
  \max \Big( C^{|\alpha|} |\alpha|^{(d+3)n} ,
 \\[-4pt]
& \qquad \ \! \qquad %
 \left.
\sum_{k=1}^d  \sum_{\ind_k \leq \beta\leq \alpha+\ind_k} \beta_k (2 +\alpha_k -\beta_k ) C^{|\alpha|-|\beta|+2} C^{|\beta|} |\beta|^{(d+3)n}
+ 2 \sum_{\mathbf{0}\leq \beta\leq \alpha} C^{|\alpha|-|\beta|} C^{|\beta|} |\beta|^{(d+3)n} \right)
  \\
  & =  \sup_{\alpha\in\N_0^d} \left( \frac{C^2}{|\alpha|^{(d+3)(n+1)}} \sum_{k=1}^d  \sum_{\ind_k \leq \beta\leq \alpha+\ind_k} |\beta|^{(d+3)n} \beta_k (2 + \alpha_k-\beta_k) + \frac{2}{|\alpha|^{(d+3)(n+1)}} \sum_{\mathbf{0}\leq \beta\leq \alpha} |\beta|^{(d+3)n} \right)
  \\
   & \leq \sup_{\alpha\in\N_0^d} \left( \frac{C^2}{|\alpha|^{(d+3)(n+1)}} \sum_{k=1}^d \frac{(2 + \alpha_k)^2}{4} \sum_{\ind_k \leq \beta\leq \alpha+\ind_k} |\beta|^{(d+3)n} + \frac{2}{|\alpha|^{(d+3)(n+1)}} \sum_{\mathbf{0}\leq \beta\leq \alpha} |\beta|^{(d+3)n} \right) \\
  & \leq \sup_{\alpha\in\N_0^d} \left( \frac{C^2}{|\alpha|^{(d+3)(n+1)}} \frac{(|\alpha|+2d)^2}{4} ( 1 + |\alpha| )^{(d+3)n+d} + \frac{2}{|\alpha|^{(d+3)(n+1)}} |\alpha|^{(d+3)n+d} \right)
  \\
    & \leq C_1,
\end{align*}
for some constant $C_1>0$ independent of $n\in\N_0$ and $V\in B(0,1)$.
Since the heat semigroup $( S(t) )_{t\geq 0}$
acts on $\ell^\infty_{w^{(n)}}({\cal C}_b(\R^d))$ componentwise, we
 also have
\begin{align*}
  \|S(t)\|_{{\cal L}(\ell^\infty_{w^{(n)}}({\cal C}_b(\R^d)) , \ell^\infty_{w^{(n)}}({\cal C}_b(\R^d)))}
  & = \sup_{\alpha\ge\mathbf{0}, \ \! j\ge-1} w'_\alpha \sum_{\beta\ge\mathbf{0}, \ \! i\ge-1} \left\| S(t) \right\|_{{\cal L}({\cal C}_b(\R^d) , {\cal C}_b(\R^d))} \frac{
    \mathbf{1}_{\{\alpha = \beta \}}
    \mathbf{1}_{\{ i = j \}}
  }{w_\beta}
  \\
  &= \left\| S(t) \right\|_{{\cal L}({\cal C}_b(\R^d) , {\cal C}_b(\R^d))}
  \\
    &\leq Ke^{\mu t},
\end{align*}
for some $K\geq 1$ and $\mu\in\R$ independent of $n\in\N_0$.
 Then, we have 
\begin{align*}
 & \big\|U(t) - \widetilde{U}(t)\big\|_{\ell^\infty_{w^{(n+1)}}({\cal C}_b(\R^d))} = \left\| \int_t^T S(s-t) [F(U(s)) - F(\widetilde{U}(s))] ds \right\|_{\ell^\infty_{w^{(n+1)}}({\cal C}_b(\R^d))}
  \\
    & \quad \leq Ke^{\mu (T-t)} \int_t^T \|F(U(s)) - F(\widetilde{U}(s))\|_{\ell^\infty_{w^{(n+1)}}({\cal C}_b(\R^d))} ds \\
    & \quad \leq Ke^{\mu (T-t)} \int_t^T \sup_{V\in B(0,1)} \|DF(V)\|_{{\cal L}(\ell^\infty_{w^{(n)}}({\cal C}_b(\R^d)) , \ell^\infty_{w^{(n+1)}}({\cal C}_b(\R^d)))} \|U(s) - \widetilde{U}(s)\|_{\ell^\infty_{w^{(n)}}({\cal C}_b(\R^d))} ds \\
    & \quad \leq C_1 Ke^{\mu (T-t)} \int_t^T \|U(s) - \widetilde{U}(s)\|_{\ell^\infty_{w^{(n)}}({\cal C}_b(\R^d))} ds.
\end{align*}
 Repeating the above estimates, we get
\begin{equation*}
  \begin{split}
    \big\|U(t) - \widetilde{U}(t)\big\|_{\ell^\infty_{w^{(n)}}({\cal C}_b(\R^d))} &\leq
    \big( C_1 Ke^{ (T-t)\mu} \big)^n \int_{t\leq s_1\leq \cdots \leq s_n\leq T}
    \big\|U(s_n) - \widetilde{U}(s_n)\big\|_{\ell^\infty_w ({\cal C}_b(\R^d))} ds_1 \cdots ds_n \\
    &\leq \frac{\big( C_1 Ke^{ (T-t)\mu } \big)^n}{n!}
    \sup_{t\in[0,T]}
    \big(
    (T-t)^n
    \big\|U(t) - \widetilde{U}(t)\big\|_{\ell^\infty_w ({\cal C}_b(\R^d))}
    \big),
  \end{split}
\end{equation*}
which yields the bound 
\begin{equation}\label{eqn-23}
  n! \sup_{t\in[0,T]} \big\|U(t) - \widetilde{U}(t)\big\|_{\ell^\infty_{w^{(n)}}({\cal C}_b(\R^d))} \leq
  ( C_1 K)^n
  \sup_{t\in[0,T]}
  \big(
  ( e^{\mu (T-t)} (T-t) )^n
  \big\|U(t) - \widetilde{U}(t)\big\|_{\ell^\infty_w ({\cal C}_b(\R^d))}
  \big)  
  , 
\end{equation}
 which tends to $0$ as $n$ tends to infinity,
 for $T$ small enough so that $C_1 Ke^{\mu (T-t)} (T-t)<1$. 

 \smallskip
 
 \noindent
 $iii)$
   On the other hand,
 if $U\ne\widetilde{U}$ then there exist at least one pair $(\alphastar,\jstar )$ and $t_0\in[0,T]$ such that $\|u_{\alphastar,\jstar }(t_0) - \tilde{u}_{\alphastar,\jstar }(t_0)\|_\infty > 0$, hence 
\begin{equation*}
  n! \sup_{t\in[0,T]} \big\|U(t) - \widetilde{U}(t)\big\|_{\ell^\infty_{w^{(n)}}} \geq \frac{n!}{C^{|\alphastar|} |\alphastar|^{(d+3)n}} \|u_{\alphastar,\jstar }(t_0) - \tilde{u}_{\alphastar,\jstar }(t_0)\|_\infty
\end{equation*}
where the right hand side tends to
infinity as $n$ tends to infinity, 
which leads to a contradiction to \eqref{eqn-23}.
\end{proof}
\section{Stochastic dominance: binary tree} %
\label{s5}
\noindent
 Our study of integrability of
 $\mathcal{H}_{t,T}\big(\mathcal{X}_t^x(c)\big)$ generalizes
 the stochastic dominance approach applied
 to ODEs in \cite{huangprivault2}. %
 From now on, we let $g(\alpha )$ denote either $g_{\theta , r}(\alpha )$
in \eqref{gthetar} or $g_\theta (\alpha )$ in \eqref{gtheta}. 
 We let $t\in [0,T]$ and $c\in \C (f)$, and
 consider the multiplicative weighted progeny 
\begin{equation}\label{w-progeny}
  N^{\sigma_\circ , \sigma_{\rm b}}_{t,T}(c) :=
  \prod_{{\mathbf{k}} \in \mathcal{K}_{t,T}^\circ(c)} \sigma_\circ
  (c^{\mathbf{k}}; I(c^{{\mathbf{k}}}))
  \prod_{{\mathbf{k}} \in \mathcal{K}_{t,T}^{\rm b}(c)} \sigma_{\rm b} (c^{\mathbf{k}})
\end{equation}
 of the branching process
 $\mathcal{X}_t^x(c) = \big(\mathcal{X}_{t,s}^x(c)\big)_{s\in[t,T]}$,
 where 
$$
  \sigma_{\rm b} , \sigma^{(0)}_\circ  : \N_0^d \times \N_{-1} \to (0,\infty )
  \quad \mbox{and} \quad 
  \sigma_\circ^{(i)} : \N_0^d \times \N_0 \to (0,\infty ), \quad i=1,\ldots , d,
$$
 are weight functions on the branches of $\mathcal{X}_t^x(c)$ 
 defined using the notation \eqref{u-indexed}
 for $c\in \C (f)$ and $z \in \mathcal{M}(c)$ as 
\begin{flalign*}
& \sigma_\circ (c;z) := \frac{|z_1 |}{\rho_*(T) q_c(z)} &
\end{flalign*}
\begin{subequations}
  \\[-33pt] 
  \begin{empheq}[left={\hskip-0.6cm =\hskip-0.2cm \;\empheqlbrace}]{align}
\label{s1}
      & \sigma_\circ^{(0)}(\alpha , -1)
 = \frac{1}{\rho_*(T)},
 \qquad \qquad \qquad \qquad \qquad \qquad \qquad \qquad \quad 
 \displaystyle
    c=\frac{1}{\alpha!} \partial^\alpha,
    \medskip
\\
\label{s2}
&    \sigma_\circ^{(0)}(\alpha, j)
    =
    \frac{d+1}{\rho_*(T)}
    \prod_{k=1}^d ( 1 + \alpha_k ) 
,
 \qquad \qquad \qquad \qquad \qquad  \qquad \quad \!
    \displaystyle
       c=\frac{1}{\alpha!} \partial^\alpha \circ f^{(j)}, 
        \medskip
        \smallskip
\\
\label{s3}
&
    \sigma_\circ^{(i)}(\alpha, j)
    =
    \frac{d+1    }{\rho_*(T)}
    \frac{
        (2 + \alpha_i)(3 + \alpha_i )}{12 
      }
    \prod_{k=1}^d ( 1 + \alpha_k )
    ,
 \qquad \qquad \qquad \!
 c=\frac{1}{\alpha!} \partial^\alpha \circ f^{(j)},
\end{empheq}
\end{subequations}
with $z_1=1$ in \eqref{s2} and 
$ 
     \displaystyle
     z_1=-\frac{1}{2}(1+\beta_i)(1+\alpha_i-\beta_i)
     $
     in \eqref{s3},
 and 
\begin{equation}
\nonumber %
 \sigma_{\rm b} (c) := 
  \begin{cases}
    \displaystyle
    \sigma_{\rm b} (\alpha , -1) =
    \frac{g(\alpha )}{\bar\rho(T) },
    &  \displaystyle c = \frac{1}{\alpha!} \pt^\alpha,
        \medskip
    \\
    \displaystyle
    \sigma_{\rm b} (\alpha , j) =
    \min ( \rho_*(T) , 1)
  \frac{g(\alpha )}{
  \bar\rho(T)},
    &  \displaystyle c = \frac{1}{\alpha!} \pt^\alpha \circ f^{(j)}, 
\end{cases}
\end{equation} 
 \noindent
 $\mathbf{0} \leq \beta \leq \alpha$, $i=1, \ldots, d$, 
 $(\alpha , j) \in \N_0^d \times \N_0$. 
 Under 
  \eqref{bound-code-1} or
  \eqref{bound-code-2}, we have 
  \begin{equation}
\nonumber %
  \frac{\| c (\phi ) \|_\infty}{\bar\rho(T)}
 \leq {\sigma_{\rm b} (c)}, 
 \quad
 c\in \C (f),
 \ z \in \mathcal{M}(c),
\end{equation}
 and from \eqref{functional} we find 
\begin{equation}
  \label{eqn-28}
  \left| \mathcal{H}_{t,T}\big(\mathcal{X}_t^x(c)\big) \right| \leq
  N^{\sigma_\circ , \sigma_{\rm b}}_{t, T}(c), \quad a.s. 
\end{equation}
     As a consequence, showing the integrability
     of $\mathcal{H}\left(\mathcal X_t^x(c)\right)$ requires us 
     to study the multiplicative weighted progeny at the
       right hand side
        of \eqref{eqn-28}. 
 Note that 
   since $\sigma_\circ (c;z)$
   is not uniformly bounded in $c\in \C (f)$ and $z\in \M (c)$,
   the multiplicative weighted
   progeny $N^{\sigma_\circ , \sigma_{\rm b}}_{t,T}(c)$
   cannot be controlled using the
   progeny of the branching process
 $\mathcal{X}_t^x(c)$. 

   \medskip

 In order to control the integrability of
 the branching process
 $\mathcal{X}_t^x(c) = \big(\mathcal{X}_{t,s}^x(c)\big)_{s\in[t,T]}$
 which has at most two offsprings at each generation, 
 we introduce a secondary coded binary branching chain
 $\widetilde{\cal X}_t(\tilde{c}) = \big( \widetilde{\cal X}_{t,s}(\tilde{c}) \big)_{s\in [t,T]}$
 which dominates $\mathcal{X}_t^x(c)$ 
 in a stochastic sense, and will be easier to handle
 in future arguments, see Figure~\ref{f1-11}.

\begin{figure}[H]
\begin{subfigure}{0.5\textwidth}
\tikzstyle{level 1}=[level distance=4cm, sibling distance=6cm]
\tikzstyle{level 2}=[level distance=4cm, sibling distance=5cm]
\tikzstyle{level 3}=[level distance=7.5cm, sibling distance=5.2cm]
\tikzstyle{level 4}=[level distance=7.5cm, sibling distance=5cm]
\begin{center}
\resizebox{1\textwidth}{!}{
\begin{tikzpicture}[scale=1.0,grow=right, sloped]
\node[rectangle,draw,black,text=black,thick, rounded corners=5pt,fill=gray!05]{\large \textcolor{gray!15}{$t$}} 
child {
  node[
     draw,black,text=black,thick, rounded corners=5pt,fill=gray!05,xshift=-0.5cm] {\large \textcolor{gray!15}{$T_\varnothing$}} 
            child {
  node[name=data, 
         draw, rounded corners=5pt,fill=gray!05,xshift=0.5cm] (main) 
        {\large \textcolor{gray!15}{$T_{(1)}$}} 
                child{
                  node[
                     draw,black,text=black,thick,xshift=6cm,yshift=-1cm,rounded corners=5pt,fill=gray!05]{\large \textcolor{gray!15}{\textcolor{gray!15}{$T_{(1,2)}$}}} 
                edge from parent
                node[rectangle,draw,below,fill=gray!05]{\large $\beta!^{-1} \partial^\beta \circ f'$ or $(\beta + \ind_i)!^{-1}\partial^{\beta + \ind_i}\circ f'$}
                }
                child{
                  node[
                     draw,black,text=black,thick,yshift=0.4cm,xshift=6cm,rounded corners=5pt,fill=gray!05]{\large \textcolor{gray!15}{\textcolor{gray!15}{$T_{(1,1)}$}}} 
          edge from parent
                node[rectangle,draw,below,fill=gray!05,xshift=0.5cm]{\large $(\alpha - \beta)!^{-1} \partial^{\alpha - \beta} \circ f$ or $(\alpha-\beta + \ind_i)!^{-1}\partial^{\alpha - \beta + \ind_i}$} 
                }
                edge from parent
                node[rectangle,draw,below,fill=gray!05] {\large $\alpha!^{-1}\partial^\alpha \circ f$} 
            }
                edge from parent
                node[rectangle,draw,fill=gray!05,below] {\large $\alpha!^{-1}\partial^\alpha $} 
};
\end{tikzpicture}
}
\end{center}
\caption{Original branching chain.} 
\end{subfigure}
\begin{subfigure}{0.5\textwidth}
\tikzstyle{level 1}=[level distance=6cm, sibling distance=6cm]
\tikzstyle{level 2}=[level distance=7cm, sibling distance=6.5cm]
\tikzstyle{level 3}=[level distance=8.5cm, sibling distance=5cm]
\tikzstyle{level 4}=[level distance=7.5cm, sibling distance=4cm]
\begin{center}
\resizebox{1\textwidth}{!}{
\begin{tikzpicture}[scale=1.0,grow=right, sloped]
        \node[draw,black,text=black,thick, rounded corners=5pt,fill=gray!05] {\large \textcolor{gray!15}{$t$}} 
            child {node[name=data, fill=gray!05,
         draw, rounded corners=5pt] (main) 
        {\large \textcolor{gray!15}{$T_\varnothing$}} 
                child{
                  node[
                    draw,fill=gray!05,text=black,thick,yshift=0.4cm,xshift=6cm,rounded corners=5pt]{\large \textcolor{gray!15}{$T_{(2)}$} \nodepart{two} $\nabla c$} 
                edge from parent
                node[rectangle,draw,fill=gray!05,below]{\large $\beta!^{-1}\partial^\beta \circ f'$ or $(\beta + \ind_i)!^{-1}\partial^{\beta + \ind_i}\circ f'$}
                }
                child{
                  node[
                     draw,black,text=black,thick,xshift=6cm,fill=gray!15,rounded corners=5pt]{\large \textcolor{gray!15}{$T_{(1)}$} \nodepart{two} $fgg$} %
          edge from parent
                node[rectangle,draw,fill=gray!05,below]{\large $(\alpha-\beta)!^{-1}\partial^{\alpha - \beta} \circ f$ or
                $(\alpha - \beta + \ind_i )!^{-1}\partial^{\alpha - \beta + \ind_i}$} 
                }
                edge from parent
                node[rectangle,draw,fill=gray!05,below]  {\large $\alpha!^{-1}\partial^\alpha \circ f$} 
}; 
\end{tikzpicture}
}
\end{center}
\caption{Dominating branching chain.} 
\end{subfigure}
\caption{Original {\em vs.} dominating branching chain.} 
\label{f1-11}
\end{figure}

\vskip-0.3cm

\noindent 
We denote by $\mathbb{K} := \{\varnothing\}\cup \bigcup_{n\in\N_1} \{1, 2\}^n$ the set of binary labels. 
 The ingredients of the dominating binary branching chain
 $\widetilde{\cal X}_t(\tilde{c})$ %
 are the following.
\begin{itemize}
\item Lifetimes: a family of i.i.d. stopping times
  $\big(\widetilde{\tau}_{\mathbf{k}}\big)_{{\mathbf{k}}\in\mathbb{K}}$
   with common exponential distribution with parameter $\lambda$.
  \item Coding mechanism: we construct a mechanism $\widetilde{\M}$ on the code set $\N_0^d\times \N_0$ by letting 
\begin{equation}\label{mechanism-new}
  \widetilde{\M} ((\alpha,j)) :=
  \big\{\big( (\alpha-\beta, 0), (\beta, j+1) \big), \big( (\alpha-\beta+\ind_i, 0), (\beta+\ind_i, j+1) \big)
  \big\}_{\mathbf{0}\leq \beta\leq \alpha 
    \atop
    i=1,\ldots,d }.
\end{equation}
\item Offsprings distribution:
  to every $(\alpha,j)\in \N_0^d \times \N_0$
  we associate a random variable $\widetilde{I}(\alpha,j)$ in $\widetilde{\M}((\alpha,j))$ with probability distribution
  of the form
  \begin{equation} 
      \label{offsp-dist-new}
      q_\alpha \big( (\alpha-\beta, 0), (\beta, j+1) \big) = q_\alpha^{(0)},
      \quad 
    q_\alpha \big( (\alpha-\beta+\ind_i, 0), (\beta+\ind_i, j+1) \big) = q_\alpha^{(i)}(\beta). 
\end{equation} 
\end{itemize}
Moreover,
the families $(\widetilde{\tau}_{\mathbf{k}})_{{\mathbf{k}}\in\mathbb{K}}$
and $\big(\widetilde{I} ( c^{\mathbf{k}} ) \big)_{{\mathbf{k}}\in\mathbb{K}}$
 are mutually independent.
 Given an initial code $\tilde{c} :=(\alphastar,\jstar )
\in \N_0^d \times \N_0$, we also construct a continuous-time branching
 chain, %
 as follows. 
\begin{itemize}
\item Start a branch at time $t$, indexed by the label $\varnothing$ and coded by $\tilde{c} = (\alphastar,\jstar )$. This is the zeroth generation and the common ancestor of all branches. The branch has lifetime
  $\widetilde{\tau}_\varnothing$, %
  and at the end of its lifetime 
  it splits to two first generation
  independent offsprings labelled $k = 1,2$ and coded by
   $\widetilde{I}_1( \alphastar,\jstar )$
   and $\widetilde{I}_2(\alphastar,\jstar )$ respectively. 
 \item The label of a $n$-$th$ generation 
   branch is denoted
   by ${\mathbf{k}}=(k_1, \ldots, k_{n-1}, k_n) \in \{1,2\}^n$, $n \geq 1$, and its code is given by $\tilde{c}^{{\mathbf{k}}} \in\N_0$.
   A branch ${\mathbf{k}}$
   has parent label 
   ${\mathbf{k}}^{-} := (k_1, \ldots, k_{n-2}, k_{n-1})$, 
   starts at time $\widetilde{T}_{{\mathbf{k}}^{-}}$
   and has lifetime $\widetilde{\tau}_{\mathbf{k}}$.
   At time $\widetilde{T}_{{\mathbf{k}}} := \widetilde{T}_{{\mathbf{k}}^{-}} + \widetilde{\tau}_{\mathbf{k}}$ the branch splits to two $(n+1)$-th generation
   independent offsprings
   with labels $({\mathbf{k}},1)$,
   $({\mathbf{k}},2)$ and codes
   $\widetilde{I}_1(\tilde{c}^{\mathbf{k}})$,
   $\widetilde{I}_2(\tilde{c}^{\mathbf{k}})$, respectively.
\end{itemize}
\noindent
 The above construction yields a coded branching chain denoted by
 $\widetilde{\cal X}_t(\alphastar,\jstar )$. 
 We denote by
   \begin{itemize}
   \item
     $\widetilde{\mathcal{K}}_{t,T}^\circ(\alphastar,\jstar ) $, 
$\widetilde{\mathcal{K}}_{t,T}^{\rm b}(\alphastar,\jstar )\subset \mathbb{K}$ 
 the label set of all branches %
 living at or having died before time $T$, respectively,
\item
  $\widetilde{\cal K}_{t,T} (\alphastar,\jstar )
  := \widetilde{\cal K}^\circ_{t,T} (\alphastar,\jstar )\cup \widetilde{\cal K}^b_{t,T}(\alphastar,\jstar )$ 
 the label set of all branches ever living in the time period
 $[t, T]$,
\item
  $\widetilde{X}_{t,T} = \big|\widetilde{K}_{t,T}(\alphastar,\jstar )\big|$
  the total progeny of
  $\widetilde{\cal X}_t(\alphastar,\jstar )$.
\end{itemize}
   Note that since
   $\widetilde{\cal X}_t(\alpha,j)$
   is always binary, in the distribution sense the total
   progeny
   $\widetilde{X}_{t,T}$
   does not depend on the initial code
   $(\alpha,j)$.
   The distribution of
   $\widetilde{X}_{t,T}$ is given in Lemma~\ref{progeny}. %

\medskip

 Our stochastic dominance result will be proved in
 Proposition~\ref{stoch-dom} in two steps.  
\begin{enumerate}[i)] 
\item For each $\alpha\in\N_0^d$, we merge pairs of adjacent branches coded by $\pt^\alpha$ and $\pt^\alpha \circ f$ into a single branch. %
\item We stochastically dominate the original distribution $\rho$ of the lifetimes of branches using an exponential distribution. 
\end{enumerate}  
\noindent 
 In Proposition~\ref{stoch-dom}, 
 given $g(\alpha )$ defined by either \eqref{gthetar}
 or \eqref{gtheta}, we let 
\begin{equation}
  \label{fkl323}
  \widetilde{\sigma}_{\rm b} (\alpha,j) :=
    \frac{g(\alpha )}{\bar{\rho}(T) },
    \quad
    (\alpha , j) \in \N_0^d \times \N_0, 
  \end{equation}
and
\begin{subequations}
  \begin{empheq}[left={\hskip-0.2cm \widetilde{\sigma}_\circ ( \alpha , j ;
        \tilde{z} ) :=\hskip-0.2cm \;\empheqlbrace}]{align}
\label{aaa1}
& {\sigma}_\circ^{(0)}(\alpha,j ),
 \ \ \ 
 \displaystyle
 \tilde{z} = \big( ( \alpha-\beta, 0), (\beta, j+1) \big) %
 \medskip
\\
\label{aaa2}
&
    {\sigma}_\circ^{(i)}(\alpha,j ),
\displaystyle
    \ \ \ \ \!
 \tilde{z} = \big( (\alpha-\beta+\ind_i, 0), (\beta+\ind_i, j+1) \big),
    \medskip
        \smallskip
\end{empheq}
\end{subequations}
$\mathbf{0} \leq \beta \leq \alpha$, $i=1,\ldots , d$,
$j\geq 0$, 
 be weight functions on the branches of
 $\widetilde{\cal X}_t$, 
  where $\sigma_\circ^{(i)}(\alpha, j)$, $i=0,1,\ldots , d$,
  are defined in \eqref{s2}-\eqref{s3}. 
\begin{proposition}[Stochastic dominance] 
\label{stoch-dom}
Let $g(\alpha )$ be defined by either \eqref{gthetar}
 or \eqref{gtheta}, under the respective conditions
\eqref{*-2} and
\eqref{cond-theta-2-0-2}.
Then, under Assumption~\hyperlink{H}{\bf H}$_\rho$, %
 let $( \alpha , j ) \in\N_0^d \times \N_0$, $t\in [0,T]$,
 and consider the multiplicative progenies 
\begin{equation}
\label{comp1} 
  \widetilde{N}^{\widetilde{\sigma}_\circ, \widetilde{\sigma}_{\rm b}}_{t,T} ( \alphastar,\jstar ) :=
  \prod_{{\mathbf{k}} \in \widetilde{\mathcal{K}}^\circ_{t,T}(\alphastar,\jstar )}
       \widetilde{\sigma}_\circ \big(\tilde{c}^{\mathbf{k}};
       \widetilde{I}( c^{\mathbf{k}})\big)
   \prod_{{\mathbf{k}} \in \widetilde{\mathcal{K}}^{\rm b}_{t,T}(\alphastar,\jstar )}
    \widetilde{\sigma}_{\rm b} (\tilde{c}^{\mathbf{k}}). 
\end{equation}
 Then, the multiplicative weighted progeny
 $N^{\sigma_\circ , \sigma_{\rm b}}_{t,T}(c)$
 in \eqref{w-progeny} is stochastically dominated by
 $$
   \begin{cases}
    \displaystyle
\widetilde{N}^{\widetilde{\sigma}_\circ, \widetilde{\sigma}_{\rm b}}_{t,T}( \alphastar,0)
& \mbox{ if } \displaystyle 
c = \displaystyle \frac{1}{\alphastar!} \pt^{\alphastar}, 
        \medskip
    \\
    \displaystyle
\widetilde{N}^{\widetilde{\sigma}_\circ, \widetilde{\sigma}_{\rm b}}_{t,T}( \alphastar,\jstar )    
& \mbox{ if } \displaystyle
c = \displaystyle \frac{1}{\alphastar!} \pt_x^{\alphastar} \circ f^{(\jstar )}.
 \end{cases}
$$
 In other words, for any $x\in\R$ we have 
\begin{equation*}
  \P \big( N^{\sigma_\circ , \sigma_{\rm b}}_{t,T}(c) \geq x \big) \le
  \begin{cases}
    \P \big( \widetilde{N}^{\widetilde{\sigma}_\circ, \widetilde{\sigma}_{\rm b}}_{t,T}(\alphastar,0) \geq x \big) &
    \mbox{ if }
    \displaystyle
    c = \frac{1}{\alphastar!} \pt^{\alphastar},
    \medskip
    \\
    \P \big( \widetilde{N}^{\widetilde{\sigma}_\circ, \widetilde{\sigma}_{\rm b}}_{t,T}( \alphastar,\jstar ) \geq x \big) &
    \mbox{ if }
    \displaystyle
    c = \frac{1}{\alphastar!} \pt^{\alphastar} \circ f^{(\jstar )}.
\end{cases}
\end{equation*}
\end{proposition} 
\begin{proof}
 $i)$
 On the one hand, we have 
\begin{subequations}
\begin{empheq}[left=\empheqlbrace]{align}
\label{weight-asmpt-1}
\displaystyle
& \sigma_{\rm b} (\alpha,-1) = \frac{\sigma_{\rm b} (\alpha,0)}{
  \min ( \rho_*(T) , 1)},
   \smallskip
  \bigskip
  \\[2pt]
  \label{weight-asmpt-1-1}
  &\displaystyle
   \sigma_\circ^{(0)} (\alpha,-1) = \frac{1}{\min ( \rho_*(T) , 1)},
\end{empheq}
\end{subequations}   
and on the other hand, from \eqref{fkl323} %
 and Lemma~\ref{jklfds19a}, 
 in both cases \eqref{gthetar} and \eqref{gtheta} we have 
\begin{subequations}
\begin{empheq}[left=\empheqlbrace]{align}
\label{weight-asmpt-2}
\displaystyle
   &1 \leq \widetilde{\sigma}_{\rm b} (\alpha-\beta,0)
\widetilde{\sigma}_{\rm b} (\beta,j+1)
\frac{
  {\sigma}_\circ^{(0)} (\alpha,j )
}{\widetilde{\sigma}_{\rm b} (\alpha,j)}
, 
\\[3pt]
\label{weight-asmpt-3}
\displaystyle
&1 \leq \widetilde{\sigma}_{\rm b} (\alpha-\beta+\ind_i,0)
\widetilde{\sigma}_{\rm b} (\beta+\ind_i,j+1)
\frac{
{\sigma}_\circ^{(i)} (\alpha,j )
}{\widetilde{\sigma}_{\rm b} (\alpha,j)}
  ,
\end{empheq}
\end{subequations}   
 $( \alpha , j ) \in\N_0^d \times \N_0$,
 $\mathbf{0}\leq \beta\leq \alpha$, $1\leq i\leq d$. 
 In what follows, instead of 
 \eqref{weight-asmpt-1}-\eqref{weight-asmpt-3}
 we work under the weaker conditions
\begin{subequations}
  \hspace*{-1.3cm}
  \begin{empheq}[left=\empheqlbrace]{flalign}
    \label{weight-dom-1}
  &  \displaystyle    %
    \sigma_{\rm b} (\alpha,-1) \leq \widetilde{\sigma}_{\rm b} (\alpha,0), \quad \sigma_{\rm b} (\alpha,j) \leq \widetilde{\sigma}_{\rm b} (\alpha,j), 
\\[3pt]
\label{weight-dom-2}
&
\displaystyle
    \max \left( 1 ,
    \sigma_{\rm b} (\alpha-\beta, 0) \sigma_{\rm b} (\beta,j+1)
    \frac{\sigma_\circ^{(0)} (\alpha,j )}{\sigma_{\rm b} (\alpha,j)}
    \right) 
        \leq
    \displaystyle
    \widetilde{\sigma}_{\rm b} (\alpha-\beta, 0)
    \widetilde{\sigma}_{\rm b} (\beta,j+1)
    \frac{ {\sigma}_\circ^{(0)} (\alpha,j )}{\widetilde{\sigma}_{\rm b} (\alpha,j)},
\\[3pt]
\nonumber 
&
\displaystyle
     \max \Big(1, \max
     \big( \sigma_{\rm b} \left(\alpha-\beta+\mathbf{1}_i,-1\right), \sigma_{\rm b} \left(\alpha-\beta+\mathbf{1}_i, 0\right) \sigma_\circ^{(0)}\left(\alpha-\beta+\mathbf{1}_i,-1\right) \big)
       \\
       \label{weight-dom-3}
       &
       \displaystyle 
    \left.
   ~~~
            \times \sigma_{\rm b} \left(\beta+\mathbf{1}_i, j+1\right) \frac{\sigma_\circ^{(i)}(\alpha, j)}{\sigma_{\rm b} (\alpha, j)}\right) 
           \leq
            \widetilde{\sigma}_{\rm b} \left(\alpha-\beta+\mathbf{1}_i, 0\right) \widetilde{\sigma}_{\rm b} \left(\beta+\mathbf{1}_i, j+1\right) \frac{ {\sigma}_\circ^{(i)}(\alpha, j)}{\widetilde{\sigma}_{\rm b} (\alpha, j)},  
\end{empheq}
\end{subequations}   
 $( \alpha , j ) \in\N_0^d \times \N_0$, 
 $\mathbf{0}\leq \beta\leq \alpha$, and $i=1,\ldots , d$. 

\smallskip
\noindent
$ii)$
 We introduce an intermediate binary branching chain
 $\widebar{\cal X}_t (\alphastar,\jstar )$ via a pathwise modification of the original branching chain $\mathcal{X}_t^x(c)$: since each branch coded by ${\alpha!}^{-1} \pt^\alpha$ has a unique offspring, which is coded by
  ${\alpha!}^{-1} \pt^\alpha\circ f$, we collect all such father-son pair of branches, and merge each pair as a single branch, recode it
 by ${\alpha!}^{-1} \pt^\alpha\circ f$, as shown in Figure~\ref{f1}.
 
 \vspace{-0.2cm}

 \begin{figure}[H]
\begin{subfigure}{\textwidth}
\tikzstyle{level 1}=[level distance=4cm, sibling distance=6cm]
\tikzstyle{level 2}=[level distance=4cm, sibling distance=5cm]
\tikzstyle{level 3}=[level distance=7.5cm, sibling distance=5.2cm]
\tikzstyle{level 4}=[level distance=7.5cm, sibling distance=5cm]
\begin{center}
\resizebox{0.85\textwidth}{!}{
\begin{tikzpicture}[scale=1.0,grow=right, sloped]
\node[rectangle,draw,black,text=black,thick, rounded corners=5pt,fill=gray!05]{\large \textcolor{white}{$T_\varnothing$}} 
child {
  node[
     draw,black,text=black,thick, rounded corners=5pt,fill=gray!05,xshift=-0.5cm] {\large \textcolor{white}{$T_\varnothing$}} 
            child {
  node[name=data, 
         draw, rounded corners=5pt,fill=gray!05,xshift=0.5cm] (main) 
        {\large \textcolor{white}{$T_\varnothing$}} 
                child{
                  node[
                     draw,black,text=black,thick,xshift=7.5cm,yshift=-1cm,rounded corners=5pt,fill=gray!05]{\large \textcolor{white}{$T_\varnothing$}} 
                          child{
                    edge from parent
                    node[above,trapezium,draw,fill=gray!05]{$(1,2,2)$}
                    node[rectangle,draw,fill=gray!05,below]{$\cdots $}%
                    }
                    child{
                    edge from parent
                    node[above,trapezium,draw,fill=gray!05]{$(1,2,1)$} %
                    node[rectangle,draw,fill=gray!05,below,xshift=0cm]{\large $\cdots$} %
                    }
                edge from parent
                node[above,trapezium,draw,fill=gray!05]{$(1,2)$}
                node[rectangle,draw,below,fill=gray!05]{\large $\beta!^{-1} \partial^\beta \circ f'$ or $(\beta + \ind_i)!^{-1}\partial^{\beta + \ind_i}\circ f'$}
                }
                child{
                  node[
                     draw,black,text=black,thick,yshift=0.4cm,xshift=6cm,rounded corners=5pt,fill=gray!05]{\large \textcolor{white}{$T_\varnothing$}} 
                                                   child{
                    edge from parent
                    node[above,trapezium,draw,fill=gray!05]{$(1,1,2)$} %
                    node[rectangle,draw,fill=gray!05,below]{$\cdots $}%
                    }
                    child{
                    edge from parent
                    node[above,trapezium,draw,fill=gray!05]{$(1,1,1)$} %
                    node[rectangle,draw,fill=gray!05,below,xshift=0cm]{\large $\cdots$} %
                    }
          edge from parent
                node[above,trapezium,draw,fill=gray!05]{$(1,1)$}
                node[rectangle,draw,below,fill=gray!05,xshift=0.5cm]{\large $(\alpha - \beta)!^{-1} \partial^{\alpha - \beta} \circ f$ or $(\alpha-\beta + \ind_i)!^{-1}\partial^{\alpha - \beta + \ind_i}$} 
                }
                edge from parent
                node[above,trapezium,draw,fill=gray!05] {$(1)$}
                node[rectangle,draw,below,fill=gray!05] {\large $\alpha!^{-1}\partial^\alpha \circ f$} 
            }
                edge from parent
                node[above,trapezium,draw,fill=gray!05] {$\varnothing$}
                node[rectangle,draw,fill=gray!05,below] {\large $\alpha!^{-1}\partial^\alpha $} 
};
\end{tikzpicture}
}
\end{center}
  \caption{Tree with a pair of initial branches.} 
\end{subfigure}
\end{figure} 

\begin{figure}[H]\ContinuedFloat
\begin{subfigure}{\textwidth}
\tikzstyle{level 1}=[level distance=6cm, sibling distance=6cm]
\tikzstyle{level 2}=[level distance=7cm, sibling distance=6.5cm]
\tikzstyle{level 3}=[level distance=8.5cm, sibling distance=5cm]
\tikzstyle{level 4}=[level distance=7.5cm, sibling distance=4cm]
\begin{center}
\resizebox{0.9\textwidth}{!}{
\begin{tikzpicture}[scale=1.0,grow=right, sloped]
        \node[draw,black,text=black,thick, rounded corners=5pt,fill=gray!05] {\large \textcolor{white}{$T_\varnothing$}} 
            child {node[name=data, fill=gray!05,%
         draw, rounded corners=5pt] (main) 
        {\large \textcolor{white}{$T_\varnothing$}} %
                child{
                  node[%
                    draw,fill=gray!05,text=black,thick,yshift=0.4cm,xshift=4cm,rounded corners=5pt]{\large \textcolor{white}{$T_\varnothing$} \nodepart{two} $\nabla c$} 
                    child{
                    edge from parent
                    node[above,trapezium,draw,fill=gray!05]{$(2,2)$} %
                    node[rectangle,draw,fill=gray!05,below]{$\cdots $}%
                    }
                    child{
                    edge from parent
                    node[above,trapezium,draw,fill=gray!05]{$(2,1)$} %
                    node[rectangle,draw,fill=gray!05,below]{\large $\cdots$} %
                    }
                edge from parent
                node[above,trapezium,draw,fill=gray!05]{$(2)$}
                node[rectangle,draw,fill=gray!05,below]{\large $\beta!^{-1}\partial^\beta \circ f'$ or $(\beta + \ind_i)!^{-1}\partial^{\beta + \ind_i}\circ f'$}%
                }
                child{
                  node[%
                     draw,black,text=black,thick,xshift=8cm,fill=gray!15,rounded corners=5pt]{\large \textcolor{white}{$T_\varnothing$} \nodepart{two} $fgg$} %
                          child{
                    edge from parent
                    node[above,trapezium,draw,fill=gray!05]{$(1,2)$} %
                    node[rectangle,draw,fill=gray!05,below]{$\cdots $}%
                    }
                    child{
                    edge from parent
                    node[above,trapezium,draw,fill=gray!05]{$(1,1)$} %
                    node[rectangle,draw,fill=gray!05,below,xshift=0cm]{\large $\cdots$} %
                    }
          edge from parent
                node[above,trapezium,draw,fill=gray!05]{$(1)$}
                node[rectangle,draw,fill=gray!05,below]{\large $(\alpha-\beta)!^{-1}\partial^{\alpha - \beta} \circ f$ or
                $(\alpha - \beta + \ind_i )!^{-1}\partial^{\alpha - \beta + \ind_i}$} %
                }
                edge from parent
                node[above,trapezium,draw,fill=gray!05] {$\varnothing$}
                node[rectangle,draw,fill=gray!05,below]  {\large $\alpha!^{-1}\partial^\alpha \circ f$} %
}; 
\end{tikzpicture}
}
\end{center}
  \caption{Merged tree with a single branch.} 
\end{subfigure}
  \caption{Merging of pair branches into a single branch.} 
    \label{f1}
\end{figure}

\vskip-0.3cm

\noindent  
 This modification yields the same mechanism as \eqref{mechanism-new} and the same offspring distribution as \eqref{offsp-dist-new}, and thus, a binary branching chain with the same distribution as replacing the exponential lifetime distribution of $\widetilde{\cal X}_t$ by the original lifetime distribution $\rho$.

  \smallskip

  \noindent
  $iii)$
  Next, we reweigh all branches of the modified chain $\widebar{\cal X}_t(\alphastar,\jstar )$
   coded by ${\alpha!}^{-1} \pt^\alpha \circ f^{(j)}$ using the boundary weight $\widetilde{\sigma}_{\rm b} (\alpha,j)$ and the inner weights
  ${\sigma}_\circ^{(0)} (\alpha,j )$ and ${\sigma}_\circ^{(i)} (\alpha,j )$.
  We denote the resulting multiplicative weighted 
  progeny of $\widebar{\cal X}_t(\alphastar,\jstar )$ by 
  \begin{equation*}
  \widebar{N}^{\widetilde{\sigma}_\circ, \widetilde{\sigma}_{\rm b}}_{t,T} ( \alphastar,\jstar ) :=
  \prod_{{\mathbf{k}} \in \widebar{\mathcal{K}}^\circ_{t,T}(\alphastar,\jstar )}
       \widetilde{\sigma}_\circ \big(\bar{c}^{\mathbf{k}};
       \widebar{I}(c^{\mathbf{k}})\big)
   \prod_{{\mathbf{k}} \in \widebar{\mathcal{K}}^{\rm b}_{t,T}(\alphastar,\jstar )}
    \widetilde{\sigma}_{\rm b} (\bar{c}^{\mathbf{k}})
  , 
\end{equation*}
 which can be rewritten as 
 \begin{equation}
   \label{weighted-progeny-dom-2}
  \widebar{N}^{\widetilde{\sigma}_\circ, \widetilde{\sigma}_{\rm b}}_{t,T}(\alphastar,\jstar ) =
  \widetilde{\sigma}_{\rm b} (\alphastar,\jstar ) \prod_{{\mathbf{k}}\in \mathbb{K}, \widebar{T}_{\mathbf{k}} \leq T} \Delta \widebar{N}^{\widetilde{\sigma}_\circ, \widetilde{\sigma}_{\rm b}}_{\widebar{T}_{\mathbf{k}} ,T}( \alphastar,\jstar ),
\end{equation} 
 where
 \begin{equation}
\nonumber %
  \Delta \widebar{N}^{\widetilde{\sigma}_\circ, \widetilde{\sigma}_{\rm b}}_{\widebar{T}_{\mathbf{k}} ,T}(\alphastar,\jstar ) :=
  \widetilde{\sigma}_{\rm b} (\widebar{I}_1(\bar{c}^{\mathbf{k}}))
  \widetilde{\sigma}_{\rm b} (\widebar{I}_2(\bar{c}^{\mathbf{k}}))
  \frac{\widetilde{\sigma}_\circ
    \big(\bar{c}^{\mathbf{k}};
    \widebar{I}(\bar{c}^{\mathbf{k}})\big)}{\widetilde{\sigma}_{\rm b} (\bar{c}^{\mathbf{k}})}, 
\end{equation}
 and, letting $\bar{c}^{\mathbf{k}} = (\alpha,j)$, we have 
\begin{align}
  \label{weighted-progeny-dom-delta-2-2}
  & \Delta \widebar{N}^{\widetilde{\sigma}_\circ, \widetilde{\sigma}_{\rm b}}_{\widebar{T}_{\mathbf{k}} ,T}(\alphastar,\jstar ) =
  \\
  \nonumber
  &
  \begin{cases}
    \displaystyle
    \widetilde{\sigma}_{\rm b} (\alpha-\beta, 0)
    \widetilde{\sigma}_{\rm b} (\beta,j+1)
    \frac{{\sigma}_\circ^{(0)} (\alpha,j )}{\widetilde{\sigma}_{\rm b} (\alpha,j)}, & \widebar{I}(\bar{c}^{\mathbf{k}}) = \big( (\alpha-\beta, 0), (\beta, j+1) \big),
    \medskip
    \\
\displaystyle
\widetilde{\sigma}_{\rm b} (\alpha-\beta+\ind_i, 0)
\widetilde{\sigma}_{\rm b} (\beta+\ind_i,j+1)
\frac{{\sigma}_\circ^{(i)} (\alpha,j )}{\widetilde{\sigma}_{\rm b} (\alpha,j)}, & \widebar{I}(\bar{c}^{\mathbf{k}}) = \big( (\alpha-\beta+\ind_i, 0), (\beta+\ind_i, j+1) \big).
  \end{cases}
\end{align} 
 By Assumption~\hyperlink{H}{\bf H}$_\rho$-\hyperlink{asmpt-1}{\rm ii)}, %
 $\widebar{T}_{\mathbf{k}} - \widebar{T}_{{\mathbf{k}}^-}$ stochastically dominates $\widetilde{T}_{\mathbf{k}} - \widetilde{T}_{{\mathbf{k}}^-}$
 for every 
 ${\mathbf{k}}\in \mathbb{K} \setminus \{\varnothing \}$, 
 as we have 
\begin{align*}
  \P ( \widebar{T}_{\mathbf{k}} - \widebar{T}_{{\mathbf{k}}^-} \leq r )  & = 1- \bar\rho(r)
  \\
  & \leq 1-e^{-\lambda r}
  \\
  & = \P\big( \widetilde{T}_{\mathbf{k}} - \widetilde{T}_{{\mathbf{k}}^-} \leq r \big), \quad r\geq 0. 
\end{align*}
By the independence of the sequences
$\big(\widebar{T}_{\mathbf{k}} - \widebar{T}_{{\mathbf{k}}^-}\big)_{{\mathbf{k}}\in \mathbb{K}}$ and
$\big(\widetilde{T}_{\mathbf{k}} - \widetilde{T}_{{\mathbf{k}}^-}
\big)_{{\mathbf{k}}\in \mathbb{K}}$,
we conclude from \cite[Theorem 6.B.3]{shaked2007stochastic} 
that $(\widetilde{\tau}_{\mathbf{k}} = \widetilde{T}_{\mathbf{k}} - \widetilde{T}_{{\mathbf{k}}^-})_{{\mathbf{k}}\in \mathbb{K}}$
is stochastically dominated by
 $(\overline{\tau}_{\mathbf{k}} = \widebar{T}_{\mathbf{k}} - \widebar{T}_{{\mathbf{k}}^-} )_{{\mathbf{k}}\in \mathbb{K}}$. 
 Using \eqref{weighted-progeny-dom-2} and the independence between $\big(\widebar{T}_{\mathbf{k}}\big)_{{\mathbf{k}}\in \mathbb{K}}$
and $\big(\Delta \widebar{N}^{{\sigma}_\circ, \widetilde{\sigma}_{\rm b}}_{\widebar{T}_{\mathbf{k}} ,T}(\alphastar,\jstar )\big)_{{\mathbf{k}}\in \mathbb{K}}$, 
 we derive 
\begin{align} 
  \nonumber
 \qquad \P \big( \widebar{N}^{\widetilde{\sigma}_\circ, \widetilde{\sigma}_{\rm b}}_{t ,T}(\alphastar,\jstar ) \geq x \big)  & = \E \left[ \P \left( \prod_{{\mathbf{k}}\in \mathbb{K}, \widebar{T}_{\mathbf{k}} \leq T} \Delta \widebar{N}^{\widetilde{\sigma}_\circ, \widetilde{\sigma}_{\rm b}}_{\widebar{T}_{\mathbf{k}} ,T}(\alphastar,\jstar ) \geq \frac{x}{\widetilde{\sigma}_{\rm b} (\alphastar,\jstar )}
    \Bigg|
     \big(\widebar{T}_{\mathbf{k}}\big)_{{\mathbf{k}}\in \mathbb{K}} \right) \right] \\
  \nonumber
  & \hskip-2cm
  = \E \left[
    {
      \underbrace{\P \left(
        \prod_{{\mathbf{k}}\in \mathbb{K}, t_{\mathbf{k}} \leq T
        \atop \text{branch ${\mathbf{k}}$ splits at $t_{\mathbf{k}}$}} \left(
        \widetilde{\sigma}_\circ
        \big(\tilde{c}^{\mathbf{k}}; \widetilde{I}(\tilde{c}^{\mathbf{k}})\big)
        \frac{\widetilde{\sigma}_{\rm b} (\widetilde{I}_i(\tilde{c}^{\mathbf{k}}))}{\widetilde{\sigma}_{\rm b} (\tilde{c}^{\mathbf{k}})} \right) \geq \frac{x}{\widetilde{\sigma}_{\rm b} (\alphastar,\jstar )}
        \right)}_{=: \Phi ((t_{\mathbf{k}})_{{\mathbf{k}}\in \mathbb{K}} )}
    }~_{  \hskip-0.16cm
      |{t_{\mathbf{k}} = \widebar{T}_{\mathbf{k}}, {\mathbf{k}}\in \mathbb{K}} }
    \right]
  \\
  \nonumber
  & \hskip-2cm
  = \E \big[ \Phi \big(\big(\widebar{T}_{\mathbf{k}}\big)_{{\mathbf{k}}\in \mathbb{K}}\big) \big]
  \\
  \label{jkla11}
  & \hskip-2cm
  =: \E [ \Psi ( (\overline{\tau}_{\mathbf{k}} )_{{\mathbf{k}}\in \mathbb{K}} ) ],
  \quad
 x\in\R. 
\end{align}
 Since each multiplier in the definition of $\Phi$ is nonnegative due to
 \eqref{weighted-progeny-dom-delta-2-2}, from the
 conditions~\eqref{weight-dom-2}-\eqref{weight-dom-3} with $1$ at the
 left hand side of inequalities, we see that the mapping
 \begin{align*}
   \Psi : [0,\infty)^{\mathbb{K}} & \to [0,1]
     \\
     (t_{\mathbf{k}})_{{\mathbf{k}}\in \mathbb{K}} & \mapsto  \Psi
     (
     (t_{\mathbf{k}})_{{\mathbf{k}}\in \mathbb{K}} )
     \end{align*} 
 defined from the relation
 $$
 \Psi ( (\overline{\tau}_{\mathbf{k}} )_{{\mathbf{k}}\in \mathbb{K}} )
 :=
 \Phi \big(\big(\widebar{T}_{\mathbf{k}}\big)_{{\mathbf{k}}\in \mathbb{K}}\big)
 $$ 
  is non-increasing in every component $t_{\mathbf{k}}$.
 Hence, since 
 $(\overline{\tau}_{\mathbf{k}})_{{\mathbf{k}}\in \mathbb{K}}$ stochastically dominates $(\widetilde{\tau}_{\mathbf{k}})_{{\mathbf{k}}\in \mathbb{K}}$,
 using the equivalent definition of stochastic dominance for random vectors \cite[Eq.~(6.B.4)]{shaked2007stochastic}
 and the fact that 
 the equations \eqref{weighted-progeny-dom-2}-\eqref{weighted-progeny-dom-delta-2-2}
 and therefore \eqref{jkla11}
 also hold for 
 $\widetilde{N}^{\widetilde{\sigma}_\circ, \widetilde{\sigma}_{\rm b}}_{t,T} ( \alphastar,\jstar )$
 in \eqref{weighted-progeny-dom-2}
 by replacing the exponential splitting times
 $\widetilde{T}_{\mathbf{k}}$
 with the splitting times
 $\widebar{T}_{\mathbf{k}}$
 of
 $\widebar{\cal X}_t(\alphastar,\jstar )$,  
 we obtain, for any $x\in\R$,
\begin{align*}
  \P \big( \widebar{N}^{\widetilde{\sigma}_\circ, \widetilde{\sigma}_{\rm b}}_{t ,T}(\alphastar,\jstar ) \geq x \big) & = \E [ \Psi ((\overline{\tau}_{\mathbf{k}})_{{\mathbf{k}}\in \mathbb{K}}) ]
  \\
  & \leq \E [ \Psi ( (\widetilde{\tau}_{\mathbf{k}} )_{{\mathbf{k}}\in \mathbb{K}}
    ) ]
  \\
   & = \P \big( \widetilde{N}^{\widetilde{\sigma}_\circ, \widetilde{\sigma}_{\rm b}}_{t ,T}(\alphastar,\jstar ) \geq x \big),
\end{align*} 
that is, $\widebar{N}^{\widetilde{\sigma}_\circ, \widetilde{\sigma}_{\rm b}}_{t ,T}(\alphastar,\jstar )$ is stochastically dominated by $\widetilde{N}^{\widetilde{\sigma}_\circ, \widetilde{\sigma}_{\rm b}}_{t ,T}(\alphastar,\jstar )$.

 \smallskip

 \noindent
 $iv)$ To prove that $N^{\sigma_\circ , \sigma_{\rm b}}_{t,T}(c)$ is stochastically dominated by
 $\widetilde{N}^{\widetilde{\sigma}_\circ, \widetilde{\sigma}_{\rm b}}_{t ,T}(\alphastar,\jstar )$, it remains to show that $N^{\sigma_\circ , \sigma_{\rm b}}_{t,T}(c)$ can be dominated
pathwise by the multiplicative weighted progeny
 $\widebar{N}^{\widetilde{\sigma}_\circ, \widetilde{\sigma}_{\rm b}}_{t ,T} ( \alphastar,\jstar )$ of the modified chain. For this, we rewrite \eqref{w-progeny} as 
\begin{equation}\label{discrete-chain}
  N^{\sigma_\circ , \sigma_{\rm b}}_{t,T}(c) = \sigma_{\rm b} (c) \prod_{{\mathbf{k}}\in \mathbb{K}, T_{\mathbf{k}} \leq T} \Delta N_{T_{\mathbf{k}}}(c), 
\end{equation}
 where
\begin{equation}
  \nonumber %
  \Delta N_{T_{\mathbf{k}}}(c) = \frac{\sigma_\circ
   (c^{\mathbf{k}} ; I(c^{\mathbf{k}}))}{\sigma_{\rm b} (c^{\mathbf{k}})}
  \prod_{i=2}^{|I(c^{\mathbf{k}} )|} \sigma_{\rm b} (I_i( c^{\mathbf{k}})). 
\end{equation}
 When $c^{\mathbf{k}} = {\alpha!}^{-1} \pt^\alpha$, we have
\begin{equation}\label{weighted-progeny-delta-2}
  \Delta N_{T_{\mathbf{k}}}(c) = \sigma_{\rm b} (\alpha,0) \frac{\sigma_\circ^{(0)} (\alpha,-1)}{ \sigma_{\rm b} (\alpha,-1)}, 
\end{equation}
\noindent
 whereas when $c^{\mathbf{k}} = {\alpha!}^{-1} \pt^\alpha \circ f^{(j)}$, we have
\begin{align}
  \label{weighted-progeny-delta-3}
  & \Delta N_{T_{\mathbf{k}}}(c) =
  \begin{cases}
    \displaystyle
    \sigma_{\rm b} (\alpha-\beta,0) \sigma_{\rm b} (\beta,j+1)
    \frac{\sigma_\circ^{(0)} (\alpha,j )}{\sigma_{\rm b} (\alpha,j)}, &
    \displaystyle
    I_1(c^{\mathbf{k}}) = 1, I_3(c^{\mathbf{k}}) = \frac{1}{\beta!} \pt^\beta \circ f^{(j+1)},
    \medskip
    \\
    \displaystyle
     \sigma_{\rm b} (\alpha-\beta+\ind_i,-1)
     \sigma_{\rm b} (\beta+\ind_i,j+1)
     \frac{\sigma_\circ^{(i)} (\alpha,j )}{\sigma_{\rm b} (\alpha,j)}, &
     \displaystyle I_3(c^{\mathbf{k}}) = -\frac{1}{2} (\beta_i+1)(\alpha_i-\beta_i+1).
  \end{cases}
\end{align} 
We observe from the construction of the modified chain $\widebar{\cal X}_t(\alphastar,\jstar )$
that, pathwise, all branches of $\mathcal{X}_t^x(c)$
 that split no later than $T$ are reserved for $\widebar{\cal X}_t(\alphastar,\jstar )$ in the time interval $[t, T]$, with some pairs of those branches merged. 
We partition the set of branches 
 $K := \{{\mathbf{k}}\in \mathbb{K} \ : \ T_{\mathbf{k}} \leq T\}$ as 
\begin{equation*}
 K = K_m \cup K_m(c), 
\end{equation*}
 into the subset $K_m$ of merged branches 
 and its complement $K_m(c)$.
 We consider the `merging' map $\Lambda$ from 
 $K$ into the set of branches  
\begin{equation*}
 \widebar{K} := \{{\mathbf{k}}\in \mathbb{K}\}_{\widebar{T}_{\mathbf{k}} \leq T}, 
\end{equation*}
defined by mapping pairs in $K_m$ into single elements in $\widebar{K}$,
and by being injective on $K_m(c)$. In this way, we can split the
product over $K$ in $N^{\sigma_\circ , \sigma_{\rm b}}_{t,T}(c)$ as 
\begin{equation*}
  \prod_{{\mathbf{k}}\in \mathbb{K}, T_{\mathbf{k}} \leq T} = \prod_{{\mathbf{k}}\in K_m}
  \prod_{{\mathbf{k}}\in K_m(c)}, 
\end{equation*}
and split the
product over $\widebar{K}$ in $\widebar{N}^{\widetilde{\sigma}_\circ, \widetilde{\sigma}_{\rm b}}_{t ,T}( \alphastar,\jstar ) $ as 
\begin{equation*}
  \prod_{{\mathbf{k}}\in \mathbb{K}, \widebar{T}_{\mathbf{k}} \leq T} =
  \prod_{{\mathbf{k}}\in \Lambda (K_m)}
  \prod_{{\mathbf{k}}\in \Lambda (K_m(c))}
  \prod_{{\mathbf{k}}\in \widebar{K} \setminus \Lambda(K)}.
\end{equation*}
 The conditions~\eqref{weight-dom-1}-\eqref{weight-dom-2} with $1$ at the
 left hand side of inequalities, together with \eqref{weighted-progeny-dom-delta-2-2}
 imply
\begin{equation*}
 \prod_{{\mathbf{k}}\in \widebar{K} \setminus \Lambda(K)} \Delta \widebar{N}^{\widetilde{\sigma}_\circ, \widetilde{\sigma}_{\rm b}}_{\widebar{T}_{\mathbf{k}} ,T}( \alphastar,\jstar ) \geq 1.
\end{equation*}
 Condition~\eqref{weight-dom-2} with
$$
 \sigma_{\rm b} (\alpha-\beta,0)
 \sigma_{\rm b} (\beta,j+1)
 \frac{\sigma_\circ^{(0)} (\alpha,j )}{\sigma_{\rm b} (\alpha,j)}
$$
 and 
$$
 \sigma_{\rm b} (\alpha-\beta+\ind_i,-1)
\sigma_{\rm b} (\beta+\ind_i,j+1)
\frac{\sigma_\circ^{(i)} (\alpha,j )}{\sigma_{\rm b} (\alpha,j)}
$$
at the left hand side,
 together with \eqref{weighted-progeny-dom-delta-2-2} and \eqref{weighted-progeny-delta-3}, imply
\begin{equation*}
  \prod_{{\mathbf{k}}\in K_m(c)} \Delta N_{T_{\mathbf{k}}}(c)
  \leq \prod_{{\mathbf{k}}\in \Lambda(K_m(c))} \Delta \widebar{N}^{\widetilde{\sigma}_\circ, \widetilde{\sigma}_{\rm b}}_{\widebar{T}_{\mathbf{k}} ,T} ( \alphastar,\jstar ) .
\end{equation*}
 On the other hand, condition~\eqref{weight-dom-3} with
$$
\sigma_{\rm b} (\alpha-\beta+\ind_i,0)
\sigma_\circ^{(0)} (\alpha-\beta+\ind_i,-1)
\sigma_{\rm b} (\beta+\ind_i,j+1)
\frac{\sigma_\circ^{(i)} (\alpha,j )}{\sigma_{\rm b} (\alpha,j)}
$$
at the left hand side,
 together with \eqref{weighted-progeny-dom-delta-2-2} and \eqref{weighted-progeny-delta-2}-\eqref{weighted-progeny-delta-3}, imply 
\begin{equation*}
  \prod_{{\mathbf{k}}\in K_m} \Delta N_{T_{\mathbf{k}}}(c)
  \leq
    \prod_{{\mathbf{k}}\in \Lambda (K_m)} \Delta \widebar{N}^{\widetilde{\sigma}_\circ, \widetilde{\sigma}_{\rm b}}_{\widebar{T}_{\mathbf{k}} ,T}( \alphastar,\jstar ) .
\end{equation*}
 Finally, the condition~\eqref{weight-dom-1} implies 
\begin{equation*}
 \sigma_{\rm b} (c) \leq \widetilde{\sigma}_{\rm b} (\alphastar,\jstar ), 
\end{equation*}
 hence we can now conclude from \eqref{weighted-progeny-dom-2}
and \eqref{discrete-chain} 
 that
\begin{equation*}
 N^{\sigma_\circ , \sigma_{\rm b}}_{t,T}(c) \leq \widebar{N}^{\widetilde{\sigma}_\circ, \widetilde{\sigma}_{\rm b}}_{t ,T} ( \alphastar,\jstar ). 
\end{equation*}
\end{proof}
\section{Integrability results} 
\label{s7}
\noindent 
 In this section, using algebraic and stochastic dominance
 arguments we prove Propositions~\ref{integ-final-1}
 and \ref{integ-final-2}, which ensure the integrability
 of $\mathcal{H}_{t,T}\big(\mathcal{X}_t^x(c)\big)$. 
\subsection{Algebraic dominance: Hamilton--Jacobi equation} %
\label{fjklds231}
\noindent
 In this subsection we use the binary branching chain
 $\widetilde{\cal X}_t(\tilde{c})$ 
 and its associated multiplicative progenies
 $\widetilde{N}^{\widetilde{\sigma}_\circ, \widetilde{\sigma}_{\rm b}}_{t,T} ( \alphastar,\jstar )$ defined in \eqref{comp1}.  
\begin{lemma} 
\label{fjkldf344}
 Let $(\alpha,j)\in \N_0^d\times \N_0$. 
The conditional expectations
$$
 A^{\widetilde{\sigma}_\circ ,\widetilde{\sigma}_{\rm b} }_k (\alpha,j): =
 \E \big[ \widetilde{N}^{\widetilde{\sigma}_\circ, \widetilde{\sigma}_{\rm b}}_{t ,T}(\alpha,j) \ \! \big| \ \!
 \widetilde{X}_{t,T}
   = 2k+1 \big],
 \quad k \geq 0, 
$$
 of the multiplicative weighted progeny
\begin{equation*}
  \widetilde{N}^{\widetilde{\sigma}_\circ, \widetilde{\sigma}_{\rm b}}_{t ,T}(\alphastar,\jstar ) :=
  \prod_{{\mathbf{k}} \in \widetilde{\mathcal{K}}^\circ_{t,T}(\alphastar,\jstar )}
  \widetilde{\sigma}_\circ
  \big(\tilde{c}^{\mathbf{k}}; \widetilde{I}( c^{\mathbf{k}}) \big)
  \prod_{{\mathbf{k}} \in \widetilde{\mathcal{K}}^{\rm b}_{t,T}(\alphastar,\jstar )} \widetilde{\sigma}_{\rm b} \big(\tilde{c}^{\mathbf{k}}\big),
\end{equation*} 
 do not depend on $t\in [0,T]$, and they satisfy the recursion 
 \begin{equation}
   \nonumber %
   \left\{
  \begin{aligned}
    &  A^{\widetilde{\sigma}_\circ ,\widetilde{\sigma}_{\rm b} }_0 (\alpha,j) = \widetilde{\sigma}_{\rm b} (\alpha,j),
    \\
 & A^{\widetilde{\sigma}_\circ ,\widetilde{\sigma}_{\rm b} }_{k+1} (\alpha,j) = \frac{1}{k+1}
    \sum_{
      \beta+\gamma = \alpha \atop 
      \beta, \gamma \in \N_0^d
    } \sum_{
      l_1+l_2 = k \atop 
      l_1, l_2 \geq 0
  }
    \Bigg( {\sigma}_\circ^{(0)} (\alpha,j )
    q_\alpha^{(0)}
    A^{\widetilde{\sigma}_\circ ,\widetilde{\sigma}_{\rm b} }_{l_1} (\gamma , 0)
    A^{\widetilde{\sigma}_\circ ,\widetilde{\sigma}_{\rm b} }_{l_2} (\beta , j+1)
    \\
    & 
      \qquad\qquad \qquad\qquad + \sum_{i=1}^d
    {\sigma}_\circ^{(i)} (\alpha,j ) q_\alpha^{(i)}(\beta)
    A^{\widetilde{\sigma}_\circ ,\widetilde{\sigma}_{\rm b} }_{l_1} (\gamma+\ind_i ,0)
    A^{\widetilde{\sigma}_\circ ,\widetilde{\sigma}_{\rm b} }_{l_2} (\beta+\ind_i ,j+1) \Bigg), \quad k\geq 0, 
  \end{aligned}\right.
  \end{equation}
\noindent
and the scaling relation 
\begin{equation}
      \label{jhklsd11}
    A^{a \widetilde{\sigma}_\circ , b \widetilde{\sigma}_{\rm b}}_k (\alpha,j) =
    a^k 
    b^{k+1}
    A^{\widetilde{\sigma}_\circ ,\widetilde{\sigma}_{\rm b} }_k (\alpha,j), \quad k \geq 0, 
    \end{equation} 
 for any $a,b>0$. 
\end{lemma}
\begin{proof}
  Using \eqref{recursion-total-sol}
  and the independence between lifetimes and coding mechanism, 
  we have 
\begin{align*}
 & A^{\widetilde{\sigma}_\circ ,\widetilde{\sigma}_{\rm b} }_{k+1} (\alpha,j)
   \P \big( \widetilde{X}_{t,T} = 2k+3 \big)
      \\
      & 
        =
      \E \big[ \widetilde{N}^{\widetilde{\sigma}_\circ, \widetilde{\sigma}_{\rm b}}_{t ,T}(\alpha,j)
        \mathbf{1}_{\{
          |\widetilde{\cal K}_{t,T}(\alpha,j)|= 2k+3
          \}} \big]
      \\
      & 
        =
        \E\big[ \E\big[ \widetilde{N}^{\widetilde{\sigma}_\circ, \widetilde{\sigma}_{\rm b}}_{t,T}(\alpha,j) \mathbf{1}_{\{
              |\widetilde{\cal K}_{t,T}(\alpha,j)|= 2k+3 
              \}}
            \mathbf{1}_{\{ T_\emptyset \le T \}}
            \mid T_\emptyset \big]\big]
      \\
      & 
        =
        \int_t^T \lambda e^{ - \lambda (s-t)} \E\left[
          \tilde\sigma_\circ(\alpha,j; I(\alpha,j)) \prod_{i=1}^2 \widetilde{N}^{\widetilde{\sigma}_\circ, \widetilde{\sigma}_{\rm b}}_{s,T}(\widetilde{I}_i(\alpha,j))
        \mathbf{1}_{
          \{
          \sum_{i=1}^2
         |\widetilde{\cal K}_{s,T}(\tilde I_i(\alpha,j))|= 2k+2
          \}
        } \right] ds
      \\
      &
        =
        \int_t^T \lambda e^{ - \lambda (s-t)}
        \\[-3pt] 
   &
   \sum_{0\le\beta\le\alpha}
        \bigg(
        \sigma_\circ^{(0)}(\alpha,j) q_\alpha^{(0)}
        \E\left[ \widetilde{N}^{\widetilde{\sigma}_\circ, \widetilde{\sigma}_{\rm b}}_{s,T}(\alpha-\beta, 0) \widetilde{N}^{\widetilde{\sigma}_\circ, \widetilde{\sigma}_{\rm b}}_{s,T}(\beta, j+1) \mathbf{1}_{\{
            |\widetilde{\cal K}_{s,T}(\alpha-\beta, 0)|
            +
            |\widetilde{\cal K}_{s,T}(\beta, j+1)|
            = 2k+2
            \}
        } \right]
   \\[-3pt] 
   &
   + \sum_{i=1}^d
   {\sigma}_\circ^{(i)} (\alpha,j )
   q_\alpha^{(i)}(\beta)
   \\[-3pt]
   &
   \qquad \qquad \quad   
     \E\big[ \widetilde{N}^{\widetilde{\sigma}_\circ, \widetilde{\sigma}_{\rm b}}_{s,T}(\alpha-\beta + \ind_i, 0) \widetilde{N}^{\widetilde{\sigma}_\circ, \widetilde{\sigma}_{\rm b}}_{s,T}(\beta
     + \ind_i,
     j+1) \mathbf{1}_{\{
      |\widetilde{\cal K}_{s,T}(\alpha-\beta+ \ind_i, 0)|
       +
       |\widetilde{\cal K}_{s,T}(\beta+ \ind_i, j+1)|
       = 2k+2
            \}
  } \big]
 \Bigg) ds
   \\
  & =
  \sum_{\mathbf{0}\leq \beta\leq \alpha} \int_t^T
  \lambda e^{ - \lambda (s-t)}
   \sum_{
      l_1+l_2 = k \atop 
      l_1, l_2 \geq 0
 }
   \P \big( \widetilde{X}_{s,T} = 2l_1+1 \big)
   \P \big( \widetilde{X}_{s,T} = 2l_2+1 \big)
   \\
   &
   \Bigg(
   {\sigma}_\circ^{(0)} (\alpha,j )
   q_\alpha^{(0)}
   A^{\widetilde{\sigma}_\circ ,\widetilde{\sigma}_{\rm b} }_{l_1} (\alpha - \beta , 0)
   A^{\widetilde{\sigma}_\circ ,\widetilde{\sigma}_{\rm b} }_{l_2} (\beta ,j+1)
   \\[-7pt] 
   &
   \quad \qquad \qquad \qquad \qquad \qquad \qquad
   \left. + \sum_{i=1}^d
   {\sigma}_\circ^{(i)} (\alpha,j )
   q_\alpha^{(i)}(\beta)
   A^{\widetilde{\sigma}_\circ ,\widetilde{\sigma}_{\rm b} }_{l_1} (\alpha - \beta + \ind_i , 0)
   A^{\widetilde{\sigma}_\circ ,\widetilde{\sigma}_{\rm b} }_{l_2} (\beta + \ind_i ,j+1) \right) ds
     \\
  & =
  \sum_{\mathbf{0}\leq \beta\leq \alpha} \int_t^T
  \lambda e^{ - \lambda (s-t)}
   \sum_{
      l_1+l_2 = k \atop 
      l_1, l_2 \geq 0
 }
   e^{ - 2\lambda( T - s)}
   \big( 1- e^{ - \lambda(T - s)} \big)^{l_1+l_2}
   \Bigg(
   {\sigma}_\circ^{(0)} (\alpha,j )
   q_\alpha^{(0)}
   A^{\widetilde{\sigma}_\circ ,\widetilde{\sigma}_{\rm b} }_{l_1} (\alpha - \beta , 0)
   A^{\widetilde{\sigma}_\circ ,\widetilde{\sigma}_{\rm b} }_{l_2} (\beta ,j+1)
   \\[-7pt] 
   &
   \quad \qquad \qquad \qquad \qquad \qquad \qquad
   \left. + \sum_{i=1}^d
   {\sigma}_\circ^{(i)} (\alpha,j )
   q_\alpha^{(i)}(\beta)
   A^{\widetilde{\sigma}_\circ ,\widetilde{\sigma}_{\rm b} }_{l_1} (\alpha - \beta + \ind_i , 0)
   A^{\widetilde{\sigma}_\circ ,\widetilde{\sigma}_{\rm b} }_{l_2} (\beta + \ind_i ,j+1) \right) ds     \\
  & =
  \lambda e^{ - \lambda (T-t)}
 \int_t^T
    \P \big( \widetilde{\cal X}_{s,T} = 2k+1 \big) ds 
   \sum_{\mathbf{0}\leq \beta\leq \alpha}
   \sum_{
      l_1+l_2 = k \atop 
      l_1, l_2 \geq 0
 } 
   \Bigg(
   {\sigma}_\circ^{(0)} (\alpha,j )
   q_\alpha^{(0)}
   A^{\widetilde{\sigma}_\circ ,\widetilde{\sigma}_{\rm b} }_{l_1} (\alpha - \beta , 0)
   A^{\widetilde{\sigma}_\circ ,\widetilde{\sigma}_{\rm b} }_{l_2} (\beta ,j+1)
   \\[-2pt] 
   &
   \ \quad \qquad \qquad \qquad \qquad \qquad \qquad
   \left. + \sum_{i=1}^d
   {\sigma}_\circ^{(i)} (\alpha,j )
   q_\alpha^{(i)}(\beta)
   A^{\widetilde{\sigma}_\circ ,\widetilde{\sigma}_{\rm b} }_{l_1} (\alpha - \beta + \ind_i , 0)
   A^{\widetilde{\sigma}_\circ ,\widetilde{\sigma}_{\rm b} }_{l_2} (\beta + \ind_i ,j+1) \right)
        \\
  & =
 \frac{ \P \big( \widetilde{\cal X}_{t,T} = 2k+3 \big)}{k+1} 
   \sum_{\mathbf{0}\leq \beta\leq \alpha}
   \sum_{
      l_1+l_2 = k \atop 
      l_1, l_2 \geq 0
 } 
   \Bigg(
   {\sigma}_\circ^{(0)} (\alpha,j )
   q_\alpha^{(0)}
   A^{\widetilde{\sigma}_\circ ,\widetilde{\sigma}_{\rm b} }_{l_1} (\alpha - \beta , 0)
   A^{\widetilde{\sigma}_\circ ,\widetilde{\sigma}_{\rm b} }_{l_2} (\beta ,j+1)
   \\[-7pt] 
   &
   \ \quad \qquad \qquad \qquad \qquad \qquad \qquad
   \left. + \sum_{i=1}^d
   {\sigma}_\circ^{(i)} (\alpha,j )
   q_\alpha^{(i)}(\beta)
   A^{\widetilde{\sigma}_\circ ,\widetilde{\sigma}_{\rm b} }_{l_1} (\alpha - \beta + \ind_i , 0)
   A^{\widetilde{\sigma}_\circ ,\widetilde{\sigma}_{\rm b} }_{l_2} (\beta + \ind_i ,j+1) \right), 
\end{align*}
 where we used the equalities  
\begin{align*}
 \lambda \int_t^T 
       \P \big( \widetilde{\cal X}_{s,T} = 2k+1 \big) ds
      &= \lambda \int_t^Te^{ - \lambda (T-s)}
      \big(1-e^{- \lambda (T-s)}\big)^k ds \\
       &= \lambda \sum_{j=0}^k (-1)^j \binom{k}{j} \int_t^T e^{-\lambda (j+1)(T-s) } ds \\
      &= \sum_{j=0}^k \frac{(-1)^j}{j+1} \binom{k}{j}
      ( 1 - e^{-\lambda (j+1)(T-t) } )
      \\
      &= \frac{1}{k+1} \big( 1- e^{-\lambda (T - t)} \big)^{k+1}
       \\
       &= \frac{e^{ \lambda (T - t)} }{k+1} 
      \P \big( \widetilde{\cal X}_{t,T} = 2k+3 \big), \quad k \geq 0. 
\end{align*}
 Finally, since the branching chain
 $\widetilde{\cal X}_t$
 is binary, we note that 
 for any $a,b>0$ we have
 \begin{align*}
  \widetilde{N}^{a \widetilde{\sigma}_\circ, b \widetilde{\sigma}_{\rm b}}_{t ,T}(\alphastar,\jstar ) & =
  \prod_{{\mathbf{k}} \in \widetilde{\mathcal{K}}^\circ_{t,T}(\alphastar,\jstar )}
  a \widetilde{\sigma}_\circ
  \big(\tilde{c}^{\mathbf{k}}; \widetilde{I}( c^{\mathbf{k}}) \big)
  \prod_{{\mathbf{k}} \in \widetilde{\mathcal{K}}^{\rm b}_{t,T}(\alphastar,\jstar )}
  b
  \widetilde{\sigma}_{\rm b} \big(\tilde{c}^{\mathbf{k}}\big)
  \\
  & =
 a^{(|\widetilde{X}_{t,T} (\alpha , j)|-1)/2} 
  b^{(|\widetilde{X}_{t,T}(\alpha , j)|+1)/2} 
  \prod_{{\mathbf{k}} \in \widetilde{\mathcal{K}}^\circ_{t,T}(\alphastar,\jstar )}
  \widetilde{\sigma}_\circ
  \big(\tilde{c}^{\mathbf{k}}; \widetilde{I}( c^{\mathbf{k}}) \big)
  \prod_{{\mathbf{k}} \in \widetilde{\mathcal{K}}^{\rm b}_{t,T}(\alphastar,\jstar )}
  \widetilde{\sigma}_{\rm b} \big(\tilde{c}^{\mathbf{k}}\big),
\end{align*} 
 which implies \eqref{jhklsd11}.
\end{proof}

\noindent
In Proposition~\ref{gfdkj4} we derive
 a contact Hamilton--Jacobi equation
 for the bivariate generating function 
 of the multiplicative weighted progeny
 $\widetilde{N}^{\widetilde{\sigma}_\circ, \widetilde{\sigma}_{\rm b}}_{t ,T}(\alphastar,\jstar )$, $(\alpha , j ) \in\N_0^d \times \N_0$. 
 In what follows, since 
 $\widetilde{\sigma}_{\rm b} $
 and 
 $\widetilde{\sigma}_\circ$
 defined in \eqref{fkl323} and \eqref{aaa1}-\eqref{aaa2}
 do not depend on $j$,
 we let 
\begin{align}
\nonumber %
A_k ( \alpha ) := A_k^{\rho_*(T) \widetilde{\sigma}_\circ ,
  \bar{\rho}(T) \widetilde{\sigma}_{\rm b} } ( \alpha,j )
 = \E \big[ \widetilde{N}^{\rho_*(T) \widetilde{\sigma}_\circ,
    \bar{\rho}(T) \widetilde{\sigma}_{\rm b}}_{t ,T}(\alpha,j) \ \! \big| \ \! \widetilde{X}_{t,T} = 2k+1 \big],
\end{align} 
 $(\alpha , k ) \in\N_0^d \times \N_0$. 
\begin{proposition}[Contact Hamilton--Jacobi equation]
\label{gfdkj4}
 The bivariate generating function 
\begin{equation}
\nonumber %
  G(s,x) := 
   \sum_{(\alpha , k ) \in\N_0^d \times \N_0} 
   s^k x^\alpha
   A_k(\alpha ),
 \quad (x,s) \in \real^d \times \real, 
\end{equation}
 satisfies the (backward in time) contact Hamilton--Jacobi equation
\begin{subequations}
\begin{empheq}[left=\empheqlbrace]{align}
\label{contact-HJ}
   &  \displaystyle
    \frac{\pt G}{\pt s} (s,x) = G(s, x)^2 + \frac{1}{2} \left| \nabla G (s, x) \right|^2, & %
    s<0, 
    \medskip
   \\
\label{contact-HJ-0}
   & \displaystyle
    G(0,x) = G_0(x), \quad x\in \real^d, 
\end{empheq}
\end{subequations}
 where 
 $$
 G_0 ( x ) := \sum_{\alpha \in\N_0^d} g(\alpha) x^\alpha, \qquad x\in \real^d,
 $$ 
 is the generating function of $(g(\alpha))_{\alpha\in\N_0^d}$. 
\end{proposition}
\begin{Proof} 
 Relation~\eqref{contact-HJ-0} follows from 
 \eqref{fkl323}, which reads 
 $$
 A_0 (\alpha) =
 \bar{\rho} (T) \widetilde{\sigma}_{\rm b} (\alpha , j)
 = 
 g(\alpha),
 \quad (\alpha , j)\in\N_0^d\times \N_0.
$$
 On the other hand, combining Lemma~\ref{fjkldf344}
 with the facts that
 $$
 \rho_*(T) \sigma_\circ^{(0)}(\alpha,j) q_\alpha^{(0)} = 1
 \quad
 \mbox{and}
 \quad 
    \rho_*(T)
   \sigma_\circ^{(i)}(\alpha, j)
   q_\alpha^{(i)} (\beta )
   =
    \frac{
(1+\beta_i)(1+\alpha_i-\beta_i)
}{2}
    $$
 which are due to \eqref{offsp-dist}-\eqref{offsp-dist-2} 
 and \eqref{s1}-\eqref{s3}, 
we have 
$$ A_{k+1} (\alpha) = \frac{1}{k+1}
    \sum_{
      \beta+\gamma = \alpha
      \atop
       \beta, \gamma \in \N_0^d}
    \sum_{
          l_1+l_2 = k \atop l_1, l_2 \geq 0
 }
    \Bigg( A_{l_1} (\gamma) A_{l_2} (\beta) \left. + \frac{1}{2} \sum_{i=1}^d (1 + \gamma_i ) (1 + \beta_i )
    A_{l_1} (\gamma + \ind_i)
    A_{l_2} (\beta + \ind_i) \right), 
    $$
     which implies 
  \begin{align*}
  & \frac{\pt G}{\pt s} (s,x) = \sum_{
    ( \alpha , k ) \in\N_0^d \times \N_0} (k+1)
  A_{k+1} (\alpha ) s^k x^\alpha \\
  & = \sum_{(\alpha , k ) \in\N_0^d\times \N_0}
  s^k x^\alpha
  \sum_{
      \beta+\gamma = \alpha \atop 
      \beta, \gamma \in \N_0^d
  }
  \sum_{
      l_1+l_2 = k \atop 
      l_1, l_2 \geq 0
  }
  \Bigg(
  A_{l_1} (\gamma )
  A_{l_2} (\beta )
  \left. + \frac{1}{2} \sum_{i=1}^d (1 + \gamma_i ) (1 + \beta_i )
  A_{l_1} (\gamma + \ind_i )
  A_{l_2} (\beta + \ind_i ) \right)
    \\
    & = \sum_{( \gamma , l_1 ) \N_0^d \times \N_0}
    A_{l_1} (\gamma ) s^{l_1} x^\gamma \sum_{( \beta , l_2 ) \in\N_0^d \times \N_0}
    A_{l_2} (\beta ) s^{l_2} x^\beta \\
    & \qquad + \frac{1}{2} \sum_{i=1}^d \left( \sum_{
      ( \gamma , l_1 ) \in\N_0^d \times \N_0} (1 + \gamma_i )
    A_{l_1} (\gamma + \ind_i ) s^{l_1} x^\gamma \sum_{
      ( \beta , l_2 ) \in\N_0^d \times \N_0} (1 + \beta_i )
    A_{l_2} (\beta + \ind_i ) s^{l_2} x^\beta \right) \\
    & = G(s, x)^2 + \frac{1}{2} \sum_{i=1}^d
    \left( \frac{\partial G}{\partial x_i} (s, x) \right)^2, 
  \end{align*}
 and yields \eqref{contact-HJ}.  
\end{Proof}
\noindent
As the contact Hamilton--Jacobi
equation
\eqref{contact-HJ}-\eqref{contact-HJ-0} may not have a closed form
 solution, we will dominate 
 $A_k (\alpha )$
 using a simpler equation in Proposition~\ref{fjkldf324}. 
\begin{proposition}[Algebraic dominance]  
  \label{fjkldf324}
 Given %
 $(g(\alpha))_{\alpha\in\N_0^d}$ a positive sequence, 
 let 
 $\big(
 \widehat{A}_k(\alpha) \big)_{( \alpha , k ) \in\N_0^d \times \N_0}$ be
 defined by the recursion 
\begin{subequations}
\begin{empheq}[left=\empheqlbrace]{align}
\label{eqn-13-2}
   &  \displaystyle
  \widehat{A}_0(\alpha) = g(\alpha),
    \medskip
   \\
\label{eqn-13-2-0}
   & \displaystyle
  \widehat{A}_{k+1}(\alpha) = \frac{1}{k+1}
    \sum_{
      \beta+\gamma = \alpha \atop 
      \beta, \gamma \in \N_0^d
    }
    \sum_{
      l_1+l_2 = k \atop 
      l_1, l_2 \geq 0
    }
    \sum_{i=1}^d (1 + \gamma_i ) (1 + \beta_i ) \widehat{A}_{l_1}(\gamma+\ind_i)
    \widehat{A}_{l_2}(\beta+\ind_i). 
\end{empheq}
\end{subequations}
 Then, %
\begin{enumerate}[i)]
\item
 the bivariate generating function 
\begin{equation*}
  \widehat{G}(s,x)
  := \sum_{( \alpha , k ) \in \N_0^d \times \N_0} \widehat{A}_k(\alpha) x^\alpha s^k,
   \quad (s,x) \in \real\times \real^d, 
\end{equation*}
 satisfies the (backward in time) Hamilton--Jacobi equation 
\begin{subequations}
\begin{empheq}[left=\empheqlbrace]{align}
   \label{HJ}
   &  \displaystyle
    \frac{\pt \widehat{G}}{\pt s} (s,x) =
    \big| \nabla \widehat{G} (s, x) \big|^2, & %
    \ s<0, 
    \medskip
   \\
   \label{HJ-0}
   & \displaystyle
\widehat{G}(0,x) = G_0(x), x\in\R^d; 
\end{empheq}
\end{subequations}
\item
 we have the domination 
\begin{equation}
  \label{jfkld111a}
  A_k %
   (\alpha )  
  =
  \E \big[ \widetilde{N}^{\rho_*(T)\widetilde{\sigma}_\circ, \bar\rho (T)\widetilde{\sigma}_{\rm b}}_{t ,T}(\alpha,j) \ \! \big| \ \! \widetilde{X}_{t,T} = 2k+1 \big]
  \leq 
   \widehat{A}_k (\alpha), 
\end{equation}
 $(\alpha, k)\in \N_0^d\times \N_0$, provided that 
\begin{equation}
  \label{dom-asmp}
    (1 + \alpha_i ) \widehat{A}_k(\alpha+\ind_i) \geq \sqrt{\frac{2}{d}} \widehat{A}_k(\alpha ),
     \quad ( \alpha , k ) \in\N_0^d \times \N_0, \ \ i=1,\ldots , d. 
\end{equation}
\end{enumerate}
\end{proposition} 
\begin{Proof} 
  $i)$
  Relation~\eqref{HJ-0} follows from \eqref{eqn-13-2-0}. 
  Next, as in the proof of Proposition~\ref{gfdkj4}, we have 
\begin{align*}
   \frac{\pt G}{\pt s} (s,x) & = \sum_{
    ( \alpha , k ) \in\N_0^d \times \N_0} (k+1)
  \widehat{A}_{k+1} (\alpha) s^k x^\alpha \\
  & = \frac{1}{2}
  \sum_{(\alpha , k ) \in\N_0^d\times \N_0}
  s^k x^\alpha
  \sum_{
      \beta+\gamma = \alpha \atop 
      \beta, \gamma \in \N_0^d
  }
  \sum_{
      l_1+l_2 = k \atop 
      l_1, l_2 \geq 0
  }
 \sum_{i=1}^d (1 + \gamma_i ) (1 + \beta_i )
  \widehat{A}_{l_1 } (\gamma + \ind_i)
  \widehat{A}_{l_2} (\beta + \ind_i ) 
    \\
    & = \frac{1}{2} \sum_{i=1}^d
    \sum_{
      ( \gamma , l_1 ) \in\N_0^d \times \N_0} (1 + \gamma_i )
    \widehat{A}_{l_1} (\gamma + \ind_i ) s^{l_1} x^\gamma \sum_{
      ( \beta , l_2 ) \in\N_0^d \times \N_0} (1 + \beta_i )
    \widehat{A}_{l_2} (\beta + \ind_i ) s^{l_2} x^\beta
    \\
    & = \frac{1}{2} \sum_{i=1}^d \left( \frac{\partial G}{\partial x_i}
    (s, x) \right)^2, 
\end{align*}
 which yields \eqref{HJ}.  

 \smallskip

 \noindent
 $ii)$
   We use induction on $k\geq 0$. When $k=0$ we clearly have $\widehat{A}_0(\alpha) \geq
   A_0 (\alpha)$ for all $\alpha\in\N_0^d$. Next, suppose that
   $\widehat{A}_{k'}(\alpha) \geq A_{k'} (\alpha)$ for all $\alpha\in\N_0^d$ and $0\leq k'\leq k$. Then, using \eqref{eqn-13-2}-\eqref{eqn-13-2-0},
    \eqref{dom-asmp} and the induction
 hypothesis, we have
\begin{align*}
   & \widehat{A}_{k+1}(\alpha) \geq \frac{1}{k+1}
    \sum_{
      \beta+\gamma = \alpha \atop 
      \beta, \gamma \in \N_0^d
   } \sum_{
      l_1+l_2 = k \atop 
      l_1, l_2 \geq 0
    }
    \left( \widehat{A}_{l_1}(\gamma ) \widehat{A}_{l_2} (\beta ) 
        + \frac{1}{2} \sum_{i=1}^d (1 + \gamma_i ) (1 + \beta_i ) \widehat{A}_{l_1}(\gamma+\ind_i) \widehat{A}_{l_2}(\beta+\ind_i) \right) \\
    & \geq \frac{1}{k+1}
    \sum_{
      \beta+\gamma = \alpha \atop 
      \beta, \gamma \in \N_0^d
    }
    \sum_{
      l_1+l_2 = k \atop 
      l_1, l_2 \geq 0
    }
    \left( A_{l_1} (\gamma)
    A_{l_2} (\beta) 
    + \frac{1}{2} \sum_{i=1}^d (1 + \gamma_i ) (1 + \beta_i )
    A_{l_1} (\gamma + \ind_i)
    A_{l_2} (\beta + \ind_i) \right)
\\
&= A_{k+1} (\alpha),
 \quad k \geq 0.
\end{align*}
\end{Proof}
\noindent
 The next result is a consequence of
 Proposition~\ref{fjkldf324}. %
\begin{corollary} 
\label{fjkldf324-2}
Let $\big(\widehat{A}_k(\alpha) \big)_{( \alpha , k ) \in\N_0^d \times \N_0}$ be defined by the recursion \eqref{eqn-13-2}-\eqref{eqn-13-2-0}. 
Under Condition~\eqref{dom-asmp}, 
 we have 
\begin{equation*}
 \E \big[ \widetilde{N}^{\widetilde{\sigma}_\circ , \widetilde{\sigma}_{\rm b} }_{t ,T} ( \alpha,j ) \big]
    \leq \frac{
    e^{-\lambda(T-t)} }{\bar\rho(T) }
    \widehat{G}_\alpha
    \left(
    \frac{1- e^{-\lambda(T-t)} }{\rho_*(T)\bar\rho(T) }
    \right), \quad
    (\alpha , j) \in \N_0^d \times \N_0, 
\end{equation*}
 where 
\begin{equation*}
 \widehat{G}_\alpha (s) := \sum_{k=0}^\infty \widehat{A}_k(\alpha) s^k
\end{equation*}
 is the univariate generating function of
 the sequence
 $\big(\widehat{A}_k(\alpha))_{k\geq 0 }$, $\alpha\in\N_0^d$.  
\end{corollary}
\begin{Proof}
  Hence, from \eqref{jhklsd11} and
  \eqref{jfkld111a}, we obtain 
\begin{align*}
\E\big[
  \widetilde{N}^{\widetilde{\sigma}_\circ, \widetilde{\sigma}_{\rm b}}_{t ,T}(\alpha, j)
  \big]
& =
\sum_{k\geq 0}
\E \big[
    \widetilde{N}^{\widetilde{\sigma}_\circ, \widetilde{\sigma}_{\rm b}}_{t ,T}(\alpha, j) \ \!
    \big| \ \! 
    \widetilde{X}_{t,T}
    = 2k+1
    \big]
\P
\big(
    \widetilde{X}_{t,T} = 2k+1
    \big)
 \\
    & = 
    e^{-\lambda(T-t)}
    \sum_{k=0}^\infty \big( 1- e^{-\lambda(T-t)} \big)^k
    A_k^{ \widetilde{\sigma}_\circ , \widetilde{\sigma}_{\rm b} } (\alpha , j)
 \\
    & = 
    \frac{e^{-\lambda (T-t)}}{\bar\rho(T)}
    \sum_{k=0}^\infty
    \left( \frac{1- e^{-\lambda(T-t)})^k}{\rho_*(T)\bar\rho(T) } \right)^k
    A_k^{ \rho_*(T) \widetilde{\sigma}_\circ , \bar{\rho}(T) \widetilde{\sigma}_{\rm b} } (\alpha , j)
    \\
    & \leq  
    \frac{e^{-\lambda (T-t)}}{\bar\rho(T)}
      \sum_{k=0}^\infty
    \left( \frac{1- e^{-\lambda(T-t)})^k}{\rho_*(T)\bar\rho(T) } \right)^k
    \widehat{A}_k (\alpha ),
    \\
    & \leq \frac{e^{-\lambda(T-t)} }{\bar\rho(T) }
    \widehat{G}_\alpha
    \left( \frac{1- e^{-\lambda(T-t)}}{\rho_*(T)\bar\rho(T) } \right), 
\end{align*}
\end{Proof}
\noindent 
In Lemma~\ref{fjkls111} we solve the
recursion \eqref{eqn-13-2}-\eqref{eqn-13-2-0} using
the Hamilton--Jacobi equation~\eqref{HJ}. 
\begin{lemma}
    \label{fjkls111}
     Assume that the generating function
     $ G_0$ of $(g(\alpha))_{\alpha\in\N_0^d}$
 is analytic in a neighborhood of $\mathbf{0} \in \R^d$, 
 and can be written as 
\begin{equation*}
  G_0(x) = H (x_1+ \cdots +x_d),
   \quad x\in \R^d, 
\end{equation*}
 where $H$ is given by a series expansion of the form 
\begin{equation*}
 H(z) = \sum_{m=0}^\infty h(m) z^m,
 \quad z\in \R,
\end{equation*}
 and satisfies $H'(z) \ne 0$ for all $z\in \R$. 
 Then,
 \begin{enumerate}[i)]
 \item
   the Hamilton--Jacobi equation~\eqref{HJ}-\eqref{HJ-0} 
   admits the solution
   \begin{equation}\label{bi-gf}
      \widehat{G}(s,x) =  - s d \left| H'(\langle y(s,x) \rangle) \right|^2 + H(\langle y(s,x) \rangle), 
\end{equation}
\item the recursion 
   \eqref{eqn-13-2}-\eqref{eqn-13-2-0} admits the solution 
\begin{align}
\label{A-hat-0}
  \widehat{A}_k(\alpha )
 & = \frac{(2d)^k}{\alpha!(k+1)!}
  \frac{\partial^{k+|\alpha|-1}}{\partial z^{k+|\alpha|-1}}\left[ H'(z) \right]^{k+1}_{\mid z = 0} 
  \\
  \nonumber
  & =
  (2d)^k\frac{(k+|\alpha|-1)!}{\alpha!(k+1)!} \sum_{
      m_1+\cdots m_{k+1} = k+|\alpha|-1 \atop 
      m_1, \ldots , m_{k+1} \geq 0
 } \prod_{j=1}^{k+1}
  \big( (1 + m_j ) h(1 + m_j ) \big),
\end{align} 
 $(\alpha , k ) \in \N_0^d\times \N_0$, 
where the above summation is over integer compositions.
 \end{enumerate}
\end{lemma}
\begin{Proof}
 $i)$ By the Hopf--Lax formula \cite[$\S$3.3]{Eva10},
 the Hamilton--Jacobi equation~\eqref{HJ} admits a unique weak
 solution implicitly given by 
\begin{equation*}
  \widehat{G}(s,x) = \min_{y\in\R^d} \left\{ -s L\left( \frac{y-x}{s} \right) + G_0(y) \right\} = \min_{y\in\R^d} \left\{ - \frac{|y-x|^2}{4s} + H(\langle y \rangle) \right\}, \quad s< 0. 
\end{equation*}
For $x$ in a neighborhood $U$ of $\mathbf{0}$ in $\real^d$,
the minimizer $y=y(s,x)$
of the above problem, 
 if the minimum is achievable, should solve 
 \begin{equation}
   \nonumber %
  \frac{y_i-x_i}{H'(\langle y \rangle)} = 2 s, \quad i=1,\ldots,d, 
\end{equation}
 i.e., after summation over $i=1,\ldots,d$, 
 \begin{equation}
   \nonumber %
  \frac{\langle y \rangle-\langle x \rangle}{H'(\langle y \rangle)} = 2  s d, 
\end{equation}
 from which we infer 
\begin{equation}\label{eqn-20-2}
  y_i - x_i = \frac{\langle y \rangle - \langle x \rangle}{d}, \quad i=1,\ldots,d.
\end{equation}
 As $H'\ne 0$, it follows from the Lagrange--B\"urmann inversion
 formula 
 that $\langle y \rangle = \langle y(s,x) \rangle$ is analytical
 in $s$ in an $x$-dependent interval $I_x$
 containing $0$, and admits the series expansion 
\begin{equation*}
  \langle y(s,x) \rangle = \langle x \rangle
  + \sum_{k=1}^{\infty} (2sd)^k y_k(x) ,
\end{equation*}
where
\begin{equation*}
  y_k(x) := \frac{1}{k!} \lim _{z \rightarrow \langle x \rangle} \frac{\partial^{k-1}}{\partial z^{k-1}} ( H'(z) )^k.
\end{equation*}
This yields the following candidate solution to \eqref{HJ}-\eqref{HJ-0}:
\begin{equation}
  \nonumber %
  \begin{split}
    \widehat{G}(s,x) &:= - \frac{|y(s,x)-x|^2}{4s} + H(\langle y(s,x) \rangle)
    \\
    & = - s d \left| H'(\langle y(s,x) \rangle) \right|^2 + H(\langle y(s,x) \rangle)
    \\
    &= -\frac{\langle y(s,x) \rangle-\langle x \rangle}{2}
    H'(\langle y(s,x) \rangle ) + H(\langle y(s,x) \rangle).
  \end{split}
\end{equation}
The domain of definition
$$
\{ (s,x) \in \R \times \R^d \ : \ s \in I_x, \ x\in U \}
$$
of $\widehat{G}$ is the same as that of $\langle y(\cdot ,\cdot ) \rangle$.
Then, by \eqref{eqn-20-2}, we have
\begin{align} 
\nonumber
    \frac{\partial \widehat{G}}{\pt x_i} (s,x) &=
    - \frac{1}{2s}
    \sum_{j=1}^d (y_j(s,x)-x_j)(\pt_{x_i} y_j(s,x)- \delta_{ij})
    + H'(\langle y(s,x) \rangle) \sum_{j=1}^d \pt_{x_i} y_j(s,x) \\
\nonumber
    &= - H'(\langle y(s,x) \rangle) \sum_{j=1}^d (\pt_{x_i} y_j(s,x)- \delta_{ij}) + H'(\langle y(s,x) \rangle) \sum_{j=1}^d \pt_{x_i} y_j(s,x) \\
\nonumber
    &= H'(\langle y(s,x) \rangle)
    \\
\nonumber
    &= \frac{\langle y(s,x) \rangle - \langle x \rangle}{2sd}
    \\
\nonumber
&= \frac{y_i(s,x) - x_i}{2s}
\\
    \label{eqn-21}
    &= \sum_{k=1}^\infty (2sd)^{k-1} y_k(x) ,
\end{align} 
where \eqref{eqn-21} is known as the Lax--Oleinik formula
\cite[\S~3.4.2]{Eva10},
and
\begin{align*}
      \frac{\partial \widehat{G}}{\pt s} (s,x) & = - \frac{1}{2s} \sum_{i=1}^d (y_i(s,x)-x_i) \pt_s y_i(s,x)
 + \frac{|y(s,x)-x|^2}{4s^2} + H'(\langle y(s,x) \rangle) \sum_{j=1}^d \pt_s y_j(s,x) \\
    &= \frac{|y(s,x)-x|^2}{4s^2}.
\end{align*}
 These prove that the bivariate generating function $\widehat{G}$
 in \eqref{bi-gf} solves the Hamilton--Jacobi equation~\eqref{HJ}.

 \smallskip

 \noindent
  $ii)$
  It follows from \eqref{eqn-21} that for any $\alpha\in\N_0^d$ with $\alpha_i\ne 0$ for some $i=1,\ldots,d$, we have 
  \begin{equation}
    \nonumber %
  \frac{\pt^\alpha \widehat{G}}{\pt x^\alpha} (s,x)
  = \sum_{l=1}^\infty  (2sd)^{l-1} \pt^{\alpha-\ind_i} y_l(x), 
\end{equation}
 hence 
\begin{align} 
  \nonumber %
  \widehat{A}_k(\alpha) &=
 \frac{1}{k!\alpha!}
  \frac{\pt^k \pt^\alpha \widehat{G}}{\pt s^k\pt x^\alpha} (0,0)
 \\
  \nonumber
  &
 =\frac{(2d)^k}{\alpha!} \pt^{\alpha-\ind_i} y_{k+1}(0)
  \\
  \nonumber
  &
  = \frac{(2d)^k}{\alpha!(k+1)!}
  \frac{\partial^{k+|\alpha|-1}}{\partial z^{k+|\alpha|-1}}\left[ H'(z) \right]^{k+1}_{\mid z = 0} 
  \\
  \nonumber
  &= (2d)^k \frac{(k+|\alpha|-1)!}{\alpha!(k+1)!} \sum_{
      m_1+\cdots m_{k+1} = k+|\alpha|-1 \atop 
      m_1 , \ldots , m_{k+1} \geq 0
  } \prod_{j=1}^{k+1}
  \big(
  (1 + m_j ) h(1 + m_j )
  \big),
  \quad   k\geq 0,
\end{align} 
 which yields \eqref{A-hat-0}. 
\end{Proof}
\subsection{Factorial and exponential growth for terminal data}
\label{fjklfd9211} 
\begin{lemma}[Factorial growth]
\label{fjlkdf324}
 Let $r,\theta>0$ and 
\begin{equation*}
 g (\alpha) =
 g_{\theta , r} (\alpha) :=
  \frac{\theta^{|\alpha|}}{\alpha!}
  \prod_{l=0}^{|\alpha|-1} ( r + l)
  , %
  \quad \alpha\in\N_0^d, 
\end{equation*}
 and let 
$\big( \widehat{A}_k ( \alpha ) \big)_{(\alpha , k ) \in \N_0^d \times \N_0}$
 be the sequence defined in \eqref{eqn-13-2}-\eqref{eqn-13-2-0}. 
 \begin{enumerate}[i)]
\item
  The sequence
$\big( \widehat{A}_k ( \alpha ) \big)_{(\alpha , k ) \in \N_0^d \times \N_0}$
  satisfies \eqref{dom-asmp}
  provided that \eqref{*} holds. %
  \item
  For any $\alpha \in \N_0^d$, the generating function 
\begin{equation*}
 \widehat{G}_\alpha (s) := \sum_{k=0}^\infty \widehat{A}_k(\alpha) s^k
\end{equation*}
 has convergence radius %
 \begin{equation}
    \nonumber %
    R_{\theta,r} := \frac{(r+1)^{r+1}}{2\theta^2 r (r+2)^{r+2}d}. %
\end{equation}
\item 
 We have the bounds  
\begin{equation}
\nonumber %
   \widehat{G}_{\mathbf{0}}(|s|) < \frac{1}{2}\left( \frac{r+2}{r+1} \right)^{r+1},
   \quad |s| < R_{\theta,r}, 
\end{equation}
  and 
  \begin{equation}
\nonumber %
    \widehat{G}_\alpha(|s|) \lesssim_{\theta,r,d}
    \frac{(2\theta d)^{|\alpha|}}{2^{-(r+2)}R_{\theta,r}-|s|},
    \quad
    |s| < \frac{R_{\theta,r}}{2^{r+2}},
    \quad
  |\alpha|\geq 1.
 \end{equation}
\end{enumerate}
\end{lemma}
\begin{proof}
\noindent
 $i)$
 We have
\begin{equation*}
  G_0(x) =
  \sum_{\alpha\geq \mathbf{0}}
  \frac{\theta^{|\alpha|}}{\alpha!} x^\alpha
  \prod_{l=0}^{|\alpha|-1} ( r + l)
  = \frac{1}{(1-\theta\langle x \rangle)^{r}} := H(\langle x \rangle), 
\end{equation*}
 and it follows from \eqref{A-hat-0} that  
 \begin{align}
   \nonumber 
    \widehat{A}_k(\alpha) &= \frac{(2d)^k}{\alpha!(k+1)!} \lim _{z \rightarrow 0} \frac{d^{k+|\alpha|-1}}{d z^{k+|\alpha|-1}}\left[ H'(z) \right]^{k+1}
    \\
    \nonumber
    &
    = (2d)^k \frac{(\theta r)^{k+1}}{\alpha!(k+1)!} \lim _{z \rightarrow 0} \frac{d^{k+|\alpha|-1}}{d z^{k+|\alpha|-1}} (1-\theta z)^{-(r+1)(k+1)}
    \\
    \nonumber
    &=
    (2d)^k \frac{(\theta r)^{k+1}}{\alpha!(k+1)!}
    (
    (r+1)(k+1)
  \theta )^{k+|\alpha|-1}
\\
\nonumber %
    &
    = (2d)^k r^{k+1} \frac{\theta^{2k+|\alpha|}}{\alpha!} \frac{\Gamma((r+2)k+r+|\alpha|)}{(k+1)!\Gamma((r+1)(k+1))},
    \quad
  |\alpha |\geq 1, \ k\geq 0, 
\end{align}
 hence
 \begin{equation*}
  (1 + \alpha_i )
  \frac{\widehat{A}_k ( \alpha+\ind_i )}{\widehat{A}_k ( \alpha )}
   = ( (r+2)k+r+|\alpha| ) \theta \geq r \theta,
\end{equation*}
 and $\big( \widehat{A}_k ( \alpha ) \big)_{(\alpha , k ) \in \N_0^d \times \N_0}$
 satisfies \eqref{dom-asmp}.

 \smallskip

 \noindent $ii)$
 Next, given that $\Gamma(x+r) / \Gamma(x) = O(x^r)$ as
 $x$ tends to infinity, the convergence radius of
 $\widehat{G}_\alpha$ is %
\begin{equation}
  \nonumber %
  \begin{split}
    R_{\theta,r} &:= \limsup_{k\to\infty} \left| \frac{\widehat{A}_{k-1}(\alpha )}{\widehat{A}_k(\alpha )} \right|
    \\
    &
    = \frac{1}{2\theta^2 r d}
    \limsup_{k\to\infty} \frac{\Gamma((r+2)k-2+|\alpha|)}{\Gamma((r+1)k)} \frac{(k+1)\Gamma((r+1)(k+1))}{\Gamma((r+2)k+r+|\alpha|)}
    \\
    &= \frac{(r+1)^{r+1}}{2\theta^2 r (r+2)^{r+2}d}. %
  \end{split}
\end{equation}
\noindent
  $iii)$
 By \eqref{bi-gf}, we have 
\begin{align*}
  \widehat{G}_{\mathbf{0}}(s) & = \widehat{G}(s,0)
  \\
   & = -\frac{\langle y(s,0) \rangle}{2} H'(\langle y(s,0) \rangle )
  + H(\langle y(s,0) \rangle)
  \\
   & = \frac{2-\theta (r+2) \langle y(s,0) \rangle}{2 (1-\theta \langle y(s,0) \rangle)^{r+1}}.
\end{align*}
Letting $s\in (0 , R_{\theta,r} )$, since $\langle y(s,0) \rangle$ is the real solution of $(1-\theta y)^{r+1} y = 2 \theta r sd$ that is closer to $0$, one can deduce that
$$
0<\langle y(s,0) \rangle< \frac{1}{\theta(r+2)}.
$$ %
This implies
$$
1< \widehat{G}_0(s) < \frac{1}{2}\left( \frac{r+2}{r+1} \right)^{r+1}.
$$ 
 On the other hand, from Stirling's formula $\Gamma(z) \sim \sqrt{2 \pi z} ({z} / {e} )^z$ we have 
\begin{equation*}
  y_{k+1}(0) = \frac{\Gamma((r+2)k+r+1)}{(k+1)!\Gamma((r+1)(k+1))} \theta^{2k+1} r^{k+1} 
  \lesssim_{\theta,r} \frac{1}{(2d R_{\theta,r})^{k+1}},
  \qquad
 k\geq 0.
\end{equation*}
 Hence, by the inequalities
\begin{equation*}
  \binom{s}{k} \leq (\lfloor s \rfloor+1) \binom{\lfloor s \rfloor}{k} < \frac{(s+1) 2^{s+1}}{\sqrt{\pi \max ( 1 , s-1 ) }},
  \quad k \in \N_0, \ s\geq k, %
\end{equation*}
and 
\begin{equation}\label{eqn-29}
  \binom{|\alpha|}{\alpha} < \frac{\sqrt 2 e^dd^{|\alpha|+{d} / {2}}}{(2\pi |\alpha|)^{(d-1)/2}}
  ,
  \quad \alpha \in\N_0^d,\   |\alpha|\geq 1,  
\end{equation}
 we get 
\begin{equation*}
  \begin{split}
    \widehat{A}_k (\alpha ) &= (2d)^k \theta^{|\alpha|-1} \frac{(|\alpha|-1)!}{\alpha!} \binom{(r+2)k+r+|\alpha|-1}{|\alpha|-1} y_{k+1}(0) \\
    &\lesssim_{\theta,r} (2d)^k \theta^{|\alpha|-1} \frac{(|\alpha|-1)!}{\alpha!}
    \frac{2^{(r+2)k+r+|\alpha|-1} }{(2d R_{\theta,r})^{k+1}} \\
    &\lesssim_{\theta,r,d} \left( \frac{2^{r+2}}{R_{\theta,r}} \right)^{k+1} \binom{|\alpha|}{\alpha} (2\theta)^{|\alpha|}
    \\
    &
    \lesssim_{\theta,r,d} \left( \frac{2^{r+2}}{R_{\theta,r}} \right)^{k+1} (2\theta d)^{|\alpha|},
    \quad
  |\alpha | \geq 1, \ k\geq 0,
  \end{split}
\end{equation*}
 which yields 
\begin{equation*}
  \widehat{G}_\alpha (|s|) = \sum_{k=0}^\infty \widehat{A}_k ( \alpha ) |s|^k \lesssim_{\theta,r,d} \frac{(2\theta d)^{|\alpha|}}{2^{-(r+2)} R_{\theta,r}-|s|}, \qquad
  |\alpha|\geq 1.
\end{equation*}
\end{proof}
\begin{lemma}[Exponential growth]
\label{fjlkdf324-2}
 Let $\theta>0$ %
 and
\begin{equation*}
 g ( \alpha )
 = g_\theta ( \alpha )
 := \frac{\theta^{|\alpha|}}{\alpha!}, \quad \alpha\in\N_0^d, 
\end{equation*}
 and let 
 $\big( \widehat{A}_k ( \alpha ) \big)_{(\alpha , k ) \in \N_0^d \times \N_0}$
 be the sequence defined in \eqref{eqn-13-2}-\eqref{eqn-13-2-0}. 
\begin{enumerate}[i)]
\item
 The sequence
 $\big( \widehat{A}_k ( \alpha ) \big)_{(\alpha , k ) \in \N_0^d \times \N_0}$
 satisfies \eqref{dom-asmp}
 provided that \eqref{cond-theta-2-0} holds. %
 \item 
 For any $\alpha \in \N_0^d$
 the generating function 
\begin{equation*}
  \widehat{G}_\alpha (s) := \sum_{k=0}^\infty
   \widehat{A}_k ( \alpha ) s^k 
\end{equation*}
 has convergence radius 
\begin{equation*}
 R_\theta := \frac{1}{2e\theta^2 d}.
\end{equation*}
\item 
  For %
  $|s| < R_\theta$ %
        we have the bounds 
\begin{equation}
  \label{hklf434}
    \widehat{G}_{\mathbf{0}}(s) < \frac{e}{2}
    \quad
    \mbox{and} \quad 
    \widehat{G}_\alpha (|s|) \lesssim_{\theta,d} \left( \frac{d}{\log R_\theta - \log |s| } \right)^{|\alpha|-1}, \quad |\alpha| \geq 1. 
\end{equation}
\end{enumerate}
\end{lemma}
\begin{Proof}
 \noindent
 $i)$ 
 We have
\begin{equation*}
  G_0(x) = \sum_{\alpha \in\N_0^d} \frac{\theta^{|\alpha|}}{\alpha!} x^\alpha = e^{\theta (x_1+ \cdots +x_d)} = e^{\theta \langle x \rangle}.
\end{equation*}
It follows from \eqref{A-hat-0} that for $|\alpha |\geq 1$ and $k\geq 0$,
\begin{align}
 \nonumber 
  \widehat{A}_k ( \alpha ) & = \frac{(2d)^k}{\alpha!} \pt^{\alpha-\ind_i} y_{k+1}(0)
  \\
  \nonumber
  & = \frac{(2d)^k \theta^{k+1}}{\alpha!(k+1)!}
  \frac{\partial^{k+|\alpha|-1}}{\partial z^{k+|\alpha|-1}} e^{(k+1) \theta z}_{\mid z = 0} 
  \\
  \label{A-hat-formula-2}
  & = (2d)^k \frac{\theta^{2k+|\alpha|}}{\alpha!k!} (k+1)^{k+|\alpha|-2}, \quad k\geq 0,  
\end{align} 
 hence
\begin{equation*}
  (1 + \alpha_i )
  \frac{\widehat{A}_k(\alpha+\ind_i)}{\widehat{A}_k(\alpha)} = (k+1) \theta \geq \theta
\end{equation*}
 and 
 $\big( \widehat{A}_k (\alpha) \big)_{( \alpha , k ) \in\N_0^d \times \N_0}$
 satisfies \eqref{dom-asmp}.

 \smallskip

 \noindent
  $ii)$
 The convergence radius of $\widehat{G}_\alpha$ is
\begin{equation*}
 R_\theta := \limsup_{k\to\infty} \left| \frac{\widehat{A}_{k-1} ( \alpha )}{\widehat{A}_k ( \alpha )} \right|
 = \frac{1}{2\theta^2 d} 
 \limsup_{k\to\infty} \left( \frac{k}{k+1} \right)^{k+|\alpha|-2} = \frac{1}{2e\theta^2 d}.
\end{equation*}
 $iii)$
 By \eqref{bi-gf}, we have 
\begin{equation*}
  \widehat{G}_{\mathbf{0}} (s) = \widehat{G}(s,0) =
  \left(
  1 - \frac{\theta}{2} \langle y(s,0) 
  \right) 
   e^{\theta \langle y(s,0) \rangle}.
\end{equation*}
Let $s \in ( 0 , {1} / ( {2e\theta^2 d} ) )$. Since $\langle y \rangle = \langle y(s,0) \rangle$ is the
real solution of $\langle y \rangle e^{-\theta \langle y \rangle} = 2\theta sd$
 closest to $0$ we have
 $0<\langle y(s,0) \rangle< {1} / {\theta}$,
 hence 
\begin{equation*}
  1< \widehat{G}_{\mathbf{0}}(s) < \frac{e}{2}.
\end{equation*}
By \eqref{A-hat-formula-2} and Stirling's formula, we have 
\begin{equation*}
  \widehat{A}_k(\alpha ) < \frac{(2d)^k e^{k+1}}{\sqrt{2 \pi} \alpha!} \theta^{2k+|\alpha|} (k+1)^{|\alpha|-2},
\end{equation*}
  thus
\begin{equation*}
  \widehat{G}_\alpha (|s|) = \sum_{k=0}^\infty \widehat{A}_k(\alpha) |s|^k < \frac{\theta^{|\alpha|-2}}{2d\sqrt{2 \pi} \alpha!|s|} \text{Li}_{(2-|\alpha|)} \left( \frac{|s|}{R_\theta} \right),
  \quad
  |\alpha | \geq 1,
\end{equation*}
where $\text{Li}$ is the polylogarithm function, which has the asymptotic behavior 
\begin{equation*}
  \text{Li}_{(2-|\alpha|)} \left( \frac{|s|}{R_\theta} \right) \sim (|\alpha|-2)! \left( \log \frac{R_\theta}{|s|} \right)^{1-|\alpha|}, \quad |\alpha|\to\infty.
\end{equation*}
 Using again \eqref{eqn-29} yields \eqref{hklf434}.
\end{Proof}
\subsection{Integrability of random functionals} 
\noindent 
We close this section with
the proofs of Propositions~\ref{integ-final-1}
and \ref{integ-final-2}, and 
 let
 $\widetilde{\sigma}_{\rm b} $, 
 $\widetilde{\sigma}_\circ$
 be defined in \eqref{fkl323} and \eqref{aaa1}-\eqref{aaa2}. 
\begin{Proofy} \hskip-0.14cm {\em of Proposition~\ref{integ-final-1}.}
 We note that Conditions~\eqref{*} and
 \eqref{cond-theta-2-0} are used
 for the application of Proposition~\ref{fjkldf324}
 as noted in Lemmas~\ref{fjlkdf324} and \ref{fjlkdf324-2}
 respectively, while
 Conditions~\eqref{*-2} and
 \eqref{cond-theta-2-0-2} are used
 for the application of Proposition~\ref{stoch-dom}. 
  By \eqref{eqn-28},
 Proposition~\ref{stoch-dom},
 Corollary~\ref{fjkldf324-2}
 and Lemma~\ref{fjlkdf324}, 
 for both $c = {\alpha!}^{-1} \pt^\alpha$ and $c = {\alpha!}^{-1} \pt^\alpha \circ f^{(j)}$, $j\geq 0$, 
 we have 
\begin{align*}
  \E \left[ \left| \mathcal{H}_{t,T}\big(\mathcal{X}_t^x(c)\big) \right| \right]  & \leq \E \big[ N^{\sigma_\circ , \sigma_{\rm b}}_{t, T}(c) \big]
  \\
   & 
\leq \E \big[ \widetilde{N}^{\widetilde{\sigma}_\circ , \widetilde{\sigma}_{\rm b}}_{t ,T} ( \alpha,\max ( j , 0)) \big], 
  \\
   & \leq \frac{
    e^{-\lambda(T-t)} }{\bar\rho(T) }
    \widehat{G}_\alpha
    \left( \frac{1- e^{-\lambda(T-t)} }{\rho_*(T)\bar\rho(T) } \right), 
\\ 
 & \lesssim_{\theta,r,d} \frac{(2\theta d)^{|\alpha|} e^{-\lambda(T-t)} /\bar{\rho}(T) }{2^{-(r+2)} R_{\theta,r}
    - ( 1- e^{-\lambda(T-t)} )  / (\bar{\rho}(T) \rho_*(T))
  }
\\
  &
  \lesssim_{\theta,r,d,\lambda,T} (2\theta d)^{|\alpha|}.
\end{align*}
\end{Proofy}
\noindent
\begin{Proofy} \hskip-0.14cm {\em of Proposition~\ref{integ-final-2}.}
  By \eqref{eqn-28},
  Proposition~\ref{stoch-dom}, 
 Corollary~\ref{fjkldf324-2}
 and Lemma~\ref{fjlkdf324-2}, 
 for both
 $c = {\alpha!}^{-1} \pt^\alpha$ and $c = {\alpha!}^{-1} \pt^\alpha \circ f^{(j)}$,
 $j\geq 0$,
 we have 
\begin{align*}
  \E \left[ \left| \mathcal{H}_{t,T}\big(\mathcal{X}_t^x(c)\big) \right| \right]
  & \leq \E \big[ N^{\sigma_\circ , \sigma_{\rm b}}_{t, T}(c) \big]
  \\
   & \leq \E \big[
    \widetilde{N}^{\widetilde{\sigma}_\circ, \widetilde{\sigma}_{\rm b}}_{t ,T} ( \alpha,
    \max ( j , 0 ) ) \big]
  \\
   & \leq \frac{
    e^{-\lambda(T-t)} }{\bar\rho(T) }
    \widehat{G}_\alpha
    \left( \frac{1- e^{-\lambda(T-t)} }{\rho_*(T)\bar\rho(T) } \right), 
\\ 
&
 \lesssim_{\theta,d,\lambda,T}
 \left( \frac{d}{\log
       \big( R_\theta \rho_*(T)\bar\rho(T) /
       (1- e^{-\lambda T} ) \big)} \right)^{|\alpha|-1}, \quad (t,x) \in [0,T]\times \R^d. 
\end{align*}
\end{Proofy}
\section{Branching representation of PDE systems} 
\label{s11}
\noindent
The goal of Proposition~\ref{unf-intg-2} below is
to provide sufficient conditions for the 
consistency condition $u_f = f(u_{\id})$ to be
satisfied in Theorem~\ref{main-visc}, %
which allows us to remove the existence assumption
for the solution of the PDE~\eqref{nl-heat} in Theorem~\ref{thm-1}. 
 To this end, we build a branching process representation for the
 solutions of the PDE system~\eqref{pde-system}.
\begin{definition}
 We let $\widehat{\C} (f)$ denote the code set 
\begin{align*}
 & \widehat{\C} (f) := \C_f
    \cup \widehat{\C}_\partial
    \cup \widehat{\C}_f
    \\
    &:= \left\{ \frac{1}{\alpha!} \pt^\alpha \circ f^{(j)} 
    \right\}_{( \alpha , j ) \in \N_0^d \times \N_0}
    \bigcup \left\{ \frac{\alpha'!}{(\alpha' + \alpha'')!} \pt^{\alpha''} u^{}_{\frac{\pt^{\alpha'}}{\alpha'!} } \right\}_{\alpha', \alpha''\in\N_0^d
      \atop \alpha'+ \alpha''> \mathbf{0}}
    \bigcup \left\{ \frac{\alpha'!}{(\alpha' + \alpha'')!} \pt^{\alpha''} u^{}_{\frac{\pt^{\alpha'} \circ f^{(j)}}{\alpha'!} } \right\}_{\alpha', \alpha''\in\N_0^d
      \atop j \in\N_0}.
\end{align*}
\end{definition}
\noindent 
  We also define a mechanism on the
  new code set $\widehat{\C}(f)$.
\begin{definition}
 Let 
  $$\widehat{\M}:\widehat{\C} (f) \to
 ( \R \times \widehat{\C} (f) )
\bigcup
( \R \times \widehat{\C} (f) \times \widehat{\C} (f) )
$$
 be the mechanism defined as follows. 
\begin{enumerate}[i)]
 \item
   For $\hat{c}\in \C_f$, let 
    \begin{align*}
   & \widehat{\M} \left( \frac{1}{\alpha!} \pt^\alpha \circ f^{(j)} \right) := \left\{ \left( 1,\ \frac{1}{(\alpha-\beta)!} \pt^{\alpha-\beta} u_f,\ \frac{1}{\beta!} \pt^\beta \circ f^{(j+1)} \right), \right. \\
  & \left.
      \left( -\frac{1}{2} ( 1 + \beta_i )( 1 + \alpha_i-\beta_i ),\ \frac{1}{(\alpha-\beta+\ind_i)!} \pt^{\alpha-\beta+\ind_i} u_{\id} ,\ \frac{1}{(\beta+\ind_i)!} \pt^{\beta+\ind_i} \circ f^{(j+1)} \right)\right\}_{\mathbf{0}\leq \beta\leq \alpha
        \atop i=1,\ldots,d}.
    \end{align*}
  \item
    For $\hat{c}\in \widehat{\C}_\partial$, let 
    \begin{align*}
  &   \widehat{\M} \left( \frac{\alpha'!}{(\alpha' + \alpha'')!} \pt^{\alpha''} u^{}_{\frac{\pt^{\alpha'}}{\alpha'!} } \right) := \left\{ \left( 1,\ \frac{\alpha'!}{(\alpha' + \alpha'')!} \pt^{\alpha''} u^{}_{\frac{\pt^{\alpha'} \circ f}{\alpha'!} } \right) \right\}. 
 \end{align*}
  \item
For $\hat{c}\in \widehat{\C}_f$, let 
    \begin{align*}
 &  \widehat{\M} \left( \frac{\alpha'!}{(\alpha' + \alpha'')!} \pt^{\alpha''} u^{}_{\frac{\pt^{\alpha'} \circ f^{(j)}}{\alpha'!} } \right) := \left\{ \left( \frac{\binom{\alpha-\beta}{\alpha'-\beta'} \binom{\beta}{\beta'}}{\binom{\alpha}{\alpha'}},\ \frac{(\alpha'-\beta')!}{(\alpha-\beta)!} \pt^{\alpha''-\beta''} u^{}_{\frac{\pt^{\alpha'-\beta'} \circ f}{(\alpha'-\beta')!} },\ \frac{\beta'!}{\beta!} \pt^{\beta''} u_{\frac{\pt^{\beta'} \circ f^{(j+1)}}{\beta'!} } \right), \right. \\
    & \left( -\frac{( 1 + \beta_i )( 1 + \alpha_i-\beta_i ) \binom{\alpha-\beta}{\alpha'-\beta'} \binom{\beta}{\beta'}}{2\binom{\alpha}{\alpha'}} ,\ \frac{(\alpha'-\beta'+\ind_i)!}{(\alpha-\beta+\ind_i)!} \pt^{\alpha''-\beta''} u^{}_{\frac{\pt^{\alpha'-\beta'+\ind_i}}{(\alpha'-\beta'+\ind_i)!} } ,\ \right. \\
      & \qquad
      \quad
      \qquad \qquad \qquad \qquad \qquad \qquad \qquad \qquad \qquad 
   \frac{(\beta'+\ind_i)!}{(\beta+\ind_i)!} \pt^{\beta''} u^{}_{\frac{\pt^{\beta'+\ind_i} \circ f^{(j+1)}}{(\beta'+\ind_i)!} } \Bigg)\Bigg\}_{{
    \alpha = \alpha'+\alpha'', \beta = \beta'+\beta'', \atop \mathbf{0}\leq \beta'\leq \alpha', \mathbf{0}\leq \beta''\leq \alpha''} \atop i=1,\ldots,d}.
\end{align*}
\end{enumerate}
\end{definition}
 \noindent 
 We let $\iota : \widehat{\C}(f) \to \C(f)$
 denote the natural projection defined by 
 $$
\begin{cases}
  \displaystyle
  \iota \left( \frac{1}{\alpha!} \pt^\alpha \circ f^{(j)} \right)
  : = \frac{1}{\alpha!} \pt^\alpha \circ f^{(j)}, 
 & 
  \medskip
  \\
  \displaystyle
 \iota \left( \frac{\alpha'!}{(\alpha' + \alpha'')!} \pt^{\alpha''} u^{}_{\frac{\pt^{\alpha'}}{\alpha'!} } \right) := \frac{1}{(\alpha' + \alpha'')!} \pt^{\alpha' + \alpha''},
 & 
  \medskip
    \\
  \displaystyle
    \iota \left( \frac{\alpha'!}{(\alpha' + \alpha'')!} \pt^{\alpha''} u^{}_{\frac{\pt^{\alpha'} \circ f^{(j)}}{\alpha'!} } \right) := \frac{1}{(\alpha' + \alpha'')!} \pt^{\alpha' + \alpha''} \circ f^{(j)},
  & 
  \end{cases}
$$ 
and note that the restriction of $\iota$ to $\C_f$ is the identity
map.

\medskip

The proof of the following lemma
proceeds similarly to that of
Lemma~\ref{classical-sol}.
\begin{lemma}
\label{classical-sol-2}
 Assume that the PDE system~\eqref{pde-system} admits
 a classical solution $\{u_c\}_{c\in\C(f)} \subset {\cal C}^{1,\infty}_b([0,T]\times \R^d)$.
 Then, the family
 $\big( \widehat{u}_{\hat{c}}\big)_{\hat{c} \in \widehat{\C}(f)}
 $
 defined as
 $$
 \widehat{u}_{\hat{c}}
 :=
  \begin{cases}
    \hat{c}(u_\id), &\hat{c}\in \C_f, \\
    \hat{c}, &\hat{c}\in \widehat{\C}_f \cup \widehat{\C}_\partial , 
  \end{cases}
$$ 
 is a classical solution of the PDE system 
 \begin{equation}
   \label{pde-system-consist}
   \left(\frac{\partial}{\partial t} + \frac{1}{2} \Delta
   \right) \widehat{u}_{\hat{c}} + \sum_{\widehat{z} \in \widehat{\mathcal{M}}(\hat{c})} \widehat{z}_1 \prod_{i=2}^{|\widehat{z}|} \widehat{u}_{\widehat{z}_i} =0, \quad \widehat{u}_{\hat{c}}(T) = \iota(\hat{c})(\phi ). 
\end{equation}
\end{lemma}
\begin{Proof}
 For any test function $F\in {\cal C}^\infty(\R )$   
 we apply the operator $\displaystyle
 \frac{\partial}{\partial t} +\frac{1}{2} \Delta$
 to $F(u_{\id})$, and use 
\begin{equation*}
 \frac{\partial u_{\id}}{\partial t}  + \frac{1}{2} \Delta u_{\id} + u_f =0
\end{equation*}
 with $u_{\id}\in {\cal C}_b^{1,\infty}$, to get 
\begin{align*}
  \left( \frac{\partial}{\partial t} +\frac{1}{2} \Delta \right) F(u_\id) & = F' (u_\id) \left(
  \frac{\partial u_{\id}}{\partial t}
  +\frac{1}{2} \Delta u_\id \right) + \frac{1}{2} \sum_{i=1}^d \pt_i[F' (u_\id)] \pt_i u_\id
  \\
   & = - F' (u_\id) u_f + \frac{1}{2} \sum_{i=1}^d \pt_i[F' (u_\id)] \pt_i u_\id. 
\end{align*}
 Also, $\pt^\alpha [F(u_{\id})] \in {\cal C}_b^{1,\infty}$ and
\begin{align*} 
  \left( \frac{\partial}{\partial t} +\frac{1}{2} \Delta \right) \pt^\alpha [F (u_\id)] & = - \sum_{\mathbf{0}\leq \beta\leq \alpha} \binom{\alpha}{\beta} \pt^\beta [F' (u_\id)] \pt^{\alpha-\beta} u_f
  \\
   & \quad + \frac{1}{2} \sum_{i=1}^d \sum_{\mathbf{0}\leq \beta\leq \alpha} \binom{\alpha}{\beta} \pt^{\beta+\ind_i} [F' (u_\id)] \pt^{\alpha-\beta+\ind_i} u_{\id}.
  \end{align*}
   Moreover, we have
\begin{equation*}
  \left( \frac{\partial}{\partial t} +\frac{1}{2} \Delta \right) \pt^{\alpha''} u^{}_{\frac{\pt^{\alpha'}}{\alpha'!} } = \pt^{\alpha''} \left( \frac{\partial}{\partial t} +\frac{1}{2} \Delta \right) u^{}_{\frac{\pt^{\alpha'}}{\alpha'!} } = -\pt^{\alpha''} u^{}_{\frac{\pt^{\alpha'} \circ f}{\alpha'!} }.
\end{equation*}
and
\begin{align*}
  & \left( \frac{\partial}{\partial t} +\frac{1}{2} \Delta \right) \pt^{\alpha''} u^{}_{\frac{\pt^{\alpha'} \circ f^{(j)}}{\alpha'!} }
   = - \sum_{\mathbf{0}\leq \beta''\leq \alpha''} \sum_{\mathbf{0}\leq \beta'\leq \alpha'} \binom{\alpha''}{\beta''} \pt^{\beta''} u^{}_{\frac{\pt^{\beta'} \circ f^{(j+1)}}{\beta'!} } \pt^{\alpha''-\beta''} u^{}_{\frac{\pt^{\alpha'-\beta'} \circ f}{(\alpha'-\beta')!} } \\
   & \quad +
   \frac{1}{2}
   \sum_{i=1}^d \sum_{\mathbf{0}\leq \beta''\leq \alpha''} \sum_{\mathbf{0}\leq \beta'\leq \alpha'} ( 1 + \beta'_i )( 1 + \alpha'_i-\beta'_i ) \binom{\alpha''}{\beta''} \pt^{\beta''} u^{}_{\frac{\pt^{\beta'+\ind_i} \circ f^{(j+1)}}{(\beta'+\ind_i)!} } \pt^{\alpha''-\beta''} u^{}_{\frac{\pt^{\alpha'-\beta'+\ind_i}}{(\alpha'-\beta'+\ind_i)!} }.
\end{align*}
\end{Proof}
\noindent
In Proposition~\ref{unf-intg-2}
we obtain the probabilistic representation
of the derivatives of the solution 
of the PDE system~\eqref{pde-system} under 
 uniform integrability assumptions on random functionals. 
\begin{proposition}
   \label{unf-intg-2}
  Assume that
  \begin{itemize} %
  \item
    the PDE system~\eqref{pde-system-consist} admits a classical solution
    $\big( \widehat{u}_{\hat{c}}\big)_{\hat{c} \in \widehat{\C}(f)}
    \subset {\cal C}^{1,\infty}_b([0,T]\times \R^d)$,
  and %
  that
\item
  for each $(t,x)\in[0,T]\times\R^d$ and $c\in \C(f)$ the sequence 
  of random variables
  \begin{equation*}
    \widehat{\mathcal{H}}_{t,T,n}^{x,\hat{c}} := \prod_{{\mathbf{k}} \in \cup_{i=0}^n \mathcal{K}^{{\rm b},i}_{t,T}(c)} \frac{c^{{\mathbf{k}}} (\phi ) \big(X_T^{{\mathbf{k}}}\big)}{\bar\rho(T-T_{{\mathbf{k}}^{-}})} \prod_{{\mathbf{k}} \in \cup_{i=0}^n \mathcal{K}^{\circ,i}_{t,T}(c)} \frac{I_1(c^{{\mathbf{k}}})}{\rho (\tau_{\mathbf{k}}) q_{c^{{\mathbf{k}}}}(I(c^{{\mathbf{k}}}))} \prod_{{\mathbf{k}} \in \mathcal{K}^{n+1}_{t,T}(c)} \widehat{u}_{\hat{c}^{{\mathbf{k}}}}
    \big(T_{{\mathbf{k}}^{-}}, X_{T_{{\mathbf{k}}^{-}}}^{{\mathbf{k}}}\big) 
  \end{equation*}
  is uniformly integrable,
  where
  \begin{equation*}
    \hat{c} :=
    \begin{cases}
      c, & c\in \C_f,
      \medskip
      \\
      \displaystyle
      \frac{1}{\alpha!} \pt^\alpha u_\id, &
      \displaystyle
      c = \frac{1}{\alpha!} \pt^\alpha \in \C_\partial.
    \end{cases}
  \end{equation*}
  \end{itemize} %
   Then, we have
$$
  c(u_\id(t,x)) = u_c (t,x)
  = \E \left[ \mathcal{H}_{t,T}(
  \mathcal{X}_t^x(c)) \right],
  \quad
  (t,x)\in[0,T]\times\R^d,
  \ c\in\C(f).
  $$ 
 In particular, taking $c := f$ yields
 the consistency condition %
 in Theorem~\ref{main-visc}, %
 i.e. 
$$
 f(u_\id(t,x)) =
 u_f (t,x) =
 \E \left[ \mathcal{H}_{t,T}(
    \mathcal{X}_t^x(f)) \right],
  \quad
  (t,x)\in[0,T]\times\R^d,
  $$ 
  and uniqueness of mild solutions holds for the PDE
  system~\eqref{pde-system-consist}. 
\end{proposition}
\begin{proof}
\noindent
$i)$
 We extend $\iota : \widehat{\C}(f) \to \C(f)$
 into a projection
 from $(\widehat{\C}(f), \widehat{\M})$ onto $(\C(f), \M)$,
 as follows. 
\begin{enumerate}[\textbullet]
\item
  For $\hat{c}\in \C_f$ %
  and
$\widehat{z} = (\widehat{z}_1, \widehat{z}_2, \widehat{z}_3)
\in \widehat{\M}(\hat{c})$, we let 
\begin{equation*}
  \iota(\widehat{z}) := (\widehat{z}_1, \iota(\widehat{z}_2),
  \widehat{z}_3 ).
\end{equation*}
  \item
For $\hat{c}\in \widehat{\C}_\partial $ and
$\widehat{z} = (\widehat{z}_1, \widehat{z}_2)
\in \widehat{\M}(\hat{c})$, we let 
\begin{equation*}
  \iota(\widehat{z}) := (\widehat{z}_1, \iota(\widehat{z}_2)).
\end{equation*}
\item
  For $\hat{c}\in \widehat{\C}_f$, and $\widehat{z}
= (\widehat{z}_1, \widehat{z}_2, \widehat{z}_3)
  \in \widehat{\M}(\hat{c})$, we let 
  \begin{align*}
    &
    \hskip-0.4cm
    \iota(\widehat{z})
  := ( \widehat{z}_1 , \iota(\widehat{z}_2), \iota ( \widehat{z}_3 ) )
  \\
  &
      \hskip-0.45cm
 = \begin{cases}
      \displaystyle
      \left( 1,\ \frac{1}{(\alpha-\beta)!} \pt^{\alpha-\beta} \circ f,\ \frac{1}{\beta!} \pt^\beta \circ f^{(j+1)} \right), &
      \displaystyle
      \widehat{z}_3 = \frac{\beta'!}{\beta!} \pt^{\beta''} u_{\frac{\pt^{\beta'} \circ f^{(j+1)}}{\beta'!}}
      \medskip
      \\
    \displaystyle
    \left( -\frac{1}{2} ( 1 + \beta_i )( 1 + \alpha_i-\beta_i ),\
    \frac{\pt^{\alpha-\beta+\ind_i}}{(\alpha-\beta+\ind_i)!} ,\
    \frac{\pt^{\beta+\ind_i} \circ f^{(j+1)} }{(\beta+\ind_i)!} \right), &
    \displaystyle
    \widehat{z}_3 = \frac{(\beta'+\ind_i)!}{(\beta+\ind_i)!} \pt^{\beta''} u^{}_{\frac{\pt^{\beta'+\ind_i} \circ f^{(j+1)}}{(\beta'+\ind_i)!} }.
\end{cases}
\end{align*}
\end{enumerate}
We note that
\begin{enumerate}[\textbullet]
\item
  when $\hat{c}\in \C_f \cup \widehat{\C}_\partial $, $z\in \M(\iota(\hat{c}))$
  uniquely determines a $\widehat{z}\in \widehat{\M}(\hat{c})$ such that $z = \iota(\widehat{z})$;
\item
  when $\hat{c} =
{\alpha!}^{-1}{\alpha'!} \pt^{\alpha''} u^{}_{{\alpha'!}^{-1} {\pt^{\alpha'} \circ f^{(j)}} } \in \widehat{\C}_f$ with $\alpha = \alpha' + \alpha''$, $\M(\iota(\hat{c})) = \M\big({\alpha!}^{-1} \pt^\alpha \circ f^{(j)}\big)$, we have
\begin{align*}
  & \iota^{-1}\left( 1,\ \frac{1}{(\alpha-\beta)!} \pt^{\alpha-\beta} \circ f,\ \frac{1}{\beta!} \pt^\beta \circ f^{(j+1)} \right)
  \\
  &
  \quad = \left\{ \left( \frac{\binom{\alpha-\beta}{\alpha'-\beta'} \binom{\beta}{\beta'}}{\binom{\alpha}{\alpha'}},\ \frac{(\alpha'-\beta')!}{(\alpha-\beta)!} \pt^{\alpha''-\beta''} u^{}_{\frac{\pt^{\alpha'-\beta'} \circ f}{(\alpha'-\beta')!} }, \right. \right.
  \frac{\beta'!}{\beta!} \pt^{\beta''} u_{\frac{\pt^{\beta'} \circ f^{(j+1)}}{\beta'!} } \Bigg) \Bigg\}
  _{\beta = \beta'+\beta''
    \atop {\mathbf{0}\leq \beta'\leq \alpha' \atop
      \mathbf{0}\leq \beta''\leq \alpha''}},
\end{align*}
 and
\begin{align*}
  & \iota^{-1}\left( -\frac{1}{2} ( 1 + \beta_i )( 1 + \alpha_i-\beta_i ),\ \frac{1}{(\alpha-\beta+\ind_i)!} \pt^{\alpha-\beta+\ind_i},\ \frac{1}{(\beta+\ind_i)!} \pt^{\beta+\ind_i} \circ f^{(j+1)} \right) \\
  & = \left\{ \left( -\frac{( 1 + \beta_i )( 1 + \alpha_i-\beta_i) \binom{\alpha-\beta}{\alpha'-\beta'} \binom{\beta}{\beta'}}{2\binom{\alpha}{\alpha'}} ,\ \frac{(\alpha'-\beta'+\ind_i)!}{(\alpha-\beta+\ind_i)!} \pt^{\alpha''-\beta''} u^{}_{\frac{\pt^{\alpha'-\beta'+\ind_i}}{(\alpha'-\beta'+\ind_i)!} } ,
\right. \right. 
  \\
  &
  \left.
  \left.
  \qquad
  \qquad
  \qquad\qquad\qquad\qquad\qquad\qquad\qquad
  \frac{(\beta'+\ind_i)!}{(\beta+\ind_i)!} \pt^{\beta''} u^{}_{\frac{\pt^{\beta'+\ind_i} \circ f^{(j+1)}}{(\beta'+\ind_i)!} } \right)\right\}_{
    {
      \alpha = \alpha'+\alpha'', \beta = \beta'+\beta''
      \atop \mathbf{0}\leq \beta'\leq \alpha'
      \mathbf{0}\leq \beta''\leq \alpha''}
        \atop i=1,\ldots,d}. 
  \end{align*}
\end{enumerate}
\noindent
$ii)$
 Next, we recode the original branching tree $\mathcal{X}_t^x(c)$
using the coding mechanism $(\widehat{\C}(f), \widehat{\M})$,
as follows. Let
 $\hat{c}\in \widehat{\C}(f)$ with $c = \iota(\hat{c})$ be the initial code. %
\begin{itemize}
\item If $\hat{c} \in \C_f \cup \widehat{\C}_\partial $, the initial
   branch has $|I(c)|-1$ offsprings coded originally by $I_{k+1}(c)$, $k=1,\ldots,|I(c)|-1$, where $I(c)$ is a random variable in $\M(c)$ with probability $q_c$. As $I(c)$ can uniquely determine a random variable $\hat{I}(\hat{c})$ in $\widehat{\M}(\hat{c})$ with $I(c) = \iota(\hat{I}(\hat{c}))$, we now recode the offsprings by $\hat{I}_{k+1}(\hat{c})$, $k=1,\ldots,|I(c)|-1$. Thus, the new offspring distribution, denoted by $\widehat{q}_{\hat{c}}$, is given by $\widehat{q}_{\hat{c}} (\widehat{z}) = q_{\iota(\hat{c})} (\iota(\widehat{z}))$.
 \item If $\hat{c} = \frac{\alpha'!}{\alpha!} \pt^{\alpha''} u^{}_{\frac{1}{\alpha'!} \pt^{\alpha'} \circ f^{(j)}} \in \widehat{\C}_f$ with $\alpha = \alpha' + \alpha''$, since $c = \iota(\hat{c}) = \frac{1}{\alpha!} \pt^\alpha \circ f^{(j)} \in \C_f$, the initial
    branch has two offsprings coded originally by $I_{k+1}(c)$, $k=1,2$, where $I(c)$ is a random variable in $\M(\frac{1}{\alpha!} \pt^\alpha \circ f^{(j)})$ with probability $q_c$. We sample a new random variable $\hat{I}(\hat{c})$ in $\iota^{-1}(z)$ conditionally to $\{I(c)=z\}$, $z \in \M(\frac{1}{\alpha!} \pt^\alpha \circ f^{(j)})$, by putting
  \begin{equation*}
    \P\big( \hat{I}(\hat{c}) = \widehat{z} \mid I(c) = z \big) = \frac{\binom{\alpha-\beta}{\alpha'-\beta'} \binom{\beta}{\beta'}}{\binom{\alpha}{\alpha'}} = \frac{\binom{\alpha''}{\beta''} \binom{\alpha‘}{\beta’}}{\binom{\alpha}{\beta}},
  \end{equation*}
  for $z = \iota(\widehat{z})$ with
  $$
  z_1 = \frac{\binom{\alpha-\beta}{\alpha'-\beta'} \binom{\beta}{\beta'}}{\binom{\alpha}{\alpha'}}
  \quad
  \mbox{or}
  \quad
  z_1 = -\frac{( 1 + \beta_i )( 1 + \alpha_i-\beta_i) \binom{\alpha-\beta}{\alpha'-\beta'} \binom{\beta}{\beta'}}{2\binom{\alpha}{\alpha'}},
  $$
   where $\beta = \beta' + \beta''$. This is a well-defined conditioning probability since
   $$\sum_{
     \beta' + \beta'' = \beta \atop
            { \mathbf{0}\leq \beta'\leq \alpha'
              \atop
              \mathbf{0}\leq \beta''\leq \alpha'' }}
   \binom{\alpha''}{\beta''} \binom{\alpha'}{\beta'} = \binom{\alpha}{\beta}.$$
  Thus, the new offspring distribution, denoted by $\widehat{q}_{\hat{c}}$, is given by
\begin{equation*}
  \widehat{q}_{\hat{c}}(\widehat{z}) = \frac{\binom{\alpha''}{\beta''} \binom{\alpha'}{\beta'}}{\binom{\alpha}{\beta}} q_c(z).
\end{equation*}
  \item Repeat the above steps recursively.
\end{itemize}
The above recoding procedure yields a new coded branching tree
$\widehat{\mathcal{X}}_t^x (\hat{c})$
with coding mechanism $(\widehat{\C}(f), \widehat{\M})$ and new offspring distributions $(\widehat{q}_{\hat{c}}: \hat{c}\in \widehat{\C}(f))$,
which shares the branching structure of the original branching tree $\mathcal{X}_t^x(c)$, but with different codes. Denoting the new code borne by the branch ${\mathbf{k}}$ by $\hat{c}^{\mathbf{k}}$,
 the relation of the original and new codes is given by
\begin{equation}\label{branching-old-new}
  c^{\mathbf{k}} = \iota\big(\hat{c}^{\mathbf{k}}\big), 
\end{equation}
 and by construction we have 
\begin{equation}\label{eqn-18}
  \frac{\widehat{z}_1}{\widehat{q}_{\hat{c}}(\widehat{z})} = \frac{\iota(\widehat{z})_1}{q_{\iota(\hat{c})}(\iota(\widehat{z}))}.
\end{equation}
\noindent
$iii)$
 Finally, we apply the arguments of the proof of
 Proposition~\ref{unf-intg} by replacing the use of
 Lemma~\ref{classical-sol} by that of Lemma~\ref{classical-sol-2},
 under the uniform integrability of the sequence 
\begin{equation*}
\widehat{\mathcal{H}}_{t,T,n}^{x,\hat{c}} = \prod_{{\mathbf{k}} \in \cup_{i=0}^n \mathcal{K}^{{\rm b},i}_{t,T}(c)} \frac{\iota \big(\hat{c}^{{\mathbf{k}}}\big)(\phi ) \big(X_T^{{\mathbf{k}}}\big)}{\bar\rho(T-T_{{\mathbf{k}}^{-}})} \prod_{{\mathbf{k}} \in \cup_{i=0}^n \mathcal{K}^{\circ,i}_{t,T}(c)} \frac{\hat{I}_1(\hat{c}^{\mathbf{k}})}{\rho (\tau_{\mathbf{k}}) \widehat{q}_{\hat{c}^{{\mathbf{k}}}}\big(\hat{I}(\hat{c}^{\mathbf{k}})\big)} \prod_{{\mathbf{k}} \in \mathcal{K}^{n+1}_{t,T}(c)} \widehat{u}_{\hat{c}^{{\mathbf{k}}}} \big(T_{{\mathbf{k}}^{-}}, X_{T_{{\mathbf{k}}^{-}}}^{{\mathbf{k}}}\big),
\end{equation*}
where the above equality follows from
 \eqref{branching-old-new}-\eqref{eqn-18}.
\end{proof}

\appendix

\section{Appendix}
\noindent
This appendix contains auxiliary lemmas which admit
independent proofs, as well as remarks of independent
interest.
\subsubsection*{Auxiliary lemmas}
\noindent 
 In Lemma~\ref{lA3} we check that 
 \eqref{offsp-dist}-\eqref{offsp-dist-2} yield a probability distribution. 
\begin{lemma}
\label{lA3}
For all $c\in \C (f)$, the family $\{q_c(z)\}_{z\in \M(c)}$
defined in \eqref{offsp-dist}-\eqref{offsp-dist-2} satisfies 
$$
 \sum_{z\in \M(c)} q_c(z) = 1. 
$$
\end{lemma}
\begin{Proof}
  It is sufficient to consider the case
  $c = {\alpha!}^{-1} \pt^\alpha \circ f^{(j)}$.
  We have 
$$ 
\sum_{\mathbf{0}\leq \beta\leq \alpha} 1 = \sum_{\beta_1 = 0}^{\alpha_1} \cdots \sum_{\beta_d = 0}^{\alpha_d} 1 = \prod_{k=1}^d ( 1+ \alpha_k )
$$
and
\begin{align*}
\sum_{\mathbf{0}\leq \beta \leq \alpha} ( 1 + \beta_i)( 1+ \alpha_i-\beta_i )
& =
\sum_{l = 0}^{\alpha_i} ( 1+ l )( 1+ \alpha_i-l )
\prod_{k=1\atop k\ne i}^d ( 1+ \alpha_k ) 
\\
 & = \frac{(2+ \alpha_i )(3+ \alpha_i )}{6} \prod_{k=1}^d ( 1+ \alpha_k ),
\end{align*}
 where the latter follows from the generating function calculation 
\begin{align*}
  \sum_{m=0}^\infty w^m \sum_{k=0}^m (k+1) (m-k+1) 
  & = \left( \sum_{k=0}^\infty (k+1) w^k \right)^2
  \\
  &
  = \frac{1}{(1-w)^4}
  \\
   & = \sum_{k \geq 0} \binom{k+3}{k} w^k.
\end{align*}
\end{Proof}
\begin{lemma}
  \label{polylog-stirling}
  The polylogarithm function $\mathrm{Li}_s (z)$ satisfies 
  \begin{equation}
    \mathrm{Li}_{-m} \left(\frac{1}{2} \right)
    = 2 \sum_{k=1}^m k! S(m, k),
    \quad m \geq 1.
    \end{equation} 
\end{lemma}
\begin{Proof}
 From \cite[Eq.~(6.4)]{Woo92}, we have 
\begin{equation*}
  \mathrm{Li}_{-m} \left(\frac{1}{2} \right) = \sum_{k=0}^m k! S(m+1, k+1).
\end{equation*}
 Recall that the Stirling numbers of the second kind satisfy
\begin{equation*}
  S(m,1) = S(m,m) = 1, \quad m\geq 1, 
\end{equation*}
and the recursion
\begin{equation*}
  S(m+1,k) = kS(m,k) + S(m,k-1), \quad m\geq k\geq 2.
\end{equation*}
Thus, we have
  \begin{equation*}
    \begin{split}
      \mathrm{Li}_{-m} \left(\frac{1}{2} \right) &= S(m+1, 1) + m! S(m+1, m+1) + \sum_{k=1}^{m-1} \left[ (k+1)! S(m, k+1) + k! S(m, k) \right] \\
      &= 1 + m! + \sum_{k=2}^m k! S(m, k) + \sum_{k=1}^{m-1} k! S(m, k) \\
      &= 2 \sum_{k=1}^m k! S(m, k).
    \end{split}
  \end{equation*}
\end{Proof} 
\begin{lemma} %
\label{jklfds19a} 
The conditions
\begin{equation} 
\label{weight-asmpt-red-1}
 g(\alpha-\beta) g(\beta) \frac{\sigma_\circ^{(0)}(\alpha , j)}{g(\alpha)}
 \geq
 \frac{1}{\bar{\rho} (T)} %
\end{equation}
and
\begin{equation}  
\label{weight-asmpt-red-2}
  g(\alpha-\beta+\ind_i)
  g(\beta+\ind_i)
  \frac{\sigma_\circ^{(i)}(\alpha , j)}{g(\alpha)}
  \geq
 \frac{1}{\bar{\rho} (T)} 
, 
\end{equation} 
$\mathbf{0}\leq \beta\leq \alpha$, $\alpha\in\N_0^d$,
 $i=1,\ldots , d$, $j\geq 0$, are satisfied by
\begin{enumerate}[i)]
\item
 the function 
\begin{equation*}
  g(\alpha) = g_{\theta , r} (\alpha) :=
  \frac{\theta^{|\alpha|}}{\alpha!}
  \prod_{l=0}^{|\alpha|-1} ( r + l)
  , %
  \quad \alpha\in\N_0^d, 
\end{equation*}
in \eqref{gthetar},
provided that $r,\theta>0$ satisfy \eqref{*-2},
\item
 the function %
\begin{equation*}
  g(\alpha) = g_\theta (\alpha)
  := \frac{\theta^{|\alpha|}}{\alpha!}, \quad \alpha\in\N_0^d, 
\end{equation*}
in \eqref{gtheta},
provided that $\theta > 0$ satisfies \eqref{cond-theta-2-0-2}. 
\end{enumerate}
\end{lemma}
\begin{Proof}
$i)$
  When $\alpha \in \N_0^d$ is fixed, the functions 
$$
  \beta \mapsto g(\alpha-\beta) \frac{g(\beta)}{g(\alpha)}
  \quad
  \mbox{and}
  \quad
  \beta \mapsto g(\alpha-\beta+\ind_i) \frac{g(\beta+\ind_i)}{g(\alpha)}
$$
  achieve their respective
  minima $1$ and $\theta^2 {(|\alpha|+r)r} / {(1 + \alpha_i )}$
  at $\beta=\mathbf{0}$ and $\beta=\alpha$.   
  Thus, for any $\mathbf{0}\leq \beta\leq \alpha$ we have 
$$ 
  g(\alpha-\beta) \frac{g(\beta)}{g(\alpha)}
  \prod_{k=1}^d (1 + \alpha_k )
   \geq \prod_{k=1}^d (1 + \alpha_k )
   \geq 1, 
$$ 
 which shows \eqref{weight-asmpt-red-1} from \eqref{s2}. 
 We also have 
 \begin{align*} 
   &
   g(\alpha-\beta+\ind_i) \frac{g(\beta+\ind_i)}{g(\alpha)}
   \frac{(2 + \alpha_i )(3 + \alpha_i )}{12} 
   \prod_{k=1}^d (1 + \alpha_k ) \\
   & \quad \geq \frac{\theta^2}{12} r (|\alpha|+r) (2 + \alpha_i )(3+ \alpha_i ) \prod_{k=1\atop
   k \ne i
   }^d (1 + \alpha_k )
   \\
   & \quad
   \geq \frac{\theta^2}{2} r^2
      \\
   & \quad
    \geq \frac{1}{d+1}, 
\end{align*}
 which shows \eqref{weight-asmpt-red-2} from \eqref{s3}. 

 \smallskip
 
 \noindent
 $ii)$
 For any $\mathbf{0}\leq \beta\leq \alpha$, we have 
$$ 
g(\alpha-\beta) \frac{g(\beta)}{g(\alpha)}
\prod_{k=1}^d (1 + \alpha_k )
  = 
\binom{\alpha}{\beta}
\prod_{k=1}^d (1 + \alpha_k )
 \geq 1, 
$$ 
 which shows \eqref{weight-asmpt-red-1} from \eqref{s2}, and 
\begin{align*}
  & 
  g(\alpha-\beta+\ind_i) \frac{g(\beta+\ind_i)}{g(\alpha)}
    \frac{(2+ \alpha_i)(3+ \alpha_i)}{12} 
    \prod_{k=1}^d (1 + \alpha_k )
  \\
  & \quad = \frac{\theta^2}{12} (3 + \alpha_i)
  \binom{\alpha+2 \ind_i}{\beta+ \ind_i}
  \prod_{k=1\atop k\ne i}^d (1 + \alpha_k )
  \\
  & \quad 
  \geq \frac{\theta^2}{2}
    \\
  & \quad 
  \geq \frac{1}{d+1}, 
\end{align*}
 which shows \eqref{weight-asmpt-red-2} from \eqref{s3}. 
\end{Proof}
\noindent 
 Lemma~\ref{progeny} is a classical result which is consequence of 
 e.g. \cite[Eq. (8) page 3]{kendall1948},
 \cite[Example 13.2 page 112]{harris1963}, 
 or \cite[Example 5 page 109]{athreya}. % ,
% see also Proposition~3.1 in \cite{huangprivault1}. 
\begin{lemma} 
\label{progeny}
 The total progeny
 of the branching chain 
 $\widetilde{\cal X}_t$ 
 introduced in Section~\ref{s5}
 has distribution 
 \begin{equation}
 \label{recursion-total-sol}
        \P \big(
       \widetilde{X}_{t,T}
       = n \big)
  = 
  \begin{cases}
    e^{-\lambda (T-t)} \big(1-e^{-\lambda (T-t)}\big)^m, & n=2m+1, %
    \\
    0, & \text{otherwise}, \ m \geq 0, 
  \end{cases}
\end{equation}
 $t\in [0,T]$. 
\end{lemma}
\subsubsection*{Combinatorial remarks}
\noindent
 Next, we make some comments on the relationship
 between the above analysis and classical combinatorics. 
\begin{remark}
$i)$ Taking 
\begin{equation*}
 g (\alpha) =
 g_{\theta , r} (\alpha) :=
  \frac{\theta^{|\alpha|}}{\alpha!}
  \prod_{l=0}^{|\alpha|-1} ( r + l)
  , %
  \quad \alpha\in\N_0^d, 
\end{equation*}
 for $r,\theta>0$, the sequence
 $\big\{ \widehat{A}_k ( \alpha ) \big\}_{(\alpha , k ) \in \N_0^d \times \N_0}$
 defined in \eqref{eqn-13-2}-\eqref{eqn-13-2-0} satisfies
 the relation 
\begin{align*}
  \widehat{A}_k(\alpha ) & = -(2d)^k \theta^{2k+|\alpha|} r^{k+1} F_k(r+2, r+1)
  \\
   & = (-1)^k (2d)^k \theta^{2k+|\alpha|} r^{k+1} F_k(-(r+1),-(r+1))
\end{align*}
 for $|\alpha|=1$ and $k\geq 0$, 
 where $F_k(p, r)$ is the Fuss--Catalan number \cite{Mlo10} defined by
\begin{equation*}
  F_k(p, r)=\frac{r}{k}\binom{k p+r-1}{k-1},
 \quad p,k,r>0.
\end{equation*}
$ii)$ 
  Taking
\begin{equation*}
 g ( \alpha )
 = g_\theta ( \alpha )
 := \frac{\theta^{|\alpha|}}{\alpha!}, \quad \alpha\in\N_0^d, 
\end{equation*}
 and plugging the formula \eqref{A-hat-formula-2} for
 $\big\{ \widehat{A}_k ( \alpha ) \big\}_{( \alpha , k ) \in\N_0^d \times \N_0 }$
 defined in \eqref{eqn-13-2}-\eqref{eqn-13-2-0}
 back into the recursion \eqref{eqn-13-2}-\eqref{eqn-13-2-0}, 
 yields the combinatorial identity 
\begin{equation*}
  \sum_{l=0}^k \binom{k}{l} (l+1)^{l-1} (k-l+1)^{k-l-1} = 2(k+2)^{k-1}, 
\end{equation*}
 which can be verified using generating functions, as follows:
\begin{align*}
 \sum_{l=0}^\infty x^l \sum_{k=0}^l \frac{(k+1)^{k-1} (l-k+1)^{l-k-1}}{k!(l-k)!}
 &= \left( \sum_{k=0}^\infty \frac{(k+1)^{k-1}}{k!} x^k \right)^2
  \\
  &
  = \left( \frac{W(-x)}{x} \right)^2
  \\
  &
  = 2 \sum_{k=0}^{\infty} (k+2)^{k-1} \frac{x^k}{k!},
\end{align*}
where $W$ is the Lambert function.
\end{remark}

\section*{Conclusion}
\noindent
In this paper, we have have derived sufficient growth conditions on coefficients that ensure the validity of an extension of the classical Feynman--Kac representation to a class of semilinear parabolic PDE systems. This approach allows us to represent the solution of the PDE system as the expectation of a functional of a coded branching diffusion process, providing a probabilistic interpretation and computational method for solving such PDEs. Numerical experiments are presented in dimension up to 1000. 

\paragraph{Acknowledgement.}
The work of Q. Huang is supported by
the National Natural Science Foundation of China under Grant No. 12501241,
the Basic Research Program of Jiangsu under Grant No. BK20251280,
the Zhishan Young Scholar Program of Southeast University, 
the Start-Up Research Fund of Southeast University under Grant No. RF1028624194, and
the Jiangsu Provincial Scientific Research Center of Applied Mathematics under Grant No. BK20233002.

\footnotesize

\newcommand{\etalchar}[1]{$^{#1}$}
\def\cprime{$'$} \def\polhk#1{\setbox0=\hbox{#1}{\ooalign{\hidewidth
  \lower1.5ex\hbox{`}\hidewidth\crcr\unhbox0}}}
  \def\polhk#1{\setbox0=\hbox{#1}{\ooalign{\hidewidth
  \lower1.5ex\hbox{`}\hidewidth\crcr\unhbox0}}} \def\cprime{$'$}

\end{document}

%% file: allen_cahn_3_summary.tex
\begin{tabular}{c c c c c c c }
\hline
Dim & Binary & NPP23 & BSDE & Binary (s) & NPP23 (s) & BSDE (s) \\
\hline
1 & 0.62572 & 0.61885 & 0.52700 & 93.81 & 108.87 & 106.57 \\
10 & 0.28985 & 0.29008 & 0.29181 & 107.40 & 129.91 & 108.86 \\
100 & 0.03925 & 0.03940 & 0.04014 & 64.67 & 35.86 & 126.96 \\
1000 & 0.00375 & 0.00374 & NaN & 267.12 & 45.62 & 304.90 \\
\hline
\end{tabular}

%% file: allen_cahn_3_dim_1000_time_behavior_t.tex
\begin{tabular}{c c c c c c c }
\hline
T & Binary & NPP23 & BSDE & Binary (s) & NPP23 (s) & BSDE (s) \\
\hline
0.10 & 0.01342 & 0.01344 & 0.01349 & 5.58 & 2.18 & 300.77 \\
0.20 & 0.00745 & 0.00743 & 0.01882 & 17.53 & 7.97 & 300.93 \\
0.30 & 0.00541 & 0.00538 & 0.00570 & 55.95 & 16.67 & 304.18 \\
0.40 & 0.00432 & 0.00436 & NaN & 168.68 & 28.37 & 300.22 \\
0.50 & 0.00375 & 0.00374 & NaN & 267.12 & 45.62 & 304.90 \\
\hline
\end{tabular}

%% file: allen_cahn_summary.tex
\begin{tabular}{c c c c c c c c }
\hline
Dim & Binary & NPP23 & BSDE & Exact & Binary (s) & NPP23 (s) & BSDE (s) \\
\hline
1 & -0.67528 & -0.67885 & -0.87043 & -0.67918 & 97.26 & 110.82 & 106.99 \\
10 & -0.67302 & -0.68571 & -0.68620 & -0.67918 & 71.57 & 86.63 & 109.46 \\
100 & -0.67829 & -0.68518 & -0.68754 & -0.67918 & 126.07 & 72.18 & 128.07 \\
1000 & -0.66048 & -0.67516 & NaN & -0.67918 & 347.05 & 64.68 & 300.03 \\
\hline
\end{tabular}

%% file: allen_cahn_dim_1000_time_behavior_t.tex
\begin{tabular}{c c c c c c c }
\hline
T & Binary & NPP23 & BSDE & Binary (s) & NPP23 (s) & BSDE (s) \\
\hline
0.10 & -0.53725 & -0.53706 & -0.53823 & 17.69 & 4.20 & 304.65 \\
0.20 & -0.56954 & -0.57329 & -0.57454 & 56.86 & 16.03 & 295.53 \\
0.30 & -0.60979 & -0.60810 & -0.61319 & 128.76 & 34.54 & 304.73 \\
0.40 & -0.64106 & -0.64427 & NaN & 264.94 & 59.86 & 294.66 \\
0.50 & -0.66048 & -0.67516 & NaN & 347.05 & 64.68 & 300.03 \\
\hline
\end{tabular}

%% file: allen_cahn_2_summary.tex
\begin{tabular}{c c c c c c c c }
\hline
Dim & Binary & NPP23 & BSDE & Exact & Binary (s) & NPP23 (s) & BSDE (s) \\
\hline
1 & 0.16359 & 0.16349 & NaN & 0.16347 & 88.36 & 102.63 & 106.98 \\
10 & 0.16393 & 0.16323 & 0.16443 & 0.16347 & 104.12 & 122.73 & 110.26 \\
100 & 0.16133 & 0.16094 & 0.16451 & 0.16347 & 130.05 & 67.40 & 127.26 \\
1000 & 0.14957 & 0.14999 & 0.23187 & 0.16347 & 251.72 & 53.47 & 300.05 \\
\hline
\end{tabular}

%% file: allen_cahn_2_dim_1000_time_behavior_t.tex
\begin{tabular}{c c c c c c c }
\hline
T & Binary & NPP23 & BSDE & Binary (s) & NPP23 (s) & BSDE (s) \\
\hline
0.10 & 0.10988 & 0.11003 & 0.11039 & 9.79 & 2.53 & 298.51 \\
0.20 & 0.11991 & 0.11994 & 0.12196 & 18.25 & 9.65 & 300.29 \\
0.30 & 0.12954 & 0.12966 & 0.83565 & 59.81 & 19.19 & 299.89 \\
0.40 & 0.13996 & 0.13972 & 0.14868 & 148.34 & 35.68 & 300.39 \\
0.50 & 0.14957 & 0.14999 & 0.23187 & 251.72 & 53.47 & 300.05 \\
\hline
\end{tabular}

%% file: exp_example_summary.tex
\begin{tabular}{c c c c c c c c }
\hline
Dim & Binary & NPP23 & BSDE & Exact & Binary (s) & NPP23 (s) & BSDE (s) \\
\hline
1 & 0.95904 & 0.96787 & 0.83476 & 0.96789 & 101.02 & 121.84 & 111.52 \\
10 & 0.95646 & 0.96843 & 0.97682 & 0.96789 & 77.98 & 98.00 & 116.08 \\
100 & 0.95820 & 0.97234 & 0.97939 & 0.96789 & 131.93 & 78.54 & 131.49 \\
1000 & 0.95285 & 0.97092 & NaN & 0.96789 & 366.26 & 65.17 & 303.74 \\
\hline
\end{tabular}

%% file: exp_example_dim_1000_time_behavior_t.tex
\begin{tabular}{c c c c c c c }
\hline
T & Binary & NPP23 & BSDE & Binary (s) & NPP23 (s) & BSDE (s) \\
\hline
0.10 & 0.84331 & 0.84392 & 0.84413 & 19.07 & 4.23 & 295.62 \\
0.20 & 0.87792 & 0.87732 & 0.87837 & 36.60 & 16.43 & 302.75 \\
0.30 & 0.90891 & 0.90119 & 0.90954 & 130.99 & 35.22 & 305.53 \\
0.40 & 0.93402 & 0.93325 & 0.94286 & 298.98 & 59.20 & 304.80 \\
0.50 & 0.95285 & 0.97092 & NaN & 366.26 & 65.17 & 303.74 \\
\hline
\end{tabular}

%% file: arxiv.bbl
\begin{thebibliography}{DPTW23b}

\bibitem[AC20]{claisse}
A.~Agarwal and J.~Claisse.
\newblock Branching diffusion representation of semi-linear elliptic {PDE}s and
  estimation using {M}onte {C}arlo method.
\newblock {\em Stochastic Processes and their Applications}, 130(8):5006--5036,
  2020.

\bibitem[AN72]{athreya}
K.B. Athreya and P.E. Ney.
\newblock {\em Branching processes}, volume 196 of {\em Die Grundlehren der
  mathematischen Wissenschaften}.
\newblock Springer-Verlag, New York-Heidelberg, 1972.

\bibitem[BBH{\etalchar{+}}20]{hutzenthaler-mlp4}
S.~Becker, R.~Braunwarth, M.~Hutzenthaler, A.~Jentzen, and Ph. von
  Wurstemberger.
\newblock Numerical simulations for full history recursive multilevel {P}icard
  approximations for systems of high-dimensional partial differential
  equations.
\newblock {\em Commun. Comput. Phys.}, 28(5):2109--2138, 2020.

\bibitem[BM10]{bakhtin}
Y.~Bakhtin and C.~Mueller.
\newblock Solutions of semilinear wave equation via stochastic cascades.
\newblock {\em Commun. Stoch. Anal.}, 4(3):425--431, 2010.

\bibitem[BRT07]{blomker}
D.~Bl{\"{o}}mker, M.~Romito, and R.~Tribe.
\newblock A probabilistic representation for the solutions to some non-linear
  {PDE}s using pruned branching trees.
\newblock {\em Ann. Inst. H. Poincar\'{e} Probab. Statist.}, 43(2):175--192,
  2007.

\bibitem[BS84]{brown-shubert}
G.G. Brown and B.O. Shubert.
\newblock On random binary trees.
\newblock {\em Math. Oper. Res.}, 9(1):43--65, 1984.

\bibitem[CLM08]{chakraborty}
S.~Chakraborty and J.A. L\'{o}pez-Mimbela.
\newblock Nonexplosion of a class of semilinear equations via branching
  particle representations.
\newblock {\em Advances in Appl. Probability}, 40:250--272, 2008.

\bibitem[CSTV07]{touzi}
P.~Cheridito, H.M. Soner, N.~Touzi, and N.~Victoir.
\newblock Second-order backward stochastic differential equations and fully
  nonlinear parabolic {PDE}s.
\newblock {\em Comm. Pure Appl. Math.}, 60(7):1081--1110, 2007.

\bibitem[DMT08]{dalang}
R.C. Dalang, C.~Mueller, and R.~Tribe.
\newblock A {F}eynman-{K}ac-type formula for the deterministic and stochastic
  wave equations and other {P}.{D}.{E}.'s.
\newblock {\em Trans. Amer. Math. Soc.}, 360(9):4681--4703, 2008.

\bibitem[DMTW19]{waymire}
R.~Dascaliuc, N.~Michalowski, E.~Thomann, and E.C. Waymire.
\newblock Complex {B}urgers equation: a probabilistic perspective.
\newblock In {\em Sojourns in probability theory and statistical physics. {I}.
  {S}pin glasses and statistical mechanics, a {F}estschrift for {C}harles {M}.
  {N}ewman}, volume 298 of {\em Springer Proc. Math. Stat.}, pages 138--170.
  Springer, Singapore, 2019.

\bibitem[DPTW23a]{waymire1}
R.~Dascaliuc, T.N. Pham, E.~Thomann, and E.C. Waymire.
\newblock Doubly stochastic {Y}ule cascades ({P}art {I}): {T}he explosion
  problem in the time-reversible case.
\newblock {\em J. Funct. Anal.}, 284(1):Paper No. 109722, 25, 2023.

\bibitem[DPTW23b]{waymire2}
R.~Dascaliuc, T.N. Pham, E.~Thomann, and E.C. Waymire.
\newblock Doubly stochastic {Y}ule cascades (part {II}): {T}he explosion
  problem in the non-reversible case.
\newblock {\em Ann. Inst. Henri Poincar\'{e} Probab. Stat.}, 59(4):1907--1933,
  2023.

\bibitem[EHJK19]{hutzenthaler-mlp0}
W.~E, M.~Hutzenthaler, A.~Jentzen, and T.~Kruse.
\newblock On multilevel {P}icard numerical approximations for high-dimensional
  nonlinear parabolic partial differential equations and high-dimensional
  nonlinear backward stochastic differential equations.
\newblock {\em Journal of Scientific Computing}, 79:1534--1571, 2019.

\bibitem[EHJK21]{hutzenthaler-mlp3}
W.~E, M.~Hutzenthaler, A.~Jentzen, and T.~Kruse.
\newblock Multilevel {P}icard iterations for solving smooth semilinear
  parabolic heat equations.
\newblock {\em Partial Differential Equations and Applications}, 2, 2021.

\bibitem[Eva10]{Eva10}
L.C. Evans.
\newblock {\em Partial differential equations}, volume~19.
\newblock American Mathematical Society, 2nd edition, 2010.

\bibitem[Har63]{harris1963}
T.E. Harris.
\newblock {\em The theory of branching processes}, volume 119 of {\em Die
  Grundlehren der mathematischen Wissenschaften}.
\newblock Springer-Verlag, Berlin; Prentice Hall, Inc., Englewood Cliffs, NJ,
  1963.

\bibitem[HHLZ26]{HanHuLongZhao2026}
J.~Han, W.~Hu, J.~Long, and Y.~Zhao.
\newblock Deep {P}icard iteration for high-dimensional nonlinear {PDE}s.
\newblock {\em SIAM Journal on Scientific Computing}, 48(1), 2026.

\bibitem[HJE17]{han2018solvingarxiv}
J.~Han, A.~Jentzen, and W.~E.
\newblock Deep learning-based numerical methods for high-dimensional parabolic
  partial differential equations and backward stochastic differential
  equations.
\newblock Preprint arXiv:1706.04702, 39 pages, 2017.

\bibitem[HJE18]{han2018solving}
J.~Han, A.~Jentzen, and W.~E.
\newblock Solving high-dimensional partial differential equations using deep
  learning.
\newblock {\em Proceedings of the National Academy of Sciences},
  115(34):8505--8510, 2018.

\bibitem[HJK22]{hutzenthaler-mlp1}
M.~Hutzenthaler, A.~Jentzen, and T.~Kruse.
\newblock Overcoming the curse of dimensionality in the numerical approximation
  of parabolic partial differential equations with gradient-dependent
  nonlinearities.
\newblock {\em Found. Comput. Math.}, 22:905--966, 2022.

\bibitem[HJKN20]{hutzenthaler-mlp2}
M.~Hutzenthaler, A.~Jentzen, T.~Kruse, and T.A. Nguyen.
\newblock Multilevel {P}icard approximations for high-dimensional semilinear
  second-order {PDE}s with {L}ipschitz nonlinearities.
\newblock Preprint arXiv:2009.02484v4, 2020.

\bibitem[HL12]{henry-labordere2012}
P.~Henry-Labord\`ere.
\newblock Counterparty risk valuation: a marked branching diffusion approach.
\newblock Preprint arXiv:1203.2369, 2012.

\bibitem[HLOT{\etalchar{+}}19]{labordere}
P.~Henry-Labord\`ere, N.~Oudjane, X.~Tan, N.~Touzi, and X.~Warin.
\newblock Branching diffusion representation of semilinear {PDE}s and {M}onte
  {C}arlo approximation.
\newblock {\em Ann. Inst. H. Poincar\'e Probab. Statist.}, 55(1):184--210,
  2019.

\bibitem[HLT21]{labordere2}
P.~Henry-Labord\`ere and N.~Touzi.
\newblock Branching diffusion representation for nonlinear {C}auchy problems
  and {M}onte {C}arlo approximation.
\newblock {\em Ann. Appl. Probab.}, 31(5):2350--2375, 2021.

\bibitem[HP26a]{huangprivault1}
Q.~Huang and N.~Privault.
\newblock Binary {G}alton--{W}atson trees with mutations.
\newblock {\em Comm. Nonlinear Sci. Numer. Simul.}, 156:109623, 2026.

\bibitem[HP26b]{huangprivault2}
Q.~Huang and N.~Privault.
\newblock Probabilistic representation of {ODE} solutions with quantitative
  estimates.
\newblock {\em J. Math. Anal. Appl.}, 555:130195, 2026.

\bibitem[HPW20]{hure2019some}
C.~Hur{\'e}, H.~Pham, and X.~Warin.
\newblock Deep backward schemes for high-dimensional nonlinear {PDE}s.
\newblock {\em Math. Comp.}, 89(324):1547--1579, 2020.

\bibitem[INW69]{inw}
N.~Ikeda, M.~Nagasawa, and S.~Watanabe.
\newblock Branching {M}arkov processes {I}, {II}, {III}.
\newblock {\em J. Math. Kyoto Univ.}, 8-9:233--278, 365--410, 95--160,
  1968-1969.

\bibitem[Ken48]{kendall1948}
D.G. Kendall.
\newblock On the generalized ``birth-and-death'' process.
\newblock {\em Ann. Math. Statistics}, 19:1--15, 1948.

\bibitem[KS91]{Karatzas-Shreve-1991}
I.~Karatzas and S.~Shreve.
\newblock {\em Brownian motion and stochastic calculus}, volume 113 of {\em
  Graduate Texts in Mathematics}.
\newblock Springer-Verlag, New York, second edition, 1991.

\bibitem[LM96]{lm}
J.A. L{\'o}pez-Mimbela.
\newblock A probabilistic approach to existence of global solutions of a system
  of nonlinear differential equations.
\newblock In {\em Fourth Symposium on Probability Theory and Stochastic
  Processes (Spanish) (Guanajuato, 1996)}, volume~12 of {\em Aportaciones Mat.
  Notas Investigaci\'on}, pages 147--155. Soc. Mat. Mexicana, M\'exico, 1996.

\bibitem[LS97]{sznitman}
Y.~{Le Jan} and A.~S. Sznitman.
\newblock Stochastic cascades and $3$-dimensional {N}avier-{S}tokes equations.
\newblock {\em Probab. Theory Related Fields}, 109(3):343--366, 1997.

\bibitem[McK75]{hpmckean}
H.P. McKean.
\newblock Application of {B}rownian motion to the equation of
  {K}olmogorov-{P}etrovskii-{P}iskunov.
\newblock {\em Comm. Pure Appl. Math.}, 28(3):323--331, 1975.

\bibitem[Mey23]{Mey23}
C.D. Meyer.
\newblock {\em Matrix analysis and applied linear algebra}.
\newblock SIAM, 2nd edition, 2023.

\bibitem[M{\l}o10]{Mlo10}
W.~M{\l}otkowski.
\newblock Fuss-{C}atalan numbers in noncommutative probability.
\newblock {\em Doc. Math.}, 15:939--955, 2010.

\bibitem[Ngu26]{Nguyen2026MLP}
T.A. Nguyen.
\newblock Multilevel {P}icard approximations overcome the curse of
  dimensionality when approximating semilinear heat equations with
  gradient-dependent nonlinearities in {$L^p$}-sense.
\newblock {\em Journal of Computational and Applied Mathematics}, 473:116834,
  2026.

\bibitem[NNW25]{NeufeldNguyenWu2025MLPGradient}
A.~Neufeld, T.A. Nguyen, and S.~Wu.
\newblock Multilevel {P}icard approximations overcome the curse of
  dimensionality in the numerical approximation of general semilinear {PDE}s
  with gradient-dependent nonlinearities.
\newblock {\em Journal of Complexity}, 90:101946, 2025.

\bibitem[NP23]{burgers}
J.Y. Nguwi and N.~Privault.
\newblock Numerical solution of the modified and non-{N}ewtonian {B}urgers
  equations by stochastic coded trees.
\newblock {\em Japan Journal of Industrial and Applied Mathematics},
  40:1745--1763, 2023.

\bibitem[NPP23a]{penent2022fully}
J.Y. Nguwi, G.~Penent, and N.~Privault.
\newblock A fully nonlinear {F}eynman-{K}ac formula with derivatives of
  arbitrary orders.
\newblock {\em Journal of Evolution Equations}, 23:Paper No. 22, 29pp., 2023.

\bibitem[NPP23b]{nguwipenentprivaultnv}
J.Y. Nguwi, G.~Penent, and N.~Privault.
\newblock Numerical solution of the incompressible {N}avier-{S}tokes equation
  by a deep branching algorithm.
\newblock {\em Communications on Computational Physics}, 34:261--289, 2023.

\bibitem[NPP24]{nguwipenentprivault}
J.Y. Nguwi, G.~Penent, and N.~Privault.
\newblock A deep branching solver for fully nonlinear partial differential
  equations.
\newblock {\em J. Comput. Phys.}, 499:112712, 2024.

\bibitem[NSW25]{Neufeld2025RandomDeepSplitting}
A.~Neufeld, P.~Schmocker, and S.~Wu.
\newblock Full error analysis of the random deep splitting method for nonlinear
  parabolic {PDE}s and {PIDE}s.
\newblock {\em Communications in Nonlinear Science and Numerical Simulation},
  143:108556, 2025.

\bibitem[Ott49]{Ott49}
R.~Otter.
\newblock The multiplicative process.
\newblock {\em The Annals of Mathematical Statistics}, 20(2):206--224, 1949.

\bibitem[Pen91]{peng2}
S.~Peng.
\newblock Probabilistic interpretation for systems of quasilinear parabolic
  partial differential equations.
\newblock {\em Stochastics Stochastics Rep.}, 37(1-2):61--74, 1991.

\bibitem[PP92]{pardouxpeng}
{\'E}.~Pardoux and S.~Peng.
\newblock Backward stochastic differential equations and quasilinear parabolic
  partial differential equations.
\newblock In {\em Stochastic partial differential equations and their
  applications ({C}harlotte, {NC}, 1991)}, volume 176 of {\em Lecture Notes in
  Control and Inform. Sci.}, pages 200--217. Springer, Berlin, 1992.

\bibitem[Ram06]{ramirez}
J.M. Ramirez.
\newblock Multiplicative cascades applied to {PDE}s (two numerical examples).
\newblock {\em J. Comput. Phys.}, 214(1):122--136, 2006.

\bibitem[Sev58]{sevastyanov}
B.A. Sevast{\cprime}yanov.
\newblock Branching stochastic processes for particles diffusing in a bounded
  domain with absorbing boundaries.
\newblock {\em Theory of Probability and its Applications}, {III}(2):111--126,
  1958.

\bibitem[Sko64]{skorohodbranching}
A.V. Skorokhod.
\newblock Branching diffusion processes.
\newblock {\em Teor. Verojatnost. i. Primenen.}, 9:492--497, 1964.

\bibitem[SS07]{shaked2007stochastic}
M.~Shaked and J.G. Shanthikumar.
\newblock {\em Stochastic orders}.
\newblock Springer Science+Business Media, LLC, 2007.

\bibitem[STZ12]{soner}
H.M. Soner, N.~Touzi, and J.~Zhang.
\newblock Wellposedness of second order backward {SDE}s.
\newblock {\em Probab. Theory Related Fields}, 153(1-2):149--190, 2012.

\bibitem[Woo92]{Woo92}
D.C. Wood.
\newblock The computation of polylogarithms.
\newblock {\em Technical report}, 1992.

\end{thebibliography}
